\documentclass[11pt]{article}
\usepackage{amsmath,amssymb,latexsym, times}
\usepackage{amsthm}
\usepackage{rsfso}

\usepackage{amscd}
\usepackage{booktabs} 
\usepackage{wrapfig}
\usepackage{tikz}
\usepackage{pgfplots}
\usepackage{float}
\pgfplotsset{compat=newest}
\usepackage{authblk}
\usepackage{mathrsfs}
\usepackage{bbm}
\usepackage{graphicx}
\usepackage{stmaryrd}
\usepackage[all]{xy}
\usepackage{enumitem}
\usepackage{tikz-cd}

\usepackage{tocloft}
\usepackage{hyperref}
\hypersetup{
    colorlinks=true,
    linkcolor=blue,
    filecolor=magenta,      
    urlcolor=cyan,
    citecolor=magenta
}

\usepackage{setspace}
\newtheorem{innercustomgeneric}{\customgenericname}
\providecommand{\customgenericname}{}
\newcommand{\newcustomtheorem}[2]{%
  \newenvironment{#1}[1]
  {%
   \renewcommand\customgenericname{#2}%
   \renewcommand\theinnercustomgeneric{##1}%
   \innercustomgeneric
  }
  {\endinnercustomgeneric}
}

\newcustomtheorem{customthm}{Theorem}
\newcustomtheorem{customconj}{Conjecture}  
\newcustomtheorem{customprop}{Proposition}  
\newcustomtheorem{customcor}{Corollary}

\theoremstyle{plain}
\newtheorem{thm}{Theorem}[section]
\newtheorem{lem}[thm]{Lemma}
\newtheorem{prop}[thm]{Proposition}
\newtheorem{cor}[thm]{Corollary}

\newtheorem{thm-def}[thm]{Theorem-Definition}

\theoremstyle{definition}
\newtheorem{defn}[thm]{Definition}
\newtheorem{cor-def}[thm]{Corollary-Definition}
\newtheorem{example}[thm]{Example}

\theoremstyle{remark}
\newtheorem{rem}[thm]{Remark}

\numberwithin{equation}{section}

\makeatletter
\renewcommand{\paragraph}{%
  \@startsection{paragraph}{4}%
  {\z@}{1.25ex \@plus 1ex \@minus .2ex}{-1em}%
  {\normalfont\normalsize\bfseries}%
}
\makeatother

\usepackage{adjustbox}
\newcommand{\BigWedge}{\mathord{\adjustbox{valign=B,totalheight=.6\baselineskip}{$\bigwedge$}}}

\newcommand{\gl}{\mathrm{GL}}
\newcommand{\ggl}{\mathfrak{gl}}

\newcommand{\Sp}{\mathrm{Sp}}

\newcommand{\bbg}{\mathbb{G}}

\newcommand{\bbv}{\mathbb{V}}

\newcommand{\bbf}{\mathbb{F}}

\newcommand{\bbd}{\mathbb{D}}

\newcommand{\cf}{\mathcal{F}}
\newcommand{\co}{\mathcal{O}}
\newcommand{\cp}{\mathcal{P}}
\newcommand{\cl}{\mathcal{L}}
\newcommand{\ck}{\mathcal{K}}

\newcommand{\ch}{\mathcal{H}}

\newcommand{\cd}{\mathcal{D}}
\newcommand{\ce}{\mathcal{E}}

\newcommand{\cs}{\mathcal{S}}

\newcommand{\crr}{\mathcal{R}}

\newcommand{\scp}{\mathscr{P}}

\newcommand{\xx}{\mathcal{X}}

\newcommand{\ka}{\mathfrak{a}}
\newcommand{\kg}{\mathfrak{g}}

\newcommand{\km}{\mathfrak{m}}

\newcommand{\kss}{\mathfrak{S}}

\newcommand{\bn}{\mathbf{N}}
\newcommand{\bz}{\mathbf{Z}}
\newcommand{\br}{\mathbf{R}}
\newcommand{\bp}{\mathbf{P}}
\newcommand{\bq}{\mathbf{Q}}

\newcommand{\bc}{\mathbf{C}}

\newcommand{\ba}{\mathbf{A}}

\newcommand{\rd}{\mathrm{d}}

\newcommand{\rff}{\mathrm{F}}

\newcommand{\rtt}{\mathrm{T}}

\newcommand{\rT}{\mathrm{T}}
\newcommand{\rS}{\mathrm{S}}

\newcommand{\val}{\mathrm{val}}

\newcommand{\Ad}{\mathrm{Ad}}

\newcommand{\codim}{\mathrm{Codim}}
\newcommand{\diag}{\mathrm{diag}}

\newcommand{\Id}{\mathrm{Id}}

\newcommand{\reg}{\mathrm{reg}}
\newcommand{\spec}{\mathrm{Spec}}

\newcommand{\pic}{\mathrm{Pic}}

\newcommand{\aff}{\mathrm{aff}}

\newcommand{\ic}{\mathrm{IC}}
\newcommand{\amp}{\mathrm{Amp}}
\newcommand{\uni}{\mathrm{uni}}

\newcommand{\spf}{\mathrm{Spf}}

\newcommand{\pr}{\mathrm{pr}}
\newcommand{\Hom}{\mathrm{Hom}}

\newcommand{\aut}{\mathrm{Aut}}

\newcommand{\ep}{\epsilon}
\newcommand{\varep}{\varepsilon}

\newcommand{\sm}{\mathrm{sm}} 
 
\newcommand{\red}{\mathrm{red}} 
\newcommand{\defm}{\mathrm{Def}}

\newcommand{\act}{\mathrm{act}}

\newcommand{\ql}{\mathbf{Q}_{\ell}}  
\newcommand{\zl}{\mathbf{Z}_{\ell}}

\newcommand{\fq}{\mathbf{F}_{q}}

\newcommand{\cm}{{\mathcal{M}}}
\newcommand{\ct}{\mathcal{T}}

\newcommand{\cc}{\mathcal{C}}

\def\cb{{\mathcal B}}

\def\ka{{\mathfrak{a}}}

\def\kg{{\mathfrak{g}}}

\def\km{{\mathfrak{m}}}

\def\kff{{\mathfrak{F}}}

\newcommand{\Tor}{\operatorname{Tor}}

\usepackage{xcolor}
\usepackage{version}

\excludeversion{NB}

\newcommand{\supp}{\operatorname{supp}}

\newcommand{\pair}[1]{\langle #1\rangle}

\newcommand{\rhom}{R\mathcal{H}om}

\newcommand{\bsm}{\begin{smallmatrix}}
\newcommand{\esm}{\end{smallmatrix}}

\def\cqfd{\hfill $\Box$}

\def\<{\langle\,}
\def\>{\,\rangle}

\def\kss{\mathfrak S}

\def\kg{\mathfrak g}

\def\<{\langle}
\def\>{\rangle}
\def\={\equiv}
\def\le{\leqslant}
\def\ge{\geqslant}

\def\mod{\;{\rm mod}\;}

\def\Hom{{\rm Hom}}

\def\fq{\mathbb{F}_{q}}

\def\rd{{\rm d}}
\def\val{{\rm val}}

\def\aut{{\rm Aut}}
\def\fr{{\rm Fr}}

\def\Id{{\rm Id}}

\def\Ker{\mathrm{Ker\,}}

\author{Zongbin \textsc{Chen}\thanks{School of mathematics, Shandong University, Shandong, P. R. China. Email: \texttt{zongbin.chen@email.sdu.edu.cn}}}

\title{\bf A decomposition theorem for the affine Springer fibers}

\begin{document}
\maketitle

\begin{abstract}

According to Laumon, an affine Springer fiber is homeomorphic to the universal abelian covering of the compactified Jacobian of a spectral curve.
We construct equivariant deformations $f_{n}:\overline{\mathcal{P}}_{n}\to \mathcal{B}_{n}$ of the finite abelian coverings of this compactified Jacobian, and decompose the complex $Rf_{n,*}\mathbf{Q}_{\ell}$ as direct sum of intersection complexes. 
Pass to the limit, we obtain a similar expression for the homology of the affine Springer fibers.  
A quite surprising consequence is that we can reduce the homology to its $\Lambda^{0}$-invariant subspace.
As an application, we get a sheaf-theoretic reformulation of the purity hypothesis of Goresky, Kottwitz and MacPherson. 
In an attempt to solve it, we propose a conjecture about the punctural weight of the intermediate extension of a smooth $\ell$-adic sheaf of pure weight.  

\end{abstract}

\tableofcontents

\section{Introduction}\label{main results}

Let $k$ be the field $\bc$ of complex numbers or a finite field $\fq$ of characteristic $p$, {$p$ sufficiently large}. We set $p=0$ if $k=\bc$. 
Let $F=k(\!(\varep)\!)$ be the field of Laurent series, $\co=k[\![\varep]\!]$ its ring of integers.
Let $G=\gl_{d}$, and $\kg$ its Lie algebra. 
Let $\gamma\in \kg[\![\varep]\!]$ be a regular semisimple element. Recall that the \emph{affine Springer fiber} at $\gamma$ is the closed sub ind-$k$-scheme of the \emph{affine grassmannian} $\xx=G(\!(\varep)\!)/G[\![\varep]\!]$ defined by
$$
\xx_{\gamma}=\big\{g\in \xx\,\big|\,\Ad(g^{-1})\gamma \in \kg[\![\varep]\!]\big\}.
$$ 
We have $\xx_{\gamma}\cong \xx_{\gamma}^{0}\times \bz$ with $\xx_{\gamma}^{0}$ being  the central connected component of $\xx_{\gamma}$, and $\xx_{\gamma}^{0}$ admits the action of a discrete free abelian group $\Lambda^{0}$. 
According to Laumon \cite{laumon springer}, the quotient $\Lambda^{0}\backslash \xx_{\gamma}^{0}$ is homeomorphic to the compactified Jacobian $\overline{P}_{C_{\gamma}}$ of a spectral curve $C_{\gamma}$ over $k$, and $\xx_{\gamma}^{0}$ is homeomorphic to a $\Lambda^{0}$-torsor $\overline{P}{}_{C_{\gamma}}^{\natural}$ over $\overline{P}_{C_{\gamma}}$, which can be seen as the universal abelian covering of $\overline{P}_{C_{\gamma}}$.

Let $\pi:(\cc, C_{\gamma})\to (\cb,0)$ be an algebraization of a miniversal deformation of $C_{\gamma}$. Note that $\cb$ is smooth at $0$, and we can even take it to be the affine space  $\ba^{\tau_{\gamma}}$ with $\tau_{\gamma}$ being the Tjurina number of the singularity $\spf(\co[\gamma])$.
Then $\overline{P}_{C_{\gamma}}$ fits in the projective flat family of relative compactified Jacobians 
$$
f:\overline{\cp}=\overline{\pic}{}^{\,0}_{\cc/ \cb}\to \cb.
$$ 
Let $j:\cb^{\rm sm}\to \cb$ be the open subscheme over which $f$ is smooth, and let $f^{\rm sm}$ be the restriction of $f$ to the inverse image of $\cb^{\rm sm}$. A fundamental result of Ng\^o \cite{ngo} states that
\begin{equation}\tag{1}\label{ngo support gl}
Rf_{*}\ql=\bigoplus_{i=0}^{2\delta_{\gamma}} j_{!*}(R^{\,i}f^{\sm}_{*}\ql)[-i],
\end{equation}
where $\delta_{\gamma}$ is the $\delta$-invariant of $C_{\gamma}$.
As $\Lambda^{0}\backslash\xx_{\gamma}^{0}$ is homeomorphic to $\overline{P}_{C_{\gamma}}=f^{-1}(0)$, this implies 
\begin{equation*}
H^{i}\big(\Lambda^{0}\backslash\xx_{\gamma}^{0}, \ql\big)=\bigoplus_{i'=0}^{i}\ch^{i-i'}\big(j_{!*}R^{i'}f^{\sm}_{*}\ql\big)_{0},\quad i=0,\cdots,2\delta_{\gamma}.
\end{equation*}

We would like to extend the above results to the affine Springer fibers. 
But the deformation theory of $\overline{P}{}_{C_{\gamma}}^{\natural}$ doesn't work well \cite{laumon springer}. Indeed, the deformation doesn't extend over the whole base $\cb$, and the total space is singular. 
As an alternate, we consider the deformation of the $\Lambda^{0}/n$-covering $\overline{P}_{n}$ of $\overline{P}_{C_{\gamma}}$, $(n,p)=1$. The homology of $\xx_{\gamma}^{0}$ can be recovered  as the limit
\begin{equation}\tag{2}\label{springer as limit} 
H_{*}(\xx^{0}_{\gamma},\zl)=H_{*}(\overline{P}{}_{C_{\gamma}}^{\natural}, \zl)=\varprojlim_{(n,p)=1} H_{*}\big(\overline{P}_{n},\bz_{\ell}\big).
\end{equation}
The idea is to extend the $\Lambda^{0}/n$-covering $\overline{P}_{n}\to \overline{P}_{C_{\gamma}}$ over the family $f:\overline{\cp}\to \cb$.
By the general theory of finite abelian coverings, this is equivalent to the extension of the finite \'etale group scheme $\pic_{\overline{P}_{C_{\gamma}}/k}[n]$ along the quasi-finite \'etale group scheme $\pic_{\overline{\cp}/\cb}[n]$ over $\cb$. 
By the autoduality of the compactified Jacobian, we have canonical isomorphisms
$$
\pic_{\overline{P}_{C_{\gamma}}/k}=\pic_{C_{\gamma}/k},\quad \pic_{\overline{\cp}/\cb}=\pic_{\cc/\cb},
$$
and the question can be reduced to that of the family $\pi:\cc\to \cb$.
By the canonical isomorphisms
$$
\pic_{C_{\gamma}/k}[n]=H^{1}(C_{\gamma},\mu_{n}),\quad \pic_{\cc/\cb}[n]=R^{1}\pi_{*}\mu_{n},
$$
and a support theorem for the complex $R\pi_{*}\ql$, we can confirm the existence of a quasi-finite \'etale covering $\cb_{n}$ of an open subscheme $ \cb^{\circ}$ of $\cb$ such that $\pic_{C_{\gamma}/k}[n]$ extends to a constant finite group scheme $\ct_{n}$ of $\pic_{\cc\times_{\cb}\cb_{n}/\cb_{n}}[n]$ over $\cb_{n}$. 
We obtain then $\Lambda^{0}/n$-coverings 
$$
\Psi_{n}:\cc_{n}\to \cc\times_{\cb}\cb_{n},\quad
\Phi_{n}: \overline{\cp}_{n}\to \overline{\cp}\times_{\cb}\cb_{n}.
$$ 
Compose them with the natural structural morphism to $\cb_{n}$, we get families 
$$
\pi_{n}:\cc_{n}\to \cb_{n}
\quad \text{and}\quad 
f_{n}:\overline{\cp}_{n}\to \cb_{n}.
$$  

\begin{customthm}{1}[Theorem \ref{support variant} and \ref{reduction theorem}]\label{main support variant}

For the family $f_{n}:\overline{\cp}_{n}\to \cb_{n}$, let $j_{n}:\cb_{n}^{\sm}\to \cb_{n}$ be the inclusion of the open subscheme over which $f_{n}$ is smooth, then 
$$
Rf_{n,*}\ql=\bigoplus_{i=0}^{2\delta_{\gamma}} j_{n,!*}(R^{\,i}f^{\sm}_{n,*}\ql)[-i]
\oplus 
\bigoplus \text{Similar terms from Levi subgroups}.
$$
Moreover, the remaining terms indexed by the Levi subgroups admits factorizations into simple factors resembling the main term.  
 
\end{customthm}

The proof follows the strategy of Ng\^{o} \cite{ngo}. 
But contrary to (\ref{ngo support gl}), there are quite some extra summands besides the main term in the decomposition of $Rf_{n,*}\ql$. 
This is due to the fact that the geometric fibers of $f_{n}$ can have multiple irreducible components, and one of the key technical difficulties is to determine them explicitly. 
By theorem \ref{irreducible components}, this is controlled by the finite group scheme $\ct_{n}$. To make it explicit, we relate $\ct_{n}$ to global sections of the sheaf $\varpi_{n}^{*}R^{1}\pi_{*}\bz/n$ via the isomorphism
$$
\varpi_{n}^{*}R^{1}\pi_{*}\bz/n \cong \pic_{\cc\times_{\cb}\cb_{n}/\cb_{n}}[n]\otimes \mu_{n}^{-1},
$$ 
where $\varpi_{n}:\cb_{n}\to \cb$ is the structural morphism. 
By the support theorem for the complex $R\pi_{*}\ql$, it is enough to determine these sections  over a generic point of the $\delta_{\gamma}$-stratum. A quasi-explicit expression of these sections is then obtained in proposition \ref{ci express}.
With it, we can determine the locus where the geometric fibers of $\pi_{n}:\cc_{n}\to \cb_{n}$ can have multiple irreducible components, and we can compute explicitly the number of irreducible components, see theorem \ref{pi 0 curve}. The result is then transferred to the family $f_{n}:\overline{\cp}_{n}\to \cb_{n}$ and we get theorem \ref{pi 0 picard}.
For the reduction of the remaining terms in theorem \ref{main support variant}, we use the idea of Laumon in \cite{laumon springer}, \S A.3, to identify the base change of the subfamilies of $\pi:\cc\to \cb$ over the relevant strata to their normalization with a versal deformation of a partial resolution of $C_{\gamma}$, see theorem \ref{strata}.

With theorem \ref{main support variant} and the equation (\ref{springer as limit}), we can deduce a similar expression for the homology of $\xx_{\gamma}^{0}$. Let $\cf=R^{1}\pi_{*}^{\sm}\bq_{\ell}$ and $\cf^{\vee}=\ch om(\cf, \bq_{\ell})$. Let ${\ce}\subset \cf$ be the local system of vanishing cycles on $\cb_{\{0\}}-\Delta$, let $\ce^{\perp}$ be the orthogonal complement of $\ce$ with respect to the cup product on $\cf$. Let $(\cf/\ce^{\perp})^{\vee}= \ch om(\cf/\ce^{\perp}, \bq_{\ell})\subset \cf^{\vee}$.

\begin{customthm}{2}[Theorem \ref{decomposition}]\label{main decomposition}

For any $i\in \bn_{0}$, we have
\begin{align*}
H_{i}\big({\xx}_{\gamma}^{0},\bq_{\ell}\big)=&
\bigoplus_{i'=0}^{i} 
{\rm Im}\Big\{\ch^{2\tau_{\gamma}-(i-i')}_{\{0\}}\big(j_{!*}\BigWedge^{i'}(\cf/\ce^{\perp})^{\vee}\big)\to  \ch^{2\tau_{\gamma}-(i-i')}_{\{0\}}\big(j_{!*}\BigWedge^{i'}\cf^{\vee}\big)\Big\}
\oplus 
\\
&
\bigoplus\text{Similar terms from Levi subgroups}.
\end{align*}

\end{customthm}

The decomposition is very similar in structure to  the Arthur-Kottwitz reduction of the affine Springer fibers \cite{chen fundamental}, but it is not true that the main term calculates the homology of the fundamental domain $F_{\gamma}$. In fact, it calculates the ${\Lambda^{0}}$-invariant subspace of $H_{i}\big({\xx}_{\gamma}^{0},\bq_{\ell}\big)$. 
Similar results hold for the remaining terms. This leads to the quite surprising result that we can  reduce the study of $H_{*}(\xx_{\gamma}^{0}, \bq_{\ell})$ to its $\Lambda^{0}$-invariant subspace, and that the latter has a sheaf-theoretic interpretation.

Theorem \ref{main decomposition} provides a natural deformation-theoretic framework to study the affine Springer fibers. As an application, we get a sheaf-theoretic reformulation of the purity hypothesis of Goresky, Kottwitz and MacPherson.

\begin{customcor}{3}\label{reduce purity int}

The homology of the affine Springer fiber is pure in the sense of Grothendieck-Deligne if and only if the group
$$
\bigoplus_{i'=0}^{i}{\rm Im}\Big\{\ch^{2\tau_{\gamma}-(i-i')}_{\{0\}}\big(j_{!*}\BigWedge^{i'}(\cf/\ce^{\perp})^{\vee}\big)\to  \ch^{2\tau_{\gamma}-(i-i')}_{\{0\}}\big(j_{!*}\BigWedge^{i'}\cf^{\vee}\big)\Big\}
$$
is pure of weight $-i$ for any $i\in \bn_{0}$.
\end{customcor}


With this sheaf-theoretic reformulation of the purity hypothesis, a natural question arises: how to determine the punctual weight of an intermediate extension? We propose the following conjecture:  
Let $X_{0}$ be an irreducible {smooth} quasi-projective algebraic variety over a finite field $\fq$, let $j:U_{0}\to X_{0}$ be the inclusion of an open subvariety,  let $\cf_{0}$ be a smooth $\ell$-adic sheaf on $U_{0}$, $\ell\neq p$. 
Let $x$ be a geometric point of $X$ lying over a closed point $x_{0}\in |X_{0}|$. We are interested in the weight of the fiber $(j_{!*}\cf)_{x}$, here $(j_{!*}\cf)_{x}$ is a complex of vector spaces endowed with the action of $\fr_{x_{0}}$.
Let $X_{\{x\}}$ be the strict Henselization of $X$ at $x$. Let $\bar{\xi}$ be a geometric generic point of $X_{\{x\}}$, the morphism $\rho_{x}:\pi_{1}(X_{\{x\}}\times_{X} U,\bar{\xi})\to \gl(\cf_{\bar{\xi}})$ is called the \emph{geometric local monodromy} of $\cf$ at $x$.  


\begin{customconj}{4}\label{monodromy-weight}

Under the above setting, suppose that $\cf_{0}$ is puncturally pure of weight $\beta\in \bz$, if the geometric local monodromy at $x$ is semisimple, then the fiber $(j_{!*}\cf)_{x}$ is pure of weight $\beta$.

\end{customconj}

Note that in case $\dim(X)=1$ the conjecture is a corollary of the monodromy-weight  theorem of Deligne \cite{weil2}, th\'eor\`eme 1.8.4.

Note that the purity hypothesis of Goresky, Kottwitz and MacPherson, under the form of corollary \ref{reduce purity int}, is not an immediate consequence of the above conjecture. 
Indeed, the sheaf $(\cf/\ce^{\perp})^{\vee}$ in the corollary doesn't globalize to a smooth $\ell$-adic sheaf on any dense open subvariety of $\cb$. Nonetheless, $\cf/\ce^{\perp}$ admits a filtration $0\subset \ce/\ce^{\perp}\subset \cf/\ce^{\perp}$ over $\cb_{\{0\}}$, with the grading $\ce/\ce^{\perp}$ irreducible and $\cf/\ce$ trivial.  
We hope that a proof of conjecture \ref{monodromy-weight} will shed some light on the purity hypothesis of Goresky, Kottwitz and MacPherson.

To finish, we want to mention that this work has been inspired by conjecture \ref{monodromy-weight}. In fact, the projective system $f_{n}:\overline{\cp}_{n}\to \cb_{n}, {(n,p)=1},$ has been constructed to kill the unipotent part of the geometric local monodromy of the sheaf $Rf_{*}^{\sm}\ql$ at $0$.

\subsection*{Notations, conventions and useful facts}

\paragraph{(1)} 

We consider only schemes which are separated and of finite type over $k$, unless stated otherwise. We fix an algebraic closure $\bar{k}$ of $k$. For a scheme $X$ over $k$, we denote by $|X|$ be the set of closed points of $X$. For $x\in |X|$, we use the notation $\bar{x}$ to denote a geometric point of $X$ lying over $x$, and vice versa. 
We denote by $k(x)$ the residue field at $x$.
We denote by 
$X_{\{\bar{x}\}}$ the strict Henselization of $X$ at $\bar{x}$.

\paragraph{(2)}

Let $X$ be a locally noetherian scheme, we denote by $X^{\reg}$ its regular locus and $X^{\rm sing}=X-X^{\reg}$ the singular locus. 
Let $f:X\to S$ be a morphism of schemes, we will denote by $\Delta_{f}$ the discriminant locus of $f$, and by $S^{\sm}$ the open subscheme $S-\Delta_{f}$ over which $f$ is smooth, and $f^{\sm}:X^{\rm sm}\to S^{\sm}$  the  restriction of $f$ to the inverse image of $S^{\sm}$. For any point $s\in S$, we denote by $X_{s}$ the fiber of $f$ at $s$. 

\paragraph{(3)}

Let $X$ be a locally noetherian scheme over $k$, we denote by $\cd_{c}(X,\ql)$ the derived category of $\ql$-constructible sheaves on $X$ defined by Deligne \cite{weil2}, \S1.1.2 and \S1.1.3. 
We denote by $\cp(X,\ql)$ the strict full subcategory of $\cd_{c}^{b}(X,\ql)$ consisting of the perverse $\ql$-sheaves. It is the heart of a $t$-structure on $\cd_{c}^{b}(X,\ql)$ associated to the autodual perversity. It can be shown that $\cp(X,\ql)$ is a notherian and artinian abelian category (cf. \cite{bbdg}). 
We denote by ${}^{p}\ch^{0}$ the cohomological functor $\tau_{\ge 0}\tau_{\le 0}$ from $\cd_{c}(X,\ql)$ to its heart $\cp(X,\ql)$. 
For $n\in \bz$, we denote by $[n]$ the shift to the left by $n$ in $\cd_{c}(X,\ql)$ and ${}^{p}\ch^{n}$ the composite ${}^{p}\ch^{0}\circ [n]$, we denote also by $(n)$ the Tate twist by $\mathbf{Q}_{\ell}(n)=\varprojlim_{m} \mu_{\ell^{m}}^{\otimes n}$.

\paragraph{(4)}

Let $X$ be a locally noetherian scheme over $k$, we denote by $\cd_{qc}(X)$ the derived category of  quasi-coherent $\co_{X}$-modules on $X$, and by $\cd_{coh}(X)$ the derived category of  coherent $\co_{X}$-modules on $X$.

\paragraph{(5)} 

Let $C^{\circ}=\spf(k[\![x,y]\!]/(f(x,y)))$ be a germ of plane curve singularity. The \emph{Milnor number} and \emph{Tjurina number} of $C^{\circ}$ are defined respectively by
$$
\mu=\dim(k[\![x,y]\!]/(\partial_{x}f, \partial_{y}f)), \quad \tau=\dim(k[\![x,y]\!]/(f, \partial_{x}f, \partial_{y}f)).
$$
Let $A=k[\![x,y]\!]/(f(x,y))$ and $\widetilde{A}$ the normalization of $A$, the \emph{conductor} of $A$ is the ideal $$\ka=\big\{a\in A\mid a\widetilde{A}\subset A\big\}.$$ It is known that $\dim(\widetilde{A}/A)=\dim(A/\ka)$, we call it the \emph{$\delta$-invariant} of $C^{\circ}$, denoted $\delta(C^{\circ})$ or $\delta(A)$. 

\paragraph{(6)} Let $C^{\circ}$ be a germ of plane curve singularity over $\bc$. Its irreducible components are called the \emph{branches}.  
It is known that every branch admits a \emph{Puiseux parametrization}
$$
\begin{cases}
X=t^{n},&\\
Y=\sum_{i=m}^{\infty} a_{i}t_{i}, &\text{for some }m>n.
\end{cases}
$$ 

\paragraph{(7)}

Two germs of plane curve singularities $(C^{\circ}_{1}, 0), (C^{\circ}_{2},0)$ are said to be \emph{of the same topological type} or \emph{equisingular} if there exist representatives $C_{i}\subset U_{i}$ of $C^{\circ}_{i}$ inside open neighbourhood $U_{i}\subset \bc^{2}$ of $0$, $i=1,2$, such that there exists homeomorphism $U_{1}\sim U_{2}$ sending $C_{1}$ to $C_{2}$. 
According to Zariski \cite{zariski equisingular}, they are equisingular if and only if there is a bijection between the branches of the singularities such that the datum consisting of the Puiseux characteristics of the branches and the intersection numbers between pairs of branches are the same under the bijection.

\paragraph{(8)}

Let $C^{\circ}=\spf(A)$ be a germ of plane curve singularity over $k$. We denote by $\defm^{\rm top}_{C^{\circ}}$ or $\defm^{\rm top}_{A}$ the deformation functor of $C^{\circ}$. It is known to be representable and we have a miniversal deformation of $C^{\circ}$. 
Moreover, the tangent space $\defm^{\rm top}_{C^{\circ}}(\bc[\![\varep]\!]/\varep^{2})$ can be identified with $\bc[\![x,y]\!]/(f, \partial_{x}f, \partial_{y}f)$.

\paragraph{(9)}

Let $C_{0}$ be a projective algebraic curve over $k$ with planar singularities, its \emph{$\delta$-invariant} $\delta(C_{0})$ is defined as the sum of the $\delta$-invariants of its singularities. Let $\psi: C\to S$ be a deformation of $C_{0}$ over a smooth base $S$, then the function $s\in S\mapsto \delta(C_{s})$ is upper semi-continuous (Teissier \cite{teissier resolution}, I-1.3.2). The subscheme $S_{\delta}$ parametrizing curves with $\delta$-invariant $\delta$ is then locally closed, and we obtain a stratification $S=\bigsqcup_{\delta=0}^{\delta(C_{0})}S_{\delta}$, called the \emph{$\delta$-stratification} of $S$.

\paragraph{(10)}

We denote by $\bn$ the set of natural numbers and $\bn_{0}=\bn\cup \{0\}$.

{\small
\paragraph*{Acknowledgement} 

We want to thank Prof. G\'erard Laumon for his constant support and encouragements during the years. Our work owes its birth to the long time effort to understand his ground breaking paper  \cite{laumon springer}.
Part of the work is done during the author's visit to the Research Centre for Mathematics and Interdisciplinary Sciences at Shandong University, we want to thank the institute for the wonderful hospitality.

}

\section{The affine Springer fibers and the compactified Jacobians}\label{springer as cj}

Let $\gamma\in \kg[\![\varep]\!]$ be a regular semisimple element. Let $\rtt$ be the centralizer of $\gamma$ in $G$, then $\rtt$ acts on the affine Springer fiber $\xx_{\gamma}$ by left translation. 
Let $\rS$ be the maximal $F$-split subtorus of $\rtt$.
Let $\Lambda\subset \rS(F)$ be the subgroup generated by $\chi(\varep),\,\chi\in X_{*}(\rS)$, then $\Lambda$ acts simply transitively on the irreducible components of $\xx_{\gamma}$. 
The connected components of $\xx_{\gamma}$ is naturally parametrized by $\bz$ with the morphism
$$
\xx_{\gamma}\to \bz,\quad [g]\mapsto \val(\det(g)).
$$
Let $\xx_{\gamma}^{0}$ be the central connected component of $\xx_{\gamma}$, then $\xx_{\gamma}\cong \xx_{\gamma}^{0}\times \bz$. 
Let $$
\Lambda^{0}=\{\lambda\in \Lambda\mid \val(\det(\lambda))=0\},
$$ 
then $\Lambda^{0}$ acts naturally on $\xx_{\gamma}^{0}$.
The quotient $\Lambda^{0}\backslash \xx_{\gamma}^{0}$ is a projective algebraic variety over $k$ and $\xx_{\gamma}^{0}\to \Lambda^{0}\backslash \xx_{\gamma}^{0}$ is an \'etale Galois covering of Galois group $\Lambda^{0}$.

\subsection{Laumon's work}\label{review laumon}

In an attempt to deform the affine Springer fibers, Laumon \cite{laumon springer} discovered a remarkable relationship between the affine Springer fibers and the compactified Jacobians of certain algebraic curves. 
By definition, the affine Springer fiber $\xx_{\gamma}$ parametrizes the lattices $L$ in $F^{n}$ such that $\gamma L\subset L$. Such lattices can be naturally identified with sub-$\co[\gamma]$-modules of finite type in $E:=F[\gamma]\cong F^{d}$. Moreover, such $\co[\gamma]$-modules can be globalized.   
Let $C_\gamma$ be an irreducible projective algebraic curve over $\bc$, with two points $c$ and $\infty$ such that
\begin{enumerate}[nosep, label=(\roman*)]
\item
$C_\gamma$ has unique singularity at $c$ and $\widehat{\co}_{C_\gamma,\,c}\cong \co[\gamma]$,

\item

the normalization of $C_\gamma$ is isomorphic to $\bp^{1}$ by an isomorphism sending $\infty\in \bp^{1}$ to $\infty\in C_\gamma$.

\end{enumerate}

\noindent The curve $C_{\gamma}$ will be called the \emph{spectral curve} of $\gamma$.
Let $M$ be a sub-$\co[\gamma]$-modules of finite type in $F[\gamma]$. 
Similar to the obvious fact that $\co_{C_{\gamma}}$ can be obtained by gluing $\widehat{\co}_{C_{\gamma},c}=\co[\gamma]$ and ${\co}_{C_{\gamma}\backslash\{c\}}$,  we can glue $M$ with the sheaf $\co_{C_{\gamma}\backslash\{c\}}$ along $(C_{\gamma}\backslash\{c\})\times_{C_{\gamma}} \spec(\widehat{\co}_{C_{\gamma},c})=\spec(E)$, to get a torsion-free coherent $\co_{C_{\gamma}}$-module $\cm$ of generic rank $1$. 
Recall that the compactified Jacobian $\overline{P}_{C_{\gamma}}:=\overline{\pic}{}^{0}_{C_{\gamma}/\bc}$ parametrizes the isomorphism class of the torsion-free coherent $\co_{C_{\gamma}}$-modules of generic rank $1$ and degree $0$.
Consider the functor $\overline{P}{}_{C_\gamma}^{\natural}$ which associates to an affine $k$-scheme $S$ the groupo\"id of couples $(M,\iota)$, where $M$ is a torsion free coherent $\co_{C_\gamma\times S}$-module of generic rank $1$ and degree $0$, and $\iota:M|_{(C_\gamma\backslash c)\times S}\cong \co_{(C_\gamma\backslash c)\times S}$ is a trivialization of $M$ over $(C_\gamma\backslash c)\times S$. One verifies that $\overline{P}{}_{C_\gamma}^{\natural}$ is representable by a $k$-scheme. With the glueing construction, we have a natural morphism
$$
\xx_{\gamma}^{0}\to \overline{P}_{C_\gamma}^{\natural}, 
$$
which sends a lattice $L\subset E$ satisfying $\gamma L\subset L$ to the torsion-free $\co_{C_\gamma}$-module obtained by gluing $L$ and $\co_{C_\gamma\backslash \{c\}}$ along $\spec(E)=(C_\gamma\backslash \{c\})\times_{C_\gamma}\spec(A)$, together with the obvious trivialization. Let $\overline{P}_{C_\gamma}^{\natural}\to \overline{P}_{C_\gamma}$ be the morphism forgetting the trivialization $\iota$. For the affine Springer fiber, this corresponds to the quotient $\xx_{\gamma}^{0}\to \Lambda^{0}\backslash \xx_{\gamma}^{0}$.

\begin{thm}[Laumon \cite{laumon springer}]\label{homeo SJ}

The morphism $\xx_{\gamma}^{0}\to \overline{P}_{C_\gamma}^{\natural}$ is finite radical and surjective, hence so is the induced morphism $\Lambda^{0}\backslash \xx_{\gamma}^{0}\to \overline{P}_{C_\gamma}$. In particular, both morphisms are universal homeomorphisms. 

\end{thm}

The Jacobian $J_{C_\gamma}:=\pic_{C_{\gamma}/k}^{0}$ can be described similarly. An invertible $\co_{C_\gamma}$-module is obtained by gluing a principal $A$-module $(a)$ with $a\in E^{\times}$ and $\co_{C_\gamma\backslash \{c\}}$ along $\spec(E)=(C_\gamma\backslash \{c\})\times_{C_\gamma}\spec(A)$, hence we get a morphism
\begin{equation}\label{describe pic}
\Lambda^{0}\backslash E^{\times}/A^{\times}\to \pic_{C_\gamma/k}.
\end{equation}
As before, it is finite radical and surjective. Let $\widetilde{A}$ be the normalization of $A$ in $E$, the above morphism restricts to 
\begin{equation}\label{describe jac}
\widetilde{A}^{\times}/A^{\times}\to J_{C_\gamma}.
\end{equation}
Let $\phi:\bp^{1}\to C_\gamma$ be the normalization map. The above description is compatible with the long exact sequence
\begin{align*}
1\to H^{0}(C_\gamma, \bbg_{m})\to H^{0}(\bp^{1}, \bbg_{m})&\to H^{0}(C_\gamma, \phi_{*}\bbg_{m}/\bbg_{m})\cong \widetilde{A}^{\times}/A^{\times}\\
& \to H^{1}(C_\gamma, \bbg_{m})\to H^{1}(\bp^{1}, \bbg_{m})\to 1. 
\end{align*}

In general, let $C$ be a geometrically integral projective curve over $k$ with planar singularities at $c_{j}, j\in J$. 
Let $\phi:\widetilde{C}\to C$ be the normalization of $C$, let $\phi^{-1}(c_{j})=\{\tilde{c}_{i}\mid i\in I_{j}\}$. 
For each $j\in J$, let $A_{I_{j}}$ be the completed local ring of $C$ at $c_{j}$, let $E_{I_{j}}$ be the ring of fractions of $A_{I_{j}}$.  
It is clear that the local branches of $C$ at $c_{j}$ are naturally parametrized by $I_{j}$. 
For simplicity, we assume that all the points $c_{j}, \tilde{c}_{i}$ are rational over $k$.

For each $j\in J$, let $\xx_{I_{j}}$ be the affine Springer fiber which parametrizes all the $A_{I_{j}}$-fractional ideals in $E_{I_{j}}$. It is equipped with the natural action of the $k$-group scheme $G_{I_{j}}$ ``defined" by $G_{I_{j}}(k)=E_{I_{j}}^{\times}/A_{I_{j}}^{\times}$. Let
$$
\xx(C)=\prod_{j\in J} \xx_{I_{j}},\quad G(C)=\prod_{j\in J} G_{I_{j}},\quad I=\bigsqcup_{j\in J} I_{j}, \quad \Lambda(C)=\bz^{I}.
$$
The group $G(C)$ admits a d\'evissage
$$
1\to G^{0}(C)\to G(C)\to \Lambda(C)\to 0,
$$
which at the level of $k$-points is
$$
1\to \prod_{j\in J} \co_{E_{I_{j}}}^{\times}/A_{I_{j}}^{\times} \to 
\prod_{j\in J} E_{I_{j}}^{\times}/A_{I_{j}}^{\times}\to \prod_{j\in J}\bz^{I_{j}}\to 1.
$$
The group scheme $G^{0}(C)$ appears as the maximal connected affine subgroup of the Jacobian $J_{C}$ of $C$, we have the d\'evissage of Chevalley
\begin{equation}\label{chevalley cj curve}
1\to G^{0}(C)\to J_{C}\to J_{\widetilde{C}}\to 1.
\end{equation}

As before, we have a $k$-morphism
\begin{equation}\label{uniform cj by x}
\xx(C)\to \overline{\pic}_{C/k}
\end{equation}
which sends $(M_{j}\subset E_{I_{j}})_{j\in J}$ to the torsion-free coherent $\co_{C}$-module $\cm$ of generic rank $1$ obtained by gluing $\co_{C\backslash \{c_{j}\}_{j\in J}}$ and the $M_{j}$'s. It is clear that we can twist the construction by an invertible sheaf on $C$, and  get a morphism
\begin{equation}\label{uniform cj general}
\big[\xx(C)\times J_{C}\big]/G^{0}(C)\to \overline{\pic}_{C/k}. 
\end{equation}
With the d\'evissage (\ref{chevalley cj curve}), we get a natural fibration
\begin{equation}\label{cj quotient j}
\big[\xx(C)\times J_{C}\big]/G^{0}(C)\to \pic_{\widetilde{C}/k}
\end{equation} 
with fiber $\xx(C)$.

Note that we have rigidified the invertible sheaf $\co_{C\backslash \{c_{j}\}_{j\in J}}$ in the construction of the morphism (\ref{uniform cj by x}). To get rid of it, consider the action of the group
$$
H^{0}(C\backslash \{c_{j}\}_{j\in J},\bbg_{m})/H^{0}(C, \bbg_{m})\subset \prod_{j\in J}E_{I_{j}}^{\times}/A_{I_{j}}^{\times}=G(C)(k)
$$
on $\xx(C)$. 
It is isomorphic to the image of a section $\sigma$ of the morphism $G(C)\to \Lambda(C)=\bz^{I}$ over the subgroup $\Lambda^{0}(C)=\Ker\big\{\bz^{I}\to \bz\big\}$. 
The morphism (\ref{uniform cj by x}) descends to 
$$
\sigma(\Lambda^{0}(C))\backslash\xx(C)\to \overline{\pic}_{C/k}.
$$
Similarly, the morphism (\ref{uniform cj general}) descends to
\begin{equation}\label{uniform cj general descent}
\sigma(\Lambda^{0}(C))\backslash \big[\xx(C)\times J_{C}\big]/G^{0}(C)\to \overline{\pic}_{C/k}, 
\end{equation}
and the fibration (\ref{cj quotient j}) descends to
\begin{equation}\label{cj quotient j descent}
\sigma(\Lambda^{0}(C))\backslash\big[\xx(C)\times J_{C}\big]/G^{0}(C)\to \pic_{\widetilde{C}/k}
\end{equation}
with fiber $\sigma(\Lambda^{0}(C))\backslash\xx(C)$.

\begin{prop}[Laumon \cite{laumon springer}]
\label{cj general descrip}

The $k$-morphism (\ref{uniform cj general descent}) is a universal homeomorphism.

\end{prop}

As before, we can construct a $\sigma(\Lambda^{0}(C))$-torsor $\overline{P}_{C}^{\natural}$ over the central connected component $\overline{P}_{C}$ of $\overline{\pic}_{C/k}$ by including the rigidification of $\co_{C\backslash \{c_{j}\}_{j\in J}}$ in the moduli problem. 
Let $\xx^{0}(C)$ be the central connected component of $\xx(C)$. We get then a Cartesian diagram
$$
\begin{tikzcd}
\big[\xx^{0}(C)\times J_{C}\big]/G^{0}(C)\arrow[r]\arrow[d]\arrow[rd, phantom, "\square"]& \overline{P}_{C}^{\natural}\arrow[d]\\
\sigma(\Lambda^{0}(C))\backslash\big[\xx^{0}(C)\times J_{C}\big]/G^{0}(C)\arrow[r]& \overline{P}_{C}.
\end{tikzcd}
$$

\begin{prop}[Laumon \cite{laumon springer}]
The morphism 
$$
\big[\xx^{0}(C)\times J_{C}\big]/G^{0}(C)\to \overline{P}_{C}^{\natural}
$$
is a universal homeomorphism.
\end{prop}

\subsection{Ng\^{o}'s support theorem}\label{old support}

Let $\pi:(\cc, C_{\gamma})\to (\cb, 0)$ be an algebraization of a miniversal deformation of $C_{\gamma}$, i.e. $\pi:\cc\to \cb$ is a proper flat family of curves with geometrically integral fibers such that the restriction of $\pi$ to the formal neighborhood $\widehat{\cb}_{0}$ of $0\in \cb$ is a miniversal deformation of  $C_{\gamma}$. 
The generic fiber of $\pi$ is then an irreducible smooth projective algebraic curve of genus $\delta_{\gamma}=\delta(C_{\gamma})$.
Let 
$$
f:\overline{\cp}=\overline{\pic}{}^{\,0}_{\cc/\cb}\to \cb
$$ 
be the relative compactified Jacobian of the family $\pi:\cc\to \cb$. The morphism $f$ is proper flat with geometrically integral fibers, and it is smooth whenever the fiber of $\pi$ is. As $f^{-1}(0)=\overline{P}_{C_\gamma}$, we get a natural deformation of $\overline{P}_{C_\gamma}$. 

According to Fantechi-G\"ottsche-van Straten \cite{fgv}, the total space $\overline{\cp}$ is smooth over $k$. 
As the morphism $
f:\overline{\cp}\to \cb
$
is proper, the complex $Rf_{*}\ql$ is pure by Deligne's Weil-II \cite{weil2}. The decomposition theorem of Beilinson, Bernstein, Deligne and Gabber \cite{bbdg} then implies that
$$
Rf_{*}\ql=\bigoplus_{n=0}^{2\delta_\gamma} {}^{p}\ch^{n}(Rf_{*}\ql)[-n].
$$

\begin{thm}[Ng\^o \cite{ngo}]\label{ngo support}

We have decomposition
$$
Rf_{*}\ql=\bigoplus_{n=0}^{2\delta_\gamma} j_{!*}(R^{n}f^{\sm}_{*}\ql)[-n],
$$
where $j:\cb^{\rm sm}\to \cb$ is the natural inclusion. In particular,
$$
H^{n}\big(\overline{P}_{C_\gamma}, \ql\big)=\bigoplus_{i=0}^{n}\ch^{n-i}\big(j_{!*}{\textstyle{\bigwedge^{i}}}R^{1}f^{\sm}_{*}\ql\big)_{0},\quad n=0,\cdots,2\delta_\gamma.
$$
 
\end{thm}

We make a brief recall of Ng\^o's proof. Let $X$ be a scheme of finite type over $k$, then any simple perverse sheaf on $X$ is of the form
$$
i_{x,*}\ic_{\overline{\{x\}}}(\cl_{x})[\rd_{x}]
$$
for some point (not necessarily closed) $x\in X$, and some simple $\ell$-adic local system $\cl_{x}$ on $\{x\}$, here
$
i_{x}: \overline{\{x\}}\hookrightarrow X
$ is the natural inclusion and $\rd_{x}=\dim(\overline{\{x\}})$.
A complex $K\in \cd_{c}^{b}(X,\ql)$ is said to be \emph{semisimple} if
$$
K=\bigoplus_{i}{}^{p}\ch^{i}(K)[-i].
$$
Then
$$
K=\bigoplus_{i}\bigoplus_{x\in X} i_{x,*} \ic_{\overline{\{x\}}}(\cl_{x}^{i})[\rd_{x}-i]
$$
for some points $x\in X$ and some semisimple $\ell$-adic local system $\cl_{x}^{i}$ on $\{x\}$. The set 
$$
\supp(K):=\{x\in X\mid \cl_{x}^{i}\neq 0 \text{ for some }i\}
$$
is called the \emph{support} of $K$. Let
$$
{\rm occ}_{x}(K)=\{i\mid \cl_{x}^{i}\neq 0\},
$$
and let
$$
n_{x}^{+}(K)=\max\{i\mid \cl_{x}^{i}\neq 0\},\quad  n_{x}^{-}(K)=\min\{i\mid \cl_{x}^{i}\neq 0\}.
$$
The quantity
$$
\amp_{x}(K):=n_{x}^{+}(K)-n_{x}^{-}(K)
$$
is called the \emph{amplitude} of $K$ at $x$. If $K$ is auto-dual, then $n_{x}^{-}=-n_{x}^{+}$ and $\amp_{x}(K)=2n_{x}^{+}(K)$.


The relative Jacobian $\cp=\pic_{\cc/\cb}^{0}$ acts naturally on $\overline{\cp}$.
They form a \emph{weak abelian fibration} in the sense of Ng\^o \cite{ngo}.
Recall that such a fibration over a $k$-scheme $S$ consists of a proper morphism $f:M\to S$ and a smooth commutative group scheme $g:P\to S$, with an action
$
\act: P\times_{S}M\to M
$
satisfying the properties:
\begin{enumerate}[topsep=2pt, noitemsep, label=(\arabic*)]
\item
The morphism $f$ and $g$ have the same relative dimension $d$.

\item

The action of $P$ on $M$ has only \emph{affine} stabilizers. 

\item

Let $P^{0}$ be the open sub group scheme of the neutral connected components of the fibers of $P$ and let $g^{0}:P^{0}\to S$ be the restriction of $g$ to $P^{0}$, let
$$
\rT_{\ql}(P^{0})=H^{-1}\big(g^{0}_{!}\ql[2d](d)\big)
$$
be the sheaf of the Tate modules of $P^{0}$, then $\rT_{\ql}(P^{0})$ is \emph{polarizable}, i.e. there exists an alternating bilinear form
$$
\rT_{\ql}(P^{0})\times \rT_{\ql}(P^{0})\to \ql(1)
$$ 
such that over any geometric point $\bar{s}$ of $S$, the bilinear form has kernel $\rT_{\ql}(R_{\bar{s}})$ and induces a perfect pairing for $\rT_{\ql}(A_{\bar{s}})$, where $R_{\bar{s}}$ and $A_{\bar{s}}$ are factors of the d\'evissage of Chevalley
\begin{equation}\label{chevalley 0}
1\to R_{\bar{s}} \to P^{0}_{\bar{s}}\to A_{\bar{s}}\to 1.
\end{equation}

\end{enumerate}

\noindent The key observation is that the abelian part $A_{\bar{s}}$ of $P^{0}_{\bar{s}}$ doesn't contribute to the degeneration of the nearby fibers of $M_{\bar{s}}$ to $M_{\bar{s}}$. In other words, it behaves as $A_{\bar{s}}$ can be extended to an \'etale neighborhood of $\bar{s}$. More precisely, with the action $
\act: P\times_{S}M\to M,
$
Ng\^o constructs a canonical morphism
\begin{equation}\label{tate module action 0}
\rT_{\ql}(P^{0})\otimes {}^{p}\ch^{n}(Rf_{*}\ql)\to {}^{p}\ch^{n-1}(Rf_{*}\ql).
\end{equation}
Let $K_{s}^{n}$ be the direct sum of the simple perverse factors of ${}^{p}\ch^{n}(Rf_{*}\ql)$ with support $s\in S$, its restriction to an open subscheme $V_{s}$ of $\overline{\{s\}}$ is a \emph{pure} local system denoted $\ck^{n}_{s}[-n]$.
Let 
$$
\ck_{s}=\bigoplus_{n} \ck_{s}^{n}[-n].
$$ 
Moreover, after further shrinking $V_{s}$, we can find a finite radical base change $V_{s}'\to V_{s}$, such that the d\'evissage of Chevalley (\ref{chevalley 0}) at a geometric point $\bar{s}$ over $s$ descends to $V_{s}'$ as 
$$
1\to R_{s}\to P^{0}|_{V_{s}'}\to A_{s}\to 1,
$$ 
where $R_{s}$ is a fiberwise connected smooth affine group scheme on $V_{s}'$ and $A_{s}$ is an Abelian scheme on $V_{s}'$. This implies an exact sequence of Tate modules
$$
0\to \rT_{\ql}(R_{s})\to \rT_{\ql}(P^{0}|_{V_{s}'})\to \rT_{\ql}(A_{s})\to 0. 
$$ 
As $V_{s}'$ is finite radical over $V_{s}$, they have equivalent \'etale sites, and we can consider the above exact sequence as objects over $V_{s}$. 
Notice that the Tate module $\rT_{\ql}(A_{s})$ is pure of weight $-1$, while the Tate module $\rT_{\ql}(R_{s})$ relies only on the toric part of $R_{s}$, and is therefore pure of weight $-2$.
As $\ck_{s}^{n}[-n]$ is pure of weight $n$, comparing the weights, we find that the action (\ref{tate module action 0}) actually factorizes as
$$
\rT_{\ql}(A_{s})\otimes \ck_{s}^{n}[-n]\to \ck_{s}^{n-1}[-n+1]. 
$$
Let $\boldsymbol{\Lambda}_{s}^{\vee,\bullet}=\bigoplus_{i}\bigwedge^{i}\rT_{\ql}(A_{s})[i]$. We obtain then a morphism
$$
\boldsymbol{\Lambda}_{s}^{\vee,\bullet}\otimes \ck_{s}\to \ck_{s}.
$$

\begin{prop}[Ng\^o \cite{ngo}, prop. 7.4.10]\label{ngo freeness}
The local system $\ck_{s}$ is a free graded module over the graded algebra $\boldsymbol{\Lambda}_{s}^{\vee,\bullet}$. 

\end{prop}

Let $\boldsymbol{\Lambda}_{s}^{\bullet}=\bigoplus_{i} H^{i}(A_{s}, \ql)[-i]$ be the ring of cohomology of $A_{s}$. The freeness property implies that the local system $\ck_{s}$ is of the form $\Lambda_{s}^{\bullet}[\dim(A_{s})]\otimes \cl_{s}$ for some local system $\cl_{s}$, whence the inequality
\begin{equation*}\label{amp inequality}
\amp_{s}(Rf_{*}\ql)\ge 2\dim(A_{s}).
\end{equation*}
Let $\delta_{s}^{\aff}=\dim(R_{s})$, then the above inequality can be rewritten as
\begin{equation}\label{amp inequality}
{\frac{1}{2}}\amp_{s}(Rf_{*}\ql[d_{M}])\ge d-\delta_{s}^{\aff}.
\end{equation}
On the other hand, assume that $M$ is smooth over $k$, we have the inequality of Goresky and MacPherson (\cite{ngo}, th\'eor\`eme 7.3.1),
\begin{equation}\label{gm inequality}
{\frac{1}{2}}\amp_{s}(Rf_{*}\ql[d_{M}])\le d-\codim_{S}(\{s\}).
\end{equation}
The inequality as stated is different from that of the theorem, but the proof is the same. 
Let $\cl'$ be an extra local system over $\{s\}$ with degree $n\in {\rm occ}_{s}\big(Rf_{*}\ql[d_{M}]\big), n\ge 0,$ i.e. $i_{s,*}\ic_{\overline{\{s\}}}\cl'[\rd_{s}-n]$ appears as a direct summand of $Rf_{*}\ql[d_{M}]$. 
As $Rf_{*}\ql$ is non-trivial only in cohomological degrees $[0,2d]$, this implies
\begin{equation*}
n-\rd_{s}\le 2d-d_{M}. 
\end{equation*}
With the equality $d_{M}=\dim(S)+d$, this simplifies to 
\begin{equation*}
n\le d-\codim_{S}(\{s\}).
\end{equation*}
As $M$ is smooth and $f$ is proper, the complex $Rf_{*}\ql[d_{M}]$ is autodual, and we have
$$
{\frac{1}{2}}\amp_{s}(Rf_{*}\ql[d_{M}])=\max\Big\{n\, \big |\,n\in {\rm occ}_{s}\big(Rf_{*}\ql[d_{M}]\big)\Big\},
$$ 
whence the inequality (\ref{gm inequality}).
We combine (\ref{amp inequality}) and (\ref{gm inequality}) as
$$
d-\delta_{s}^{\aff}\le \frac{1}{2}\amp_{s}(Rf_{*}\ql[\rd_{M}])\le d-\codim_{S}(\{s\}).
$$
In case that $\codim_{S}(\{s\})\ge \delta_{s}^{\aff}$, both equalities hold in the above inequality and we have
\begin{equation}\label{amph max}
\frac{1}{2}\amp_{s}(Rf_{*}\ql[\rd_{M}])= d-\delta_{s}^{\aff}=d-\codim_{S}(\{s\}).
\end{equation}
On the other hand, by proposition \ref{ngo freeness}, the extra local systems over $s$ always appear in group of the form $\boldsymbol{\Lambda}_{s}^{\bullet}[\dim(A_{s})]\otimes \cl$. 
As the complex $Rf_{*}\ql[\rd_{M}]$ is autodual, we can assume that $\cl$ is of degree $\ge 0$.  
Hence the top degree piece $\boldsymbol{\Lambda}_{s}^{\rm top}[\dim(A_{s})]\otimes \cl=\cl[-\dim(A_{s})]$ is of degree  
\begin{equation}\label{trivial n}
n\ge\dim(A_{s})=d-\delta_{s}^{\aff}.
\end{equation}
The equations (\ref{amph max}) and (\ref{trivial n}), together with the trivial $n\le \frac{1}{2}\amp_{s}(Rf_{*}\ql[\rd_{M}])$, implies that 
$$
n=\dim(A_{s})=d-\delta_{s}^{\aff}.
$$
Then $\cl$ will be of degree $0$, and $\boldsymbol{\Lambda}_{s}^{\rm top}\otimes \cl$ appears as a direct summand of $(R^{2d}f_{*}\ql)_{s}$.
We summarize the above discussions as:

\begin{thm}[Ng\^o \cite{ngo}, proposition 7.3.2] \label{ngo support main}

Let $S$ be a separated finite type $k$-scheme. Let $f:M\to S$ be a projective morphism of pure relative dimension $d$, equipped with the action of a smooth group scheme $g:P\to S$. Suppose that $M$ is smooth over $k$ and that $P$ acts on $M$ with affine stabilizers, and that the Tate module of $P$ is polarizable. Then any support $s$ of $Rf_{*}\ql$ satisfies the inequality
$$
d-\delta_{s}^{\aff}\le \frac{1}{2}\amp_{s}(Rf_{*}\ql)\le d-\codim_{S}(\{s\}).
$$
In case that both equalities hold at some support $\{s\}$, the extra local systems appears in the form $\boldsymbol{\Lambda}_{s}^{\bullet}\otimes \cl$, $\cl$ of degree $0$ such that $\boldsymbol{\Lambda}_{s}^{\rm top}\otimes \cl$ appears as direct summands of $({R^{2d}f_{*}\ql})_{s}$.

\end{thm}

We then apply the theorem to the family $f:\overline{\cp}\to \cb$. 
Let $b\in \cb$ be a support of $Rf_{*}\ql$. 
According to Diaz and Harris \cite{dh}, we have the inequality of Severi
\begin{equation}\label{lower bound delta}
\codim_{\cb}(\{b\})\ge \delta(\cc_{b}).
\end{equation}
On the other hand, we have the equality 
$$
\delta_{b}^{\aff}=\delta(\cc_{b}),
$$ 
hence the inequality of theorem \ref{ngo support main} must be equalities, and the extra local systems $\cl_{b}$ will satisfy the condition that $\boldsymbol{\Lambda}_{s}^{\rm top}\otimes \cl_{b}$ appear as a direct summand of $(R^{2\delta_{\gamma}}f_{*}\ql)_{s}$. As the geometric fibers of $f$ are irreducible, there cannot be any extra support $b$ other than the generic point of $\cb$. This finishes the sketch of proof of theorem \ref{ngo support}.

\section{The monodromy group of plane curve singularities}\label{monodromy curve}

In this section, we recollect some facts about the monodromy group of plane curve singularities, with particular emphasis on the work of A'Campo \cite{ac}.

Let $C_{0}=\spf\big(\bc[\![x,y]\!]/f(x,y)\big)$ be a plane curve singularity with $f(x,y)\in \bc[x,y]$.
Let $\varphi:(C, C_{0})\to (S, 0)$ be a miniversal deformation of $C_{0}$. 
Let $\tau$ be the Tjurina number of $C_{0}$. Then $S$ can be taken to be $\widehat{\ba}^{\tau}$, the completion of ${\ba}^{\tau}$ at $0$.  
Let $\Delta_{\varphi}\subset S$ be the discriminant locus of $\varphi$, then $R^{1}\varphi_{*}\ql$ is smooth over $S-\Delta_{\varphi}$. Let $\bar{\xi}$ be a geometric generic point of $S$, let $V=(R^{1}\varphi_{*}\ql)_{\bar{\xi}}$. Consider the monodromy action 
$$
\rho: \pi_{1}(S-\Delta_{\varphi}, \bar{\xi})\to \gl(V).
$$ 
Let $G_{\rho}$ be the Zariski closure of ${\rm Im}(\rho)$.


Let $L\subset S$ be a line in general position and sufficiently close to $0$. 
As $S$ is smooth at $0$, by the local Lefschetz theorem for the fundamental groups (cf. \cite{hl} or \cite{gm morse} Part II, Chap. 5, \S5.3), the morphism 
$$
\pi_{1}(L-\Delta_{\varphi},\bar{\xi})\to \pi_{1}(S-\Delta_{\varphi},\bar{\xi})
$$   
induced by the inclusion is surjective. 
Let $L\cap \Delta_{\varphi}=\{s_{i}\}_{i\in I}$. 
Let $\delta_{i}\in V$ be the vanishing cycle at $s_{i}$, then $\pi_{1}(L-\Delta_{\varphi},\bar{\xi})$ is the free group generated by the Picard-Lefschetz transformations
\begin{equation}\label{picard-lef}
\gamma_{i}(\alpha)=\alpha\pm \langle \alpha, \delta_{i} \rangle \delta_{i}, \quad  \alpha\in V, i\in I. 
\end{equation}
Let $\pair{\, ,}: V\times V\to \ql$ be the cup product on $V$, we ignore the Tate twist $(-1)$ here. Then the monodromy action of $\pi_{1}(S-\Delta_{\varphi},\bar{\xi})$ respects the pairing, and $G_{\rho}$ is a closed subgroup of $\Sp(V)$. 
Let $E\subset V$ be the subspace generated by the $\delta_{i}, i\in I,$ and their conjugates under the action of $\pi_{1}(S-\Delta_{\varphi},\bar{\xi})$, let $E^{\perp}$ be its orthogonal complement with respect to the pairing $\pair{\,,}$. Then we have a non-degenerate alternating pairing
$$
\pair{\,,}:E/(E\cap E^{\perp})\times E/(E\cap E^{\perp})\to \ql.
$$ 
The group $\pi_{1}(S-\Delta_{\varphi},\bar{\xi})$ acts on the quotient $W:=E/(E\cap E^{\perp})$ and the action respects the pairing. Let
$$
\bar{\rho}: \pi_{1}(S-\Delta_{\varphi}, \bar{\xi})\to \gl(W)
$$ 
be the resulting monodromy action. 

\begin{thm}\label{monodromy}

The Zariski closure of ${\rm Im}(\bar{\rho})$ is $\Sp(W)$.

\end{thm}

\begin{proof}

The theorem is well-known and we don't know to whom it should be attributed. 
The proof follows the same lines of reasoning as Deligne \cite{weil1}, \S5. 
Similar to theorem 5.4 in \textit{loc.\ cit.}, we can show that, up to sign, the vanishing cycles $\{\delta_{i}\}_{i\in I}$ are conjugate under the action of $\pi_{1}(S-\Delta_{\varphi},\bar{\xi})$. This is due to the fact that $\Delta_{\varphi}$ is {locally irreducible} at $0$ (cf. \cite{ebeling}, \S3.8, prop. 3.21(iv)), and we can move in the smooth locus of $\Delta_{\varphi}$. 
As a corollary, we obtain that the action $\bar{\rho}$ is absolutely irreducible, from which we deduce analogue of the theorem of Kazhdan-Margulis. 

\end{proof}

A precise description of the monodromy action has been obtained by A'Campo \cite{ac} and Gusein-Zade \cite{gz1}, \cite{gz2}. 
A \emph{morsification} for the singularity $C_{0}$ is a one-parameter deformation $\tilde{f}(x,y,t)\in \bc[x,y,t]$ such that the polynomial $\tilde{f}(x,y,t_{0})$ for $|t_{0}|$ sufficiently small has $\mu$ critical points in $B_{\ep}$, where $\mu$ is the Milnor number of $C_{0}$ and  $B_{\ep}$ is a sufficiently small open ball in $\bc^{2}$ at $0$. 
It gives rise to the same Morse datum as the restriction of the miniversal family $\varphi:C\to S$ to the inverse image $\varphi^{-1}(L)$.
In particular, the monodromy group ${\rm Im}(\rho)$ is generated by the Picard-Lefschetz transformations for the vanishing cycles of the Morse function $\tilde{f}(x,y,t_{0})$ on $B_{\ep}$.

Assume that $f(x,y)$ can be decomposed into products of analytically irreducible factors 
\begin{equation}\label{factorize f}
f(x,y)={\textstyle\prod_{j\in J}}f_{j}(x,y)
\end{equation}
such that the factors are all \emph{real} polynomials\footnote{The assumption imposes no loss of generality, as all the plane curve singularities admit equisingular deformation to such singularities.  Indeed, the datum consisting of the number of branches of the singularity and the $\delta$-invariants of the branches characterizes the equisingular deformation of the plane curve singularities over a smooth base \cite{teissier resolution}, it suffices to find polynomial of the required form with the same such datum.}, A'Campo constructed a morsification $\tilde{f}$ for $f$. 
The construction begins with a tower of blowing-ups
$$
\pi_{i}: X_{i}\to X_{i-1},\quad i=1,\cdots, N,
$$
of $X_{0}=\bc^{2}$, such that for the composite $\pi_{N}\circ\cdots\circ \pi_{1}$ the strict transform $C_{0}^{(N)}$ of $C_{0}$ is smooth and the exceptional divisor is of normal crossing.
Let $B_{i}$ be the exceptional divisor of the blowing-up $\pi_{i}$ and let $E_{i}$ be the exceptional divisor for the composite $\pi_{i}\circ\cdots\circ \pi_{1}$, $i=1,\cdots,N$. We  move $C_{0}^{(N)}$ slightly such that it is in generic position with respect to $E_{N}$, then contract $B_{N}$ to go back to $X_{N-1}$, let $C_{0}^{(N-1)}$ be the contract of the displacement of $C_{0}^{(N)}$. Move $C_{0}^{(N-1)}$ to a generic position with respect to $E_{N-1}$, contract $B_{N-1}$ to go back to $X_{N-2}$, we get $C_{0}^{(N-2)}$. This process can be iterated, at the end we get a curve $C_{0}^{(0)}$ in $X_{0}=\bc^{2}$. It has the property that its singularities are analytically locally isomorphic to a union of lines, but it might happen that the singularities are not ordinary double. To achieve that, we make a final deformation $C'_{0}$ such that all the singularities are nodal. The construction can be put in a one parameter family and we get a one parameter deformation $\tilde{f}(x,y,t)\in \bc[x,y,t]$ which is a morsification of $f$. Moreover, we can arrange that $\tilde{f}$ is of \emph{real} coefficients and all its critical points in $B_{\ep}$ are of \emph{real} coordinates. Such a morsification is said to be \emph{real}.

For a generic $t_{0}\in \bc$ with $|t_{0}|$ sufficiently small, the curve $C_{t_{0}}$ defined by the polynomial $\tilde{f}(x,y,t_{0})=0$ is a $\delta$-invariant deformation of $C_{0}$ and it has $\delta(C_{0})$ ordinary double points. 
Moreover, as a function over $D_{\ep}=B_{\ep}\cap \br^{2}$, $\tilde{f}(x,y,t_{0})$ has $\mu$ critical points, 
and the Morse index at a critical point $a\in D_{\ep}$ equals 
$$
{\rm Ind}_{a}(\tilde{f}\,)=\begin{cases}
0,&\text{if }\tilde{f}(a)<0,\\
1,&\text{if }\tilde{f}(a)=0,\\
2,&\text{if }\tilde{f}(a)>0.\\\end{cases}
$$
The configuration of the critical points can be read off from the real locus $\overline{C}_{t_{0}}:=C_{t_{0}}\cap D_{\ep}$. The curve $\overline{C}_{t_{0}}$ has only ordinary double points, it cuts $D_{\ep}$ as union of several connected components. Consider the connected components which are disjoint from the boundary of $D_{\ep}$, we call them the \emph{regions}. 
At each region, $\tilde{f}$ attains an extreme value (a maximum if $\tilde{f}>0$ and a minimum if $\tilde{f}<0$), hence each region represents a critical point of non-zero critical value of $\tilde{f}$. The critical points with critical value $0$ are exactly the double points of $\overline{C}_{t_{0}}$. Let $R_{1}^{+},\cdots, R_{\mu_{+}}^{+}$ be regions where $\tilde{f}$ is positive, let $D_{1},\cdots, D_{\mu_{0}}$ be the double points of $\overline{C}_{t_{0}}$, let $R_{1}^{-},\cdots, R_{\mu_{-}}^{-}$ be regions where $\tilde{f}$ is negative. 
Then 
\begin{equation}\label{delta-mu}
\mu_{0}=\delta(C_{0})\quad \text{and}\quad \mu_{+}+\mu_{0}+\mu_{-}=\mu.
\end{equation}
Let $\alpha_{i}^{+}$ be the vanishing cycle at the critical point in the region labeled $R_{i}^{+}$, $i=1,\cdots,\mu_{+}$; Let $\alpha_{j}$ be the vanishing cycle at the double point $D_{j}$, $j=1,\cdots, \mu_{0}$; Let $\alpha_{k}^{-}$ be the vanishing cycle at the critical point in the region labeled $R_{k}^{-}$, $k=1,\cdots,\mu_{-}$. Their union forms a basis called the \emph{distinguished basis} of $H_{1}(F,\bz)$, where $F=\varphi^{-1}(t)\cap B_{\ep}$ is the Milnor fiber of the singularity $C_{0}$.   
The intersection number of the vanishing cycles can be calculated as:

\begin{table}[h]\label{Tab: inter}
\centering
\begin{tabular}{r  l }
$\pair{\delta_{i},\delta_{j}}=1$ & if \,(1) $\delta_{i}=\alpha_{i}^{+}, \delta_{j}=\alpha_{j}$  and  $D_{j}\in \overline{R_{i}^{+}},$\\
& \quad (2) $\delta_{i}=\alpha_{i}, \delta_{j}=\alpha_{j}^{-}$  and  $D_{i}\in \overline{R_{j}^{-}},$\\
& \quad (3) $\delta_{i}=\alpha_{i}^{+}, \delta_{j}=\alpha_{j}^{-}$  and  $\overline{R_{i}^{+}}\cap \overline{R_{j}^{-}}$ is a segment of $\overline{C}_{t_{0}}$,\\
=0 & at all the other cases.
\end{tabular}
\caption{Intersection number of vanishing cycles}
\end{table}

\noindent With these intersection numbers, we get a precise form of the Picard-Lefschetz transformations (\ref{picard-lef}), they generate the monodromy group ${\rm Im}(\rho)$.  
These intersection numbers can be encoded in a Dynkin diagram. The vertices are indexed by the critical points. We use the symbol $\oplus$ for the critical points in $R_{i}^{+}$, the symbol $\bullet$ for the critical points $D_{j}$ and the symbol $\ominus$ for the critical points in $R_{k}^{-}$.  
Moreover, we draw an edge between two vertices if the intersection number of the corresponding vanishing cycles is non-zero. 

\begin{example}[A'Campo \cite{ac}]\label{ac gl4}

Let $f(x,y)=xy(x^{2}-y^{2})$, then for any $t\in \br, t\ne 0$, $\tilde{f}(x,y)=(x-t)(y+2t)(x^{2}-y^{2})$ is a real morsification for $f$. The curve $\overline{C}_{t_{0}}$ and the regions looks like
\begin{figure}[h]
\centering
\begin{tikzpicture}
\draw (0,0) circle (2cm);
\draw (1.42, 1.42)--(-1.42,-1.42);
\draw (1.42, -1.42)--(-1.42,1.42);
\draw (1.93, 0.5)--(-1.93, 0.5);
\draw (-1, 1.75)--(-1, -1.75);

\node at (0,0) {$\bullet$};
\node at (0.5,0.5) {$\bullet$};
\node at (-0.5,0.5) {$\bullet$};
\node at (-1,-1) {$\bullet$};
\node at (-1,0.5) {$\bullet$};
\node at (-1,1) {$\bullet$};

\node at (-0.6,0) {$\oplus$};
\node at (0,0.3) {$\ominus$};
\node at (-0.86,0.67) {$\ominus$};

\end{tikzpicture}
\end{figure}

\noindent The Dynkin diagram is
$$
\begin{tikzcd}
\bullet \arrow[r, dash] &\ominus \arrow[r, dash] \arrow[rd, dash]\arrow[d, dash]&\bullet \arrow[r, dash] \arrow[d, dash] &\ominus \arrow[r, dash] \arrow[d, dash] \arrow[ld, dash] &\bullet\\
&\bullet \arrow[r, dash]& \oplus \arrow[r, dash] \arrow[d, dash] &\bullet\\
&&\bullet&&
\end{tikzcd}
$$

\cqfd

\end{example}

\begin{rem}\label{ac large char}

The construction of this section continues to work over a field $\kappa$ such that ${\rm char}(\kappa)$ equals $0$ or is large enough compared to $\deg(f)$. Indeed, in these cases we can still resolve the singularity by successive blowing-ups of $\ba^{2}$ and make the perturbations as A'Campo.  

\end{rem}

\section{Finite abelian coverings of the compactified Jacobians}

In this section, we construct finite abelian coverings $\overline{P}_{n}\to \overline{P}_{C_{\gamma}}$, and extend them \emph{equivariantly} over the family $f:\overline{\cp}\to \cb$, to get finite abelian coverings $\overline{\cp}_{n}\to \cp\times_{\cb}\cb_{n}$. 
We then make a detailed analysis of the irreducible components of the geometric fibers of the family $
f_{n}:\overline{\cp}_{n}\to \cb_{n}.
$

\subsection{Finite abelian coverings of curves and their compactified Jacobians}

Let $C$ be a projective geometrically irreducible algebraic curve over $k$ with planar singularities. We are interested in the finite abelian coverings of the curve $C$ and its compactified Jacobian $\overline{P}_{C}$. The question is closely related to the {autoduality of the compactified Jacobian}.

Without loss of generality, we can assume that $C^{\reg}(k)\neq \emptyset$.
For any closed point $c\in C$, its defining ideal $\km_{c}$ is a simple torsion-free coherent $\co_{C}$-module of generic rank $1$. Fix a line bundle $L$ on $C$ of degree $1$, we can define a closed embedding, called the \emph{Abel-Jacobi map}:
$$
A_{L}:C\to \overline{P}_{C},\quad c\mapsto \km_{c}\otimes L. 
$$ 
The pull-back then defines a morphism
$$
A_{L}^{*}:\pic^{0}_{\overline{P}_{C}}\to \pic^{0}_{C/k}.
$$ 
In case that $C$ is smooth, $A_{L}^{*}$ is an isomorphism, this is the well-known {autoduality of the Jacobian}. It is proven by a careful analysis of the intersection of $A_{L}(C)$ with translations of the theta divisor $\Theta$ of $J_{C}$. The inverse of $A_{L}^{*}$ in this case is
$$
\beta: J_{C}\to \pic^{0}_{J_{C}},\quad x\mapsto \co_{J_{C}}(x+\Theta)^{-1}\otimes \co_{J_{C}}(\Theta). 
$$ 
In case that $C$ is geometrically integral with planar singularities, Esteves, Gagn\'e and Kleiman \cite{egk} partially  generalized the autoduality. Let $\cm^{\uni}$ be the universal sheaf on $C\times \overline{P}_{C}$, they constructed a morphism generalizing the above one
\begin{equation}\label{egk beta}
\beta:J_{C}\to \pic_{\overline{P}_{C}}^{0},\quad \cl\mapsto \det Rq_{2,*}(q_{1}^{*}\cl\otimes \cm^{\uni})^{-1}\otimes \det Rq_{2,*}\cm^{\uni},
\end{equation}
where $q_{1},q_{2}$ are the projections of $C\times \overline{P}_{C}$ to the first and second factor, and showed that
$$
A_{L}^{*}\circ \beta=\Id_{J_{C}}.
$$
Their results are complemented and extended by Arinkin \cite{arinkin 1}, \cite{arinkin 2}. He showed that $\beta$ is exactly the inverse to $A_{L}^{*}$. Moreover, the work of Esteves, Gagn\'e and Kleiman actually implies the existence of a line bundle $\mathscr{P}^{\uni}$ on $J_{C}\times \overline{P}_{C}$: Let $\cl^{\uni}$ be the universal line bundle on $C\times J_{C}$, it is the restriction of $\cm^{\uni}$ to $C\times J_{C}$, let $\pr_{ij}$ be the projection of $C\times J_{C}\times \overline{P}_{C}$ to its $(i,j)$-th factor, $i,j=1,2,3$, we define
\begin{equation}\label{egk poincare}
\scp^{\uni}=\det R\pr_{23,*}(\pr_{12}^{*}\cl^{\uni}\otimes \pr_{13}^{*}\cm^{\uni})^{-1}\otimes \det R\pr_{23,*}\,\pr_{13}^{*}\cm^{\uni}.
\end{equation}
Then $\scp^{\uni}$ induces the morphism $\beta:J_{C}\to \pic^{0}_{\overline{P}_{C}}$.
Arinkin extended the line bundle $\scp^{\uni}$ on $J_{C}\times \overline{P}_{C}$ to a coherent sheaf $\overline{\scp}{}^{\uni}$ on $\overline{P}_{C}\times \overline{P}_{C}$, and defined a \emph{Fourier-Mukai} transform
$$
\kff: \cd_{qc}^{b}(\overline{P}_{C})\to \cd_{qc}^{b}(\overline{P}_{C}), \quad M\mapsto Rp_{2,*}(Rp_{1}^{*}M\otimes \overline{\scp}{}^{\uni}),
$$
where $p_{i}: \overline{P}_{C}\times \overline{P}_{C}\to \overline{P}_{C}$ is the projection to the $i$-th factor, $i=1,2$, and proved that $\kff$ defines an auto-equivalence of derived categories. We summarize the above discussion as:

\begin{thm}[Arinkin \cite{arinkin 1}, \cite{arinkin 2}]\label{autoduality main}

The morphism $
\beta:J_{C}\to \pic^{0}_{\overline{P}_{C}}
$ 
is an isomorphism, and the Fourier-Mukai transform 
$
\kff: \cd_{qc}^{b}(\overline{P}_{C})\to \cd_{qc}^{b}(\overline{P}_{C})
$ 
defines an auto-equivalence of derived categories. 

\end{thm}

Let $X$ be a connected noetherian scheme, recall that the isomorphism classes of the finite abelian coverings of $X$ with Galois group $A$ is classified by
$$
\Hom_{\text{cont.}}\big(\pi_{1}(X,\bar{\xi}),\,A\big)\cong H^{1}(X,A),
$$
where $\bar{\xi}$ is a geometric generic point of $X$. 
In the essential case that $A=\bz/n$, the finite abelian coverings are classified by 
$$
H^{1}(X,\bz/n)=\pic_{X/k}[n]\otimes \mu_{n}^{-1}.
$$
This applies in particular to the curve $C$ and its Jacobian $J_{C}$. In case that the curve $C$ is smooth projective, with the autoduality
$$
\pic_{J_{C}}^{0}\cong J_{C},
$$ 
we deduce that for any finite abelian covering $C'\to C$ over $k$, there exists a unique finite abelian covering $J'\to J_{C}$ over $k$ which makes the diagram Cartesian
$$
\begin{tikzcd}
C'\arrow[d]\arrow[r]\arrow[dr, phantom, "\square"]&J'\arrow[d]\\
C\arrow[r,"A_{L}"]&J_{C}\,_{.}
\end{tikzcd}
$$   
This is the {geometric class field theory} of Rosenlicht and Lang (cf. \cite{serre class}). 
We would like to reformulate and generalize it with the autoduality of the compactified Jacobian.  

Let $S$ be an integral noetherian $k$-scheme, let $C$ be a flat projective curve over $S$ with  geometrically irreducible fibers having only planar singularities. Suppose that there exists a section $S\to C$ with image lying in the smooth locus of $C$. Let $L$ be a line bundle over $C$ of relative degree 1. 
For an $S$-morphism $\varphi: C'\to C$ which is a finite abelian covering with Galois group $A$ of order coprime to ${\rm char}(k)$, the group $A$ acts on $\varphi_{*}\co_{C'}$ and induces an eigen-sheaf decomposition 
$$
\varphi_{*}\co_{C'}=\bigoplus_{\chi\in \widehat{A}}\co_{{C'},\,\chi},
$$
where $\widehat{A}$ is the set of irreducible characters of $A$, and the sheaves 
$$
\co_{{C'},\,\chi}=\Big\{f\in \co_{C'}\,\big|\, f(a^{-1}x)=\chi(a)f(x) \text{ for any }a\in A, \,x\in C'   \Big\}
$$ 
are invertible $\co_{C}$-modules, we call them the eigensheaves of character $\chi$. 
Same result holds for the finite abelian coverings of $\overline{P}_{C/S}=\overline{\pic}{}^{\,0}_{C/S}$.

\begin{thm}\label{curve to jacobian}

Under the above setting, let $A$ be a finite abelian group of order coprime to ${\rm char}(k)$, then there exists a bijection between the set of finite abelian coverings $\varphi:C'\to C$ and the set of finite abelian coverings $\Phi: \overline{P}'\to \overline{P}_{C/S}$, both with Galois group $A$, characterized by the property 
$$
\co_{\overline{P}',\,{\chi}}=\beta\,(\co_{C',\,\chi}),\quad \forall\,\chi\in \widehat{A}.
$$ 
Here we are identifying $\co_{C',\,\chi}$ and $\co_{\overline{P}',\,{\chi}}$ with the points on $\pic_{C/S}$ and $\pic_{\overline{P}_{C/S}/S}$ corresponding to them. 
Moreover, we have the Cartesian diagram
$$
\begin{tikzcd}
C'\arrow[d, "\varphi"']\arrow[r]\arrow[dr, phantom, "\square"]&\overline{P}'\arrow[d, "\Phi"]
\\
C\arrow[r,"A_{L}"]&\overline{P}_{C/S}.
\end{tikzcd}
$$ 

\end{thm}

\begin{proof}

To set up the bijection between the finite abelian coverings of $C$ and $\overline{P}_{C/S}$, it is enough to show that the structures of algebra on 
\begin{equation}\label{eigen decomp}
\varphi_{*}\co_{C'}=\bigoplus_{\chi\in \widehat{A}}\co_{{C'},\,\chi}
\quad
\text{and}
\quad
\Phi_{*}\co_{\overline{P}'}=\bigoplus_{\chi\in \widehat{A}}\co_{\overline{P}',\,{\chi}}
\end{equation}
are transferred into each other under the morphism $\beta$. Note that the structure of algebra on $\varphi_{*}\co_{C'}$ together with the $A$-action is encoded in the collection of isomorphisms
\begin{equation}\label{prod structure}
\co_{{C'},\,\chi_{1}}\otimes \co_{{C'},\,\chi_{2}}\cong \co_{{C'},\,\chi_{1}+\chi_{2}},\quad \forall\,\chi_{1},\chi_{2}\in \widehat{A},
\end{equation} 
which is equivalent to the additions of the corresponding points in $\pic^{0}_{{C/S}}$. Similar results for $\Phi_{*}\co_{\overline{P}'}$. 
By theorem \ref{autoduality main}, 
{\small $
\beta:\pic^{0}_{{C/S}}\to \pic^{0}_{\overline{P}_{C/S}/S}
$} 
is an isomorphism of group schemes, hence such collections of isomorphisms are transferred into each other under $\beta$, and this sets up the bijection between the finite abelian coverings of $C$ and $\overline{P}_{C/S}$. 

For the second assertion, as $A_{L}^{*}$ is inverse to $\beta$, the pull-back via $A_{L}$ defines a morphism of algebras 
$$
A_{L}^{*}: \Phi_{*}\co_{\overline{P}'}=\bigoplus_{\chi\in \widehat{A}}\co_{\overline{P}',\,{\chi}}\to \varphi_{*}\co_{C'}=\bigoplus_{\chi\in \widehat{A}}\co_{{C'},\,\chi}.
$$
This defines the morphism $C'\to \overline{P}'$, and we get the Cartesian diagram in the theorem.

\end{proof}

For the finite abelian covering $\varphi: C'\to C$, it is clear that the eigensheaf  decomposition (\ref{eigen decomp}) and the multiplication law (\ref{prod structure}) determine completely the abelian covering. The eigensheaves correspond to a formal sum $\ct_{C'/C}=\sum n_{i}\ct_{i}$ of torsion subgroups of {\small$\pic^{0}_{C/S}$}, here the multiplicities $n_{i}$ are related to the multiplicities of the eigensheaves in the decomposition (\ref{eigen decomp}).  
Same result holds for the covering $\Phi:\overline{P}'\to \overline{P}_{C/S}$, and it is determined by the subgroup $\beta(\ct_{C'/C})$ of {\small $\pic^{0}_{\overline{P}_{C/S}/S}$}.

We proceed to analyze the irreducible components of the geometric fibers of the families $C'\to S$ and $\overline{P}'\to S$. For this, it is enough to treat the case $S=\spec(\bar{k})$. 
Let $\phi:\widetilde{C}\to C$ be the normalization of $C$. The pull back of invertible sheaves along $\phi$ defines a morphism
$$
\phi^{*}: J_{C}\to J_{\widetilde{C}}.
$$
For a formal sum $\ct=\sum n_{i}\tau_{i}$ of torsion points of $J_{C}$, we denote $\ct^{\red}=\sum\tau_{i}$ and $\phi^{*}(\ct)=\sum n_{i}\phi^{*}(\tau_{i})$. Same convention for the formal sum of torsion points of $J_{\overline{P}_{C}}$.

\begin{thm}\label{irreducible components}

Suppose that $S=\spec(\bar{k})$ and that $\ct_{C'/C}$ is a reduced sub group scheme of $J_{C}$, 
then the irreducible components of $C'$ and that of $\overline{P}'$ are both parametrized by $\Ker(\phi^{*}|_{\ct_{C'/C}})$, which is of cardinal $|A|\,\big/\,|\phi^{*}(\ct_{C'/C})^{\red}|$.

\end{thm}

\begin{proof}

Let $\widetilde{C'}$ be the normalization of $C'$. The finite abelian covering $\varphi: C'\to C$ induces a finite abelian covering $\tilde{\varphi}:\widetilde{C'}\to \widetilde{C}$ with the same Galois group $A$. It is clear that $C'$ and $\widetilde{C'}$ have the same number of irreducible components. 
By construction, the eigen-sheaves of $\tilde{\varphi}_{*}\co_{\widetilde{C'}}$ under the action of $A$ are just ${\phi}^{*}(\co_{C',\,\chi}), \chi\in \widehat{A}$. 
This implies that $\ct_{\widetilde{C'}/\widetilde{C}}$ equals $\phi^{*}(\ct_{C'/C})^{\red}$,  with a multiplicity of $\big|\Ker(\phi^{*}|_{\ct_{C'/C}})\big|$. 
Hence $\widetilde{C'}$ is a trivial $\Ker(\phi^{*}|_{\ct_{C'/C}})$-covering of the finite abelian covering of $\widetilde{C}$ which is associated to the torsion subgroup $\phi^{*}(\ct_{C'/C})^{\red}$. 
As $\widetilde{C}$ is projective smooth, this finite abelian covering is irreducible.
Hence the irreducible components of $\widetilde{C'}$ are naturally parametrized by $\Ker(\phi^{*}|_{\ct_{C'/C}})$, and so is $C'$.

For the finite abelian covering $\Phi:\overline{P}'\to \overline{P}_{C}$, we restrict it to the preimage of $J_{C}$ and get a finite abelian covering $\Phi':J'\to J_{C}$. 
As the singularities of $C$ are planar, according to Altman-Iarrobino-Kleiman \cite{aik} and Rego \cite{rego}, $\overline{P}_{C}$ is geometrically integral and $J_{C}$ is dense open in $\overline{P}_{C}$. 
Hence by restriction we get a bijection between the set of irreducible components of $\overline{P}'$ and that of $J'$. 
By construction, the eigen-sheaves of $\Phi'_{*}\co_{J'}$ under the action of $A$ are just the restriction of $\co_{\overline{P}'_{C},\,\chi}$ to $J_{C}$, $\chi\in \widehat{A}$.
Let $\iota:J_{C}\to \overline{P}_{C}$ be the natural inclusion, it induces the morphism
$
\iota^{*}:\pic_{\overline{P}_{C}/k} \to \pic_{J_{C}/k},
$
then 
\begin{equation}\label{torsion 1}
\ct_{J'/J_{C}}=\iota^{*}\ct_{\overline{P}'/ \overline{P}_{C}}.
\end{equation}
To understand $\pic_{J_{C}/k}$, recall that $J_{C}$ admits the d\'evissage of Chevalley
$$
1\to R\to J_{C}\xrightarrow{\phi^{*}}J_{\widetilde{C}}\to 1,
$$
where $R$ is the maximal connected affine sub group scheme of $J_{C}$. A precise form of the d\'evissage can be found in (\ref{chevalley cj curve}).  
In particular, it is shown there that $R$ is commutative, being the product of a split  multiplicative subgroup and a commutative unipotent subgroup. 
Hence $
\phi^{*}: J_{C}\to J_{\widetilde{C}}
$ induces an isomorphism
$$
\phi_{\star}: \pic_{J_{\widetilde{C}}/k}\cong \pic_{J_{C}/k}. 
$$
We denote by $\iota^{\star}$ the composite 
$$
\iota^{\star}:\pic_{\overline{P}_{C}/k}\xrightarrow{\iota^{*}} \pic_{J_{C}/k}\stackrel{\phi_{\star}^{-1}}{\longrightarrow} \pic_{J_{\widetilde{C}}/k}.
$$
By construction, we have the commutative diagram
$$
\begin{tikzcd}
{\rm Pic}^{0}_{\overline{P}_{C}/k}\arrow[r, "\iota^{\star}"]\arrow[d,"A_{L}^{*}"'] &{\rm Pic}^{0}_{J_{\widetilde{C}}/k}\arrow[d, "A_{\phi^{*}L}^{*}"]
\\
J_{C}\arrow[r, "\phi^{*}"]& J_{\widetilde{C}}.
\end{tikzcd}
$$
Combining equation (\ref{torsion 1}), we obtain that
\begin{align*}
A_{\phi^{*}L}^{*}\circ \phi_{\star}^{-1}(\ct_{J'/J_{C}})&=A_{\phi^{*}L}^{*}\circ \iota^{\star}(\ct_{\overline{P}'/ \overline{P}_{C}})=\phi^{*}\circ A_{L}^{*}(\ct_{\overline{P}'/ \overline{P}_{C}})
=\phi^{*}(\ct_{C'/C})
\\
&=\big|\Ker(\phi^{*}|_{\ct_{C'/C}})\big|\cdot \phi^{*}(\ct_{C'/C})^{\red}.
\end{align*}
As $A_{\phi^{*}L}^{*}$ and $\phi_{\star}$ are isomorphisms, this means that $\Phi':J'\to J_{C}$ is a trivial $\Ker(\phi^{*}|_{\ct_{C'/C}})$-covering of a finite abelian covering of $J_{C}$, which restricts via the morphism $J_{C}\to J_{\widetilde{C}}$ to the finite abelian covering of $J_{\widetilde{C}}$ associated to the torsion subgroup $\phi^{*}(\ct_{C'/C})^{\red}$. As $\widetilde{C}$ is projective smooth, this last finite abelian covering is irreducible. Hence the irreducible components of $J'$ are naturally parametrized by $\Ker(\phi^{*}|_{\ct_{C'/C}})$, and so is $\overline{P}'$.

\end{proof}

\subsection{Applications to the spectral curve and its compactified Jacobian}
\label{finite abelian curve section}


Coming back to our spectral curve $C_{\gamma}$.
We apply the theory of the previous section to the deformation $\pi: \cc\to \cb$ of $C_{\gamma}$ and its relative compactified Jacobian $f:\overline{\cp}\to \cb$.

\paragraph{(1) Finite abelian coverings of deformations of the spectral curves}

Let $\gamma=\diag(\gamma_{i})_{i=1}^{r}$ with $\gamma_{i}\in \mathfrak{gl}_{d_{i}}(\co)$ such that the characteristic polynomial of $\gamma_{i}$ is irreducible over $F$ for all $i$. 
Let $A=\co[\gamma]$ and $\widetilde{A}$ the normalization of $A$ in $E=\mathrm{Frac}(A)$.
By theorem \ref{curve to jacobian} and its proof, a finite abelian covering $C'\to C_{\gamma}$ is  uniquely determined by a finite formal sum $\ct_{C'/C_{\gamma}}$ of torsion subgroups of $J_{C_{\gamma}}$. By the equation (\ref{describe jac}), we have
$J_{C_\gamma}\cong \widetilde{A}^{\times}/A^{\times}$. It is product of a unipotent subgroup and a split maximal torus $\rS_{\gamma}=\bbg_{m}^{r}/\bbg_{m}$. 
Hence the torsion subgroup of $J_{C_\gamma}$ coincides with that of $\rS_{\gamma}$. 
For $n\in \bn, (n,p)=1$, we denote by $\psi_{n}: C_{n}\to C_{\gamma}$ the finite abelian covering corresponding to the torsion subgroup $\rS_{\gamma}[n]=\mu_{n}^{r}/\mu_{n}$.
The Galois group of the covering is naturally identified with the dual 
$$
\Hom\big(\rS_{\gamma}[n], {\overline{\mathbf{Q}}}_{\ell}^{\times}\big)=
\Hom\big(\,{\textstyle\frac{1}{n}}X_{*}(\rS_{\gamma})/X_{*}(\rS_{\gamma}), \bq/\bz\big)
\cong \Lambda^{0}/n.
$$
Moreover, as $C_{\gamma}$ is rational, by theorem \ref{irreducible components}, the irreducible components of $C_{n}$ are naturally parametrized by $\rS_{\gamma}[n]$.

\begin{prop}\label{killed by finite covering}

$H^{1}(C_{\gamma}, \bz/n)$ is killed under the pull-back $\psi_{n}^{*}:H^{1}(C_{\gamma},\bz/n)\to H^{1}(C_{n},\bz/n)$. 

\end{prop}

\begin{proof}

We have canonical isomorphism
\begin{equation*}
H^{1}(C_\gamma,\bz/n)=J_{C_\gamma}[n]\otimes \mu_{n}^{-1}=\rS_{\gamma}[n]\otimes \mu_{n}^{-1}.
\end{equation*} 
In terms of invertible sheaves, the torsion points in $\rS_{\gamma}[n]$ correspond bijectively to the eigen-sheaves in the decomposition
$$
\psi_{n,*}\co_{C_{n}}=\bigoplus_{\chi\in \widehat{\Lambda^{0}/n}} \co_{C_{n},\, \chi}.
$$
We need to show that their pull-back under $\psi_{n}$ are all trivial.  
As $\psi_{n}$ is \'etale Galois, we have 
$$
\psi_{n}^{*} \psi_{n,*}\co_{C_{n}}=\co_{C_{n}}\otimes_{\co_{C_{\gamma}}}\co_{C_{n}}\cong \co_{C_{n}}^{\oplus \deg(\psi_{n})},
$$ 
whence 
$$
\psi_{n}^{*}\co_{C_{n},\,\chi}=\co_{C_{n}},\quad \forall\, \chi\in \widehat{\Lambda^{0}/n}.
$$

\end{proof}

To extend the finite abelian covering $\psi_{n}:C_{n}\to C_{\gamma}$ across the family $\pi:\cc\to \cb$, we need to extend the finite subgroup scheme $\ct_{C_{n}/C_{\gamma}}=J_{C_{\gamma}}[n]$ over $\cb$. 
As $n$ is coprime to $p$, the morphism $\pic_{{\cc}/{\cb}}[n]\to {\cb}$ is quasi-finite and  \'etale (cf. \cite{blr}, \S7.3, lemma 2).  
We have canonical isomorphism
\begin{equation}
\label{h1 as jac tor}
R^{1}\pi_{*}\bz/n=\pic_{{\cc}/{\cb}}[n]\otimes \mu_{n}^{-1},
\end{equation}
hence the obstruction for the quasi-finite group scheme $\pic_{{\cc}/{\cb}}[n]$ over ${\cb}$ to be finite \'etale lies in the monodromy action
$$
\rho^{(n)}:\pi_{1}(\cb-\Delta_{\pi},\, \bar{\xi})\to \aut\big(H^{1}(\cc_{\bar{\xi}}, \bz/n)\big), 
$$
where $\bar{\xi}$ is a geometric generic point of $\cb$.

\begin{lem}\label{support curve}

For the family $\pi:\cc\to \cb$, we have
$$
R\pi_{*}\ql=\bigoplus_{i=0}^{2} j_{!*}(R^{i}\pi^{\rm sm}_{*}\ql)[-i].
$$

\end{lem}

\begin{proof}

Since $\cc$ is smooth over $k$ and $\pi$ is proper, $R\pi_{*}\ql$ is pure by Weil-II \cite{weil2}, and so
$$
R\pi_{*}\ql=\bigoplus_{i=0}^{2} {}^{p}\ch^{i}(R\pi_{*}\ql)[-i]
$$
by the decomposition theorem. 
To determine the support of each ${}^{p}\ch^{i}(R\pi_{*}\ql)$, we apply the inequality of Goresky-MacPherson (\cite{ngo}, th\'eor\`eme 7.3.1). 
According to the theorem, as $\pi$ is proper of pure relative dimension $1$, any extra support of ${}^{p}\ch^{i}(R\pi_{*}\ql)$ will be of codimension $1$, and over its generic point the sheaf $R^{2}\pi_{*}\ql$ will admit a non-trivial local system as a direct summand. But the fibers of $\pi$ are all geometrically irreducible, hence there can not be any extra support.

\end{proof}

Let $V=(R^{1}\pi^{\rm sm}_{*}\ql)_{\bar{\xi}}$. For any geometric point $x$ of $\cb$, let $\bar{\xi}_{x}$ be a geometric generic point of $\cb_{\{x\}}$.
We fix an embedding $\bar{\xi}\hookrightarrow \bar{\xi}_{x}$, which allow us to identify $\pi_{1}(\cb_{\{x\}}-\Delta,\,\bar{\xi}_{x})$ as a subgroup of $\pi_{1}(\cb-\Delta,\,\bar{\xi})$ and to identify $(R^{1}\pi^{\rm sm}_{*}\ql)_{\bar{\xi}_{x}}$ with $V$. Geometrically, this can be done by fixing a smooth path connecting $\bar{\xi}$ and $\bar{\xi}_{x}$. We have then
\begin{equation*}
H^{1}(\cc_{x},\ql)=(R^{1}\pi_{*}\ql)_{x}=(j_{!*}R^{1}\pi_{*}^{\sm}\ql)_{x}=V^{\pi_{1}(\cb_{\{x\}}-\Delta,\,\bar{\xi}_{x})}.
\end{equation*}
The equation actually holds for $\bz_{\ell}$-coefficient, as both $H^{1}(C_{\gamma}, \bz_{\ell})$ and $H^{1}(\cc_{\bar{\xi}}, \bz_{\ell})$ are free $\bz_{\ell}$-modules. Let $\mathbb{V}_{\ell}=H^{1}(\cc_{\bar{\xi}}, \bz_{\ell})$, then
\begin{equation}\label{h1 curve as invariant}
H^{1}(\cc_{x},\zl)=\mathbb{V}_{\ell}^{\pi_{1}(\cb_{\{x\}}-\Delta,\,\bar{\xi}_{x})}.
\end{equation}
Let $\widehat{\bz}^{(p)}=\prod_{l\neq p \text{ prime}}\bz_{l}$ and ${\bbv}=\prod_{l\neq p \text{ prime}}\bbv_{l}$, then 
\begin{equation}\label{h1 curve full}
H^{1}(\cc_{x},\widehat{\bz}^{(p)})=\bbv^{\pi_{1}(\cb_{\{x\}}-\Delta,\,\bar{\xi}_{x})},
\end{equation}
and so
\begin{equation}\label{h1 curve mod n}
H^{1}(\cc_{x},\bz/n)=({\bbv}/n)^{\pi_{1}(\cb_{\{x\}}-\Delta,\,\bar{\xi}_{x})}=H^{1}\big(\cc_{\bar{\xi}}, \bz/n\big)^{\pi_{1}(\cb_{\{x\}}-\Delta,\,\bar{\xi}_{x})}.
\end{equation}

Consider the monodromy action $\rho^{(n)}$ and its local analogue
$$
\rho^{(n)}_{x}: \pi_{1}(\cb_{\{x\}}-\Delta,\,\bar{\xi}_{x})\to \aut\big(H^{1}(\cc_{\bar{\xi}_{x}}, \bz/n)\big).
$$
As ${\rm char}(k)$ is assumed to be large enough, both actions of the fundamental groups factor through their tame quotients. 
Let $G_{\rho}^{(n)}$ and $G_{\rho_{x}}^{(n)}$ be their images. 
With the embedding $\bar{\xi}\hookrightarrow \bar{\xi}_{x}$, $G_{\rho_{x}}^{(n)}$  can be naturally identified as a subgroup of $G_{\rho}^{(n)}$. 
Let $\Delta_{0}$ be the irreducible component of $\Delta$ containing $0$, and let $\Delta^{(0)}$ be the union of the other irreducible components of $\Delta$. Let $\cb^{\circ}=\cb-\Delta^{(0)}$.

\begin{lem}\label{move pi 1}

For any geometric point $x$ of $\Delta-\Delta^{(0)}$, the monodromy group $G_{\rho_{x}}^{(n)}$ is conjugate to a subgroup of $G_{\rho_{0}}^{(n)}$ by an element of $G_{\rho}^{(n)}$.

\end{lem}

\begin{proof}

By assumption, ${\rm char}(k)$ is large enough, hence the local fundamental groups $\pi_{1}^{\rm tm}(\cb_{\{x\}}-\Delta,\,\bar{\xi}_{x})$ are isomorphic to their complex analytic  counter part. For this case, as $\cb$ is smooth, by the local Lefschetz theorem for the fundamental groups, we have surjections
$$\pi_{1}(L_{x}-\Delta,\,\bar{\xi}_{x})\twoheadrightarrow
\pi_{1}(\cb_{\{x\}}-\Delta,\,\bar{\xi}_{x}),
$$ 
where $L_{x}$ is a generic line sufficiently close to $x$ in a sufficiently small open ball $B_{x}(r_{x})$ with center $x$. 
The monodromy group $G_{\rho_{x}}^{(n)}$ is then generated by the image under $\rho_{x}^{(n)}$ of a set of generators for $\pi_{1}(L_{x}-\Delta,\,\bar{\xi}_{x})$.

Recall that the equisingular stratum containing $x$ is smooth (Teissier, \cite{zariski plane branch} appendix), and that the union of the strict $\delta$-strata is Whitney regular along the equisingular strata lying at their frontier \cite{chen root datum}. 
This implies that as we move $x$ along the equisingular stratum containing it, $G_{\rho_{x}}^{(n)}$ is changed to one of its conjugates under the action of $G_{\rho}^{(n)}$.

Note that we can move $x$ along its equisingular stratum to a point $x'$ which is as close as we want to the equisingular stratum containing $0$. Indeed, the curve $\cc_{x}$ can be degenerated until it has a unique singularity with Milnor number $\mu_{\gamma}$, the degenerated curve corresponds then to a point $0'$ on the equisingular stratum containing $0$, because plane curve singularities having the same Milnor number are equisingular \cite{teissier resolution}. 
It is then clear that the set of generators for $\pi_{1}(L_{x'}-\Delta,\,\bar{\xi}_{x'})$ becomes part of a set of generators for $\pi_{1}(L_{0'}-\Delta,\,\bar{\xi}_{0'})$.  
This implies that $G_{\rho_{x}}^{(n)}$ is conjugate to a subgroup of $G_{\rho_{0}}^{(n)}$ by an element of $G_{\rho}^{(n)}$ by what we have explained above. 

\end{proof}

\begin{prop}\label{ext to etale neigh}

There exists a quasi-finite \'etale covering of $\cb^{\circ}$ over which we can extend $J_{C_{\gamma}}[n]$ to a finite \'etale sub group scheme of $\pic_{\cc/\cb}[n]$.

\end{prop}

\begin{proof}

From the isomorphism (\ref{h1 as jac tor}) and (\ref{h1 curve mod n}), we obtain that any element of $J_{\cc_{x}}[n]$ can be extended to a section of $\pic_{{\cc}/{\cb}}[n]$ over an \'etale neighborhood of $x$. 
The above lemma implies that in the process of extending $J_{C_{\gamma}}[n]$ to a finite \'etale sub group scheme of $\pic_{\cc/\cb}[n]$ the ramification appears only at the divisors $\Delta^{(0)}$. So we can go to a quasi-finite \'etale covering of $\cb^{\circ}$ which kills the ramifications to get the extension.

\end{proof}

\begin{defn}\label{define bn}

Let $\cb_{n}\to \cb^{\circ}$ be the minimal quasi-finite \'etale covering over which we can extend $J_{C_{\gamma}}[n]$ to a \emph{constant} finite \'etale sub group scheme, denoted $\ct_{n}$, of $\pic_{\cc/\cb}[n]$. Let $\varpi_{n}$ be the composite $\cb_{n}\to \cb^{\circ}\hookrightarrow \cb$. 
\end{defn}

By construction, $\ct_{n}$ is a constant finite sub group scheme of $\pic_{\cc\times_{\cb}\cb_{n}/\cb_{n}}[n]$.
By theorem \ref{curve to jacobian} and its proof, it corresponds to a $\Lambda^{0}/n$-covering
$$
\Psi_{n}:\cc_{n}\to \cc\times_{\cb}\cb_{n}.
$$
Let $0_{n}$ be the geometric point of $\cb_{n}$ lying over $0\in \cb$ such that the covering $\Psi_{n}$ at this point is exactly $\psi_{n}:C_{n}\to C_{\gamma}$. Consider the family of curves
$$
\pi_{n}: \cc_{n}\to \cc\times_{\cb}\cb_{n}\to \cb_{n}.
$$
The construction can be summarized in the commutative diagram 
\begin{equation}\label{cd curve}
\begin{tikzcd}
{\cc}_{n}\arrow[dr, "{\pi}_{n}"']\arrow[r, "\Psi_{n}"]&{\cc}\times_{\cb}\cb_{n}\arrow[d]\arrow[r]\arrow[dr, phantom, "\square"]&{\cc}\arrow[d, "{\pi}"]
\\
&{\cb}_{n}\arrow[r, "\varpi_{n}"]&{\cb}.
\end{tikzcd}
\end{equation}

\begin{prop}\label{finite cover smooth}

Both ${\cc}_{n}$ and ${\cb}_{n}$ are smooth over $k$.

\end{prop} 

\begin{proof}

$\cb_{n}$ is smooth over $k$ because it is \'etale over the open subscheme $\cb^{\circ}$ of $\cb$, which is smooth over $k$.
For the smoothness of $\cc_{n}$, we have a tower of \'etale morphisms
$$
\cc_{n}\xrightarrow{\Psi_{n}} \cc\times_{\cb}\cb_{n}\to \cc.
$$
As $\cc$ is smooth over $k$, so will be $\cc_{n}$.

\end{proof}

The commutative diagram (\ref{cd curve}) induces the morphism $\Psi_{n}^{*}:\varpi_{n}^{*}R^{1}\pi_{*}\bz/n \to R^{1}\pi_{n,*}\bz/n$. 
Let $\crr_{n}$ be the smooth subsheaf of $\varpi_{n}^{*}R^{1}\pi_{*}\bz/n$ which corresponds to $\ct_{n}\otimes \mu_{n}^{-1}$ under the isomorphism
\begin{equation}
\label{h1 as jac tor rel}
\varpi_{n}^{*}R^{1}\pi_{*}\bz/n\cong \pic_{\cc\times_{\cb}\cb_{n}/\cb_{n}}[n]\otimes \mu_{n}^{-1}.
\end{equation}

\begin{prop}\label{killed by finite covering curve}

The morphism $\Psi_{n}^{*}:\varpi_{n}^{*}R^{1}\pi_{*}\bz/n \to R^{1}\pi_{n,*}\bz/n$ has kernel $\crr_{n}$. 

\end{prop}

\begin{proof}

In terms of Picard, we need to show that the morphism
$$
\Psi_{n}^{*}: \pic_{\cc\times_{\cb}\cb_{n}/\cb_{n}}[n]\to \pic_{\cc_{n}/\cb_{n}}[n]
$$
has kernel $\ct_{n}$.
With the same reasoning as proposition \ref{killed by finite covering}, we have  $\ct_{n}\subset \Ker(\Psi_{n}^{*})$. On the other hand, as $\Psi_{n}$ is finite Galois, for any invertible sheaf $\cl$ on $\cc\times_{\cb}\cb_{n}$, the push-forward $\Psi_{n,*}\Psi_{n}^{*}\cl$ contains $\cl$ as a direct summand.
In case that $\cl\in \Ker(\Psi_{n}^{*})$, it appears as a direct summand of $\Psi_{n,*}\co_{\cc_{n}}$, hence $\cl\in \ct_{n}$.

\end{proof}

By construction, for $m,n$ coprime to $p$ and $n|m$, there exists an \'etale morphism $\varpi_{m,n}:\cb_{m}\to \cb_{n}$, such that the isogeny $[m/n]:\ct_{m}\to \ct_{m}$ factors through $\ct_{n}\times _{\cb_{n}}\cb_{m}$. 
Let $\widetilde{\varpi}_{m,n}: \ct_{m}\to \ct_{n}\times _{\cb_{n}}\cb_{m}$ be the resulting isogeny. It corresponds to a morphism of sheaves $\crr_{m}\to \varpi_{m,n}^{*}\crr_{n}$.
Let 
$$
\cb_{\infty}=\varprojlim_{(n,p)=1}\cb_{n},
\quad 
\widehat{\ct}=\varprojlim_{(n,p)=1} \ct_{n}
\quad \text{and}\quad
\widehat{\crr}=\varprojlim_{(n,p)=1} \crr_{n}.
$$
Then $\widehat{\crr}$ is a product of smooth $l$-adic sheaf over $\cb_{\infty}$ with $l$ running through all the prime numbers different from $p$. 

\paragraph{(2) Finite abelian coverings of deformations of the compactified Jacobian}

Consider the relative compactified Jacobian
$
f: \overline{\cp}\to \cb
$
of the family $\pi:\cc\to \cb$. Via the morphism $\varpi_{n}:{\cb_{n}}\to \cb$, we get the family
$$
\overline{\cp}\times_{\cb}\cb_{n}\to \cb_{n}.
$$
By theorem \ref{curve to jacobian}, the constant finite \'etale sub group scheme $\ct_{n}$ of $\pic_{\cc\times_{\cb}\cb_{n}/\cb_{n}}[n]$ determines a $\Lambda^{0}/n$-covering of $\cb_{n}$-schemes 
$$
\begin{tikzcd}[column sep={4.5em,between origins},row sep=1.5em]
\overline{\cp}_{n}\arrow[rr, "\Phi_{n}"]\arrow[rd, "f_{n}"']& &\overline{\cp}\times_{\cb}\cb_{n}\arrow[ld]\\
&\cb_{n}&
\end{tikzcd}.
$$
The family $f_{n}:\overline{\cp}_{n}\to \cb_{n}$ is projective flat of pure relative dimension $\delta_{\gamma}$. Let $\overline{P}_{n}$ be the fiber of $f_{n}$ at $0_{n}\in \cb_{n}$, then $\Phi_{n}$ specializes to the  $\Lambda^{0}/n$-covering $\overline{P}_{n}\to \overline{P}_{C_{\gamma}}$ at $0_{n}$.
In terms of the affine Springer fibers, it corresponds to the natural projection  $n\Lambda^{0}\backslash \xx_{\gamma}^{0}\to \Lambda^{0}\backslash \xx_{\gamma}^{0}$.
The construction can be summarized in a commutative diagram with the right one being Cartesian
\begin{equation}\label{cd picard}
\begin{tikzcd}
\overline{\cp}_{n}\arrow[dr, "f_{n}"']\arrow[r,"\Phi_{n}"]& \overline{\cp}\times_{\cb}\cb_{n}\arrow[d]\arrow[r]\arrow[dr, phantom, "\square"] & \overline{\cp}\arrow[d, "f"]
\\
&\cb_{n} \arrow[r, "\varpi_{n}"]&\cb,
\end{tikzcd}
\end{equation}
which is parallel to the commutative diagram (\ref{cd curve}). 
The commutative diagram induces the morphism 
$$
\Phi_{n}^{*}:\varpi_{n}^{*}R^{1}f_{*}\bz/n \to R^{1}f_{n,*}\bz/n.
$$ 
With the isomorphism 
$
R^{1}\pi_{*}\bz/n \cong R^{1}f_{*}\bz/n
$ 
and similarly for $\pi_{n}$ and $f_{n}$,
$\Phi_{n}^{*}$ is equivalent to the morphism $\Psi_{n}^{*}:\varpi_{n}^{*}R^{1}\pi_{*}\bz/n \to R^{1}\pi_{n,*}\bz/n$, and proposition \ref{killed by finite covering curve} can be reformulated as

\begin{prop}\label{killed by finite covering picard}

The morphism $\Phi_{n}^{*}:\varpi_{n}^{*}R^{1}f_{*}\bz/n \to R^{1}f_{n,*}\bz/n$ has kernel $\crr_{n}$. 

\end{prop}

We can restrict the $\Lambda^{0}/n$-covering $\Phi_{n}:\overline{\cp}_{n}\to \overline{\cp}\times_{\cb}\cb_{n}$ to the preimage of $\cp\times_{\cb}\cb_{n}$, and get a $\Lambda^{0}/n$-covering 
$
\Phi_{n}^{\circ}:\cp_{n}\to \cp\times_{\cb}\cb_{n}
$
and hence a fiberwise dense open subfamily $f_{n}^{\circ}:\cp_{n}\to \cb_{n}$ of $f_{n}$.
Then the action of $\cp\times_{\cb}\cb_{n}$ on $\overline{\cp}\times_{\cb}\cb_{n}$ can be lifted to an action of $\cp_{n}$ on $\overline{\cp}_{n}$, and they form a weak abelian fibration over $\cb_{n}$ in the sense of Ng\^o.


\begin{prop}\label{smooth picard total}

The total space $\overline{\cp}_{n}$ is smooth over $k$.

\end{prop}

\begin{proof}

As $\cb_{n}\to \cb^{\circ}\subset \cb$ is \'etale, we get a tower of \'etale morphism
$$
\overline{\cp}_{n}\xrightarrow{\Phi_{n}} \overline{\cp}\times_{\cb}\cb_{n}\to \overline{\cp}.
$$
The smoothness of $\overline{\cp}_{n}$ over $k$ then follows from that of $\overline{\cp}$.

\end{proof}

\subsection{Analysis of the irreducible components of the geometric fibers}

We proceed to analyse the irreducible components of the geometric fibers of $\pi_{n}:\cc_{n}\to \cb_{n}$ and $f_{n}:\overline{\cp}_{n}\to \cb_{n}$.
Recall that we have the canonical isomorphism
\begin{equation}\label{h1 curve copy}
H^{1}(C_\gamma,\bz/n)=J_{C_\gamma}[n]\otimes \mu_{n}^{-1}=\rS_{\gamma}[n]\otimes \mu_{n}^{-1}={\textstyle\frac{1}{n}}X_{*}(\rS_{\gamma})/X_{*}(\rS_{\gamma}),
\end{equation} 
and the isomorphisms (\ref{h1 curve as invariant})-(\ref{h1 curve mod n}) at $x=0$,
where $\rS_{\gamma}=\bbg_{m}^{r}/\bbg_{m}$ is the maximal torus of $J_{C_{\gamma}}=\widetilde{A}^{\times}/A^{\times}$.
For $i=1,\cdots, r$, let $\ep_{i}\in X_{*}(\rS_{\gamma})$ be the cocharacter sending $\bbg_{m}$ identically to the $i$-th factor of $\rS_{\gamma}$.

\begin{defn}

Let $c_{i}\in {\bbv}^{\pi_{1}\big(\cb_{\{0\}}-\Delta,\,\bar{\xi}_{0}\big)}, i=1,\cdots,r,$ be the element such that $c_{i}\mod n$ corresponds to $\frac{1}{n}\ep_{i}$ under the isomorphism (\ref{h1 curve mod n}) for $x=0$ and (\ref{h1 curve copy}) for all $n\in \bn, (n,p)=1$.

\end{defn}

\begin{prop}\label{h1 inv generator}

The invariant subgroup $\bbv^{\pi_{1}(\cb_{\{0\}}-\Delta, \,\bar{\xi}_{0})}$ is generated by the classes $c_{i}, i=1,\cdots,r$.

\end{prop}

\begin{proof}

It is clear that $\frac{1}{n}X_{*}(\rS_{\gamma})/X_{*}(\rS_{\gamma})$ is generated by the cocharacters $\frac{1}{n}\ep_{i}, i=1,\cdots,r$. By definition, they correspond to the classes $c_{i} \mod n$. Hence these classes generate $(\bbv/n)^{\pi_{1}(\cb_{\{0\}}-\Delta, \,\bar{\xi}_{0})}$. Pass to the projective limit, we get the proposition. 

\end{proof}

We determine explicitly the elements $c_{i}$. We need the description of $\pi_{1}\big(\cb_{\{0\}}-\Delta,\,\bar{\xi}_{0}\big)$ in \S\ref{monodromy curve}, which carries over to the case that ${\rm char}(k)$ is large enough by remark \ref{ac large char}.  
Henceforth we switch to the complex analytic case and assume that the defining polynomial $P(x,y)=0$ of the singularity of $C_{\gamma}$ at $c$ is of real coefficients, and that it can be decomposed into products of analytically irreducible factors 
\begin{equation}\label{factorize p}
P(x,y)=P_{1}(x,y)\cdots P_{r}(x,y)
\end{equation}
such that the factors are all of real coefficients. For $i=1,\cdots,r$, let $C^{\circ}_{i}$ be the branch of the singularity of $C_{\gamma}$ at $c$ defined by $P_{i}(x,y)=0$. Let $B_{\ep}\subset \bc^{2}$ be a sufficiently small open ball centered at $c$. Let $\widetilde{P}(x,y,t)$ be a \emph{real}  morsification for $P(x,y)$. By construction, $\widetilde{P}(x,y,t)$ induces a real  morsification for each local branch of $C_{\gamma}$ at $c$. In other words, $\widetilde{P}(x,y,t)$ can be factorized as
\begin{equation}\label{factorize p tilde}
\widetilde{P}(x,y,t)=\widetilde{P}_{1}(x,y,t)\cdots \widetilde{P}_{r}(x,y,t),
\end{equation}
such that each factor $\widetilde{P}_{i}(x,y,t)$ is a real morsification for $P_{i}(x,y)$.
For $t\in \bc$ with $|t|$ sufficiently small, let $C_{t}$ be the deformation of $C_{\gamma}$ with the singularity at $c$ deformed into the singularities defined by $\widetilde{P}(x,y,t)=0$. 
Then $C_{t}$ defines a $\delta$-invariant deformation of $C_{\gamma}$.
Let $
C^{\circ}_{t}=C_{t}\cap B_{\ep}$  
and 
$$
C^{\circ}_{i,\,t}=\Big\{(x,y)\in B_{\ep}\,\big| \,\widetilde{P}_{i}(x,y,t)=0\Big\},\quad i=1,\cdots,r.
$$ 
Then $C^{\circ}_{t}$ defines a $\delta$-invariant deformation of the singularity of $C_{\gamma}$ at $c$, and similarly $C^{\circ}_{i,\,t}$ is a $\delta$-invariant deformation of $C_{i}^{\circ}$. Let $\overline{C}{}^{\circ}_{t}=C^{\circ}_{t}\cap \br^{2}$ and let $\overline{C}{}^{\circ}_{i,\,t}=C^{\circ}_{i,\,t}\cap \br^{2}$, then $\overline{C}{}^{\circ}_{i,\,t}$ are the irreducible components of $\overline{C}{}^{\circ}_{t}$. The curve $\overline{C}{}^{\circ}_{t}$ cuts $D_{\ep}=B_{\ep}\cap \br^{2}$ as union of disjoint connected components. The connected components which are disjoint from the boundary of $D_{\ep}$ are called regions, and they can be labelled $\oplus$ and $\ominus$ according to the sign of $\widetilde{P}(x,y,t)$ at the region. 
We use the same notation $R_{i}^{+}, R_{k}^{-}, i=1,\cdots, \mu_{+}, \,k=1,\cdots, \mu_{-},$ as in \S\ref{monodromy curve} for the regions labelled $\oplus$ and $\ominus$, and the same notation $D_{j},j=1,\cdots, \mu_{0}$ for the double points of $\overline{C}{}^{\circ}_{t}$. Note that $\mu_{0}=\delta_{\gamma}$. 
The notations for the vanishing cycles $\alpha_{i}^{+}, \alpha_{j}, \alpha_{k}^{-}$ will be kept as well. Notice that they are elements of the homology $H_{1}(\cc_{\bar{\xi}_{0}},\bz)$ and we need to pass to the cohomology $H^{1}(\cc_{\bar{\xi}_{0}},\bz)$ by the duality between them. 
We use the same notation $\alpha_{i}^{+}, \alpha_{j}, \alpha_{k}^{-}$ for the corresponding element in $H^{1}(\cc_{\bar{\xi}_{0}},\bz)\subset \bbv$. 

\begin{prop}\label{inters number}

The number of intersection points of any two branches $\overline{C}{}^{\circ}_{i,t}$ and $\overline{C}{}^{\circ}_{j,t}$ equals the intersection number $C_{i}^{\circ}\cdot C_{j}^{\circ}$. 

\end{prop}

\begin{proof}

This is due to the fact that the intersection number is invariant under deformation.

\end{proof}

For $i=1,\cdots, r$, let $J_{i}$ be the subset of $\{1,\cdots, \mu_{0}\}$ such that $D_{j}\in \overline{C}{}^{\circ}_{i, t}$ if and only if $j\in J_{i}$. 
Let $J_{i}^{\circ}\subset J_{i}$ be the subset of the index of the double points which lie on the other branches as well.  
Then the double points indexed by elements in $J_{i}\backslash J_{i}^{\circ}$ are exactly the points of self-intersection of the curve $\overline{C}{}^{\circ}_{i,t}$.

\begin{prop}\label{ci express}

For $i=1,\cdots, r$, we have 
$$
c_{i}=\sum_{j\in J_{i}^{\circ}}(\pm) \alpha_{j}.
$$
Moreover, for any $i\neq i'$ and $j\in J_{i}^{\circ}\cap J_{i'}^{\circ}$, the sign of $\alpha_{j}$ in the expression of $c_{i}$ and $c_{i'}$ are opposite. 

\end{prop}

\begin{proof}


For $i=1,\cdots,r$, $a\in k^{\times}$, let $\cl_{i,a}$ be the invertible sheaf on $C_{\gamma}$ corresponding to the element $(\cdots, 1,a,1,\cdots)$ of $\rS_{\gamma}=\bbg_{m}^{r}/\bbg_{m}$, where $a$ sits at the $i$-th position. It can be constructed by gluing $\co_{C_{\gamma}\backslash \{c\}}$ and the $A$-submodule of $E=\prod_{i=1}^{r}E_{i}$ generated by $(\cdots, 1,a,1,\cdots)$  along $\spec(E)=(C_{\gamma}\backslash \{c\})\times_{C_{\gamma}}\spec(A)$.
Suppose that $\cl_{i,a}$ is of order $n$, $(n,p)=1$. As explained in the proof of proposition \ref{ext to etale neigh}, the point of $J_{C_{\gamma}}[n]$ corresponding to $\cl_{i,a}$ extends to a local section of $\pic_{\cc/\cb}[n]$ over an \'etale neighbourhood of $0\in \cb$. The local section corresponds to a flat family of invertible sheaves over the restriction of $\pi:\cc\to \cb$ to this \'etale neighbourhood. 
Let $\cl_{i,a,t}$ be the restriction of this family of invertible sheaves along the $\delta$-invariant deformation of $C_{\gamma}$ defined by $\widetilde{P}(x,y,t)$. Then $\cl_{i,a,t}$ can be obtained by gluing the invertible sheaf $\co_{C_{t}\backslash \{D_{j},\, j\in J_{i}^{\circ}\}}$ and the invertible sheaf $\prod_{j\in J_{i}^{\circ}}a\widehat{\co}_{C_{t},D_{j}}$.  
Let $\phi_{t}:\widetilde{C}_{t}\to C_{t}$ be the normalization. For $j=1,\cdots, \mu_{0}$, let $i_{D_{j}}:D_{j}\hookrightarrow C_{t}$ be the natural inclusion. 
From the exact sequences
$$
1\to \bz/n\to \phi_{t,*}\bz/n\to \bigoplus_{j=1}^{\mu_{0}}i_{D_{j}, *}\bz/n\to 0,
$$
and 
$$
1\to \bbg_{m}\to \phi_{t,*}\bbg_{m}\to \phi_{t,*}\bbg_{m}/\bbg_{m}\to 1, 
$$
we deduce the commutative diagram
$$
\begin{tikzcd}
\bigoplus_{j=1}^{\mu_{0}}\bz/n \arrow [r]\arrow[d]& H^{1}(C_{t}, \bz/n)\arrow[d]\\
\textstyle{\prod_{j=1}^{\mu_{0}}} \big(\widetilde{\co}_{C_{t}, D_{j}}^{\times}/{\co}_{C_{t}, D_{j}}^{\times}\big)[n]\otimes \mu_{n}^{-1}\arrow[r]& {\rm Pic}_{C_{t}/k(t)}[n]\otimes \mu_{n}^{-1},
\end{tikzcd}
$$
with all the arrows being isomorphisms. 
Under the isomorphism (\ref{h1 curve copy}), the element $\frac{1}{n}\ep_{i}$ corresponds to the invertible sheaf $\cl_{i,\zeta_{n}}$ for a primitive $n$-th root of unity $\zeta_{n}$. 
Consider its deformation $\cl_{i,\zeta_{n},t}$, which is an invertible sheave on the curve  $C_{t}$, let $[\cl_{i,\zeta_{n},t}]$ be the torsion point of $\pic_{C_{t}/k(t)}[n]$ corresponding to it.
With the gluing construction of $\cl_{i,\zeta_{n},t}$, we can trace back the commutative diagram, and obtain that the pre-image of $[\cl_{i,\zeta_{n},t}]$ under the right vertical arrow is  
\begin{equation}\label{c coh deform}
c_{i,t}=\sum_{j\in J_{i}^{\circ}}(\pm)\alpha_{j}\mod n\in H^{1}(C_{t}, \bz/n).
\end{equation}
Here we are using the fact that all the singularities of $C_{t}$ are ordinary double points. 
By construction, $c_{i}$ specializes to $c_{i, t}$ under the specialization
$$
H^{1}(C_{\gamma}, \bz/n)\cong H^{1}(\cc\times _{\cb}\cb_{\{0\}}, \bz/n) \to H^{1}(\cc_{t}, \bz/n).
$$
By lemma \ref{support curve}, the specialization map is injective, whence the expression of  $c_{i}$ in the proposition. 

For the second assertion, observe that in the glueing construction of $\cl_{i,a,t}$ and $\cl_{i',a,t}$ their restrictions to the formal neighborhood $\prod_{j\in J_{i}^{\circ}\cap J_{i'}^{\circ}}\spec\big(\widehat{\co}_{C_{t}, D_{j}}\big)$ are inverse to each other if we rigidify 
$$\cl_{i,a,t}|_{C_{t}\backslash \{D_{j},j\in J_{i}^{\circ}\}}=\co_{C_{t}\backslash \{D_{j},j\in J_{i}^{\circ}\}}
\quad\text{and}\quad 
\cl_{i',a,t}|_{C_{t}\backslash \{D_{j},j\in J_{i'}^{\circ}\}}=\co_{C_{t}\backslash \{D_{j},j\in J_{i'}^{\circ}\}}.
$$
Hence in the expression (\ref{c coh deform}) for $c_{i,t}$ and $c_{i',t}$ the elements $\alpha_{j}, j\in J_{i}^{\circ}\cap J_{i'}^{\circ},$ will have opposite signs. This finishes the proof of the second assertion.

\end{proof}

\begin{rem}

With the explicit description of the monodromy action by A'Campo and Gusein-Zade, we can determine the signs of $\alpha_{j}$ in the expression of $c_{i}$ uniquely up to a global $\pm$ sign by the condition that $c_{i}$ is invariant under the monodromy action of $\pi_{1}(\cb_{\{0\}}-\Delta, \bar{\xi})$. 
But this is not essential for our purpose.

\end{rem}

\begin{example}[Continuation of Example \ref{ac gl4}]

We label the irreducible components of $\overline{C}{}^{\circ}_{t_{0}}$ and the vanishing cycles $\alpha_{j}$ as indicated in figure \ref{gl4 example}.
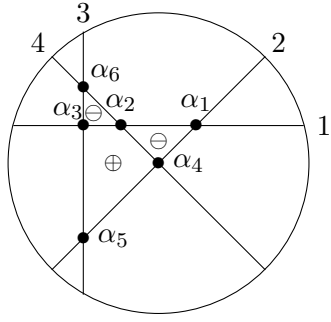
\begin{figure}[h]
\centering
\begin{tikzpicture}
\draw (0,0) circle (2cm);
\draw (1.42, 1.42)--(-1.42,-1.42);
\draw (1.42, -1.42)--(-1.42,1.42);
\draw (1.93, 0.5)--(-1.93, 0.5);
\draw (-1, 1.75)--(-1, -1.75);

\node at (0,0) {$\bullet$};
\node at (0.5,0.5) {$\bullet$};
\node at (-0.5,0.5) {$\bullet$};
\node at (-1,-1) {$\bullet$};
\node at (-1,0.5) {$\bullet$};
\node at (-1,1) {$\bullet$};

\node at (-0.6,0) {\footnotesize $\oplus$};
\node at (0,0.3) {\footnotesize $\ominus$};
\node at (-0.86,0.67) {\footnotesize $\ominus$};

\node at (2.2, 0.5) {$1$};
\node at (1.6, 1.6) {$2$};
\node at (-1, 2) {$3$};
\node at (-1.6, 1.6) {$4$};

\node at (0.5,0.8) {$\alpha_{1}$};
\node at (-0.5,0.8) {$\alpha_{2}$};
\node at (-1.2,0.7) {$\alpha_{3}$};

\node at (0.4,0) {$\alpha_{4}$};
\node at (-0.6,-1) {$\alpha_{5}$};

\node at (-0.7,1.2) {$\alpha_{6}$};

\end{tikzpicture}

\caption{Label of the vanishing cycles}
\label{gl4 example}

\end{figure}

\noindent Then the classes $c_{i}, i=1,\cdots,4$, have expression
\begin{align*}
c_{1}&=\alpha_{1}-\alpha_{2}+\alpha_{3},\\
c_{2}&=-\alpha_{1}+\alpha_{4}-\alpha_{5},\\
c_{3}&=\alpha_{5}-\alpha_{3}+\alpha_{6},\\
c_{4}&=-\alpha_{6}+\alpha_{2}-\alpha_{4},
\end{align*}
and the expression is unique up to a multiplication by $-1$ for all the classes $c_{i}, i=1,\cdots,4$.

\cqfd

\end{example}

\begin{cor-def}

Any element $c\in \bbv^{\pi_{1}(\cb_{\{0\}}-\Delta,\,\bar{\xi}_{0})}$ is a linear combination of the classes $\{c_{i}\}_{i=1}^{r}$, hence has a unique expansion
$
c=\sum_{j=1}^{\mu_{0}} a_{j}\alpha_{j} 
$
for some $a_{j}\in \widehat{\bz}^{(p)}$. 
We call the number of $\alpha_{j}$ with non-zero coefficient in the expression the \emph{height} of $c$. For $I\subset \{1,\cdots,r\}$, let $c_{I}=\sum_{i\in I}c_{i}$ and we denote its height by $h_{I}$.

\end{cor-def}

\begin{prop}\label{non zero coeff c}

Let $c_{I}=\sum_{j=1}^{\mu_{0}} a_{I,j}\alpha_{j}$, then $a_{I,j}\neq 0$ if and only if $j\in \bigsqcup_{i\in I, i'\in \bar{I}}J_{i}^{\circ}\cap J_{i'}^{\circ}$, where $\bar{I}$ is the complement of $I$ in $\{1,\cdots,r\}$. In particular, we have $h_{I}=\sum_{i\in I,\, i'\in \bar{I}}C_{i}^{\circ}\cdot C_{i'}^{\circ}$. 

\end{prop}

\begin{proof}

As $J_{i}^{\circ}=\bigcup_{i'\in \{1,\cdots,r\}\backslash\{i\}} J_{i}^{\circ}\cap J_{i'}^{\circ}$ and the union is disjoint, we can write
\begin{equation}\label{regroup c}
c_{i}=\sum_{i'\in \{1,\cdots,r\}\backslash\{i\}} \sum_{j\in J_{i}^{\circ}\cap J_{i'}^{\circ}}
(-1)^{s_{i,j}}\alpha_{j}, \quad \text{with }s_{i,j}\in \{0,1\}, 
\end{equation}
Notice that the common terms in the expression of $c_{i}$ and $c_{i'}$ are exactly those indexed by $j\in J_{i}^{\circ}\cap J_{i'}^{\circ}$.
By the second assertion of proposition \ref{ci express}, the sign of these common terms are opposite, hence they will cancel out in $c_{i}+c_{i'}$. Moreover, only these common terms are cancelled out in $c_{i}+c_{i'}$. Applying this observation to $c_{I}=\sum_{i\in I}c_{i}$ with the expansion (\ref{regroup c}) for $c_{i}$, we obtain that $a_{I,j}\neq 0$ if and only if 
$$
j\in \bigcup_{i\in I}\Big(J_{i}^{\circ}-\bigcup_{i\neq i'\in I}(J_{i}^{\circ}\cap J_{i'}^{\circ})\Big)
=\bigsqcup_{i\in I, i''\in \bar{I}}J_{i}^{\circ}\cap J_{i''}^{\circ}.
$$  
The height of $c_{I}$ then equals
$$
\sum_{i\in I}\sum_{i''\in \bar{I}}\big|J_{i}^{\circ}\cap J_{i''}^{\circ}\big|=\sum_{i\in I,\, i'\in \bar{I}}C_{i}\cdot C_{i'},
$$
here the equality is due to proposition \ref{inters number}.

\end{proof}

For a partition $\{1,\cdots,r\}=\bigsqcup_{j=1}^{l}I_{j}$ of length $l$, denoted $I_{\bullet}\vdash \{1,\cdots,r\}$ of length $\ell(I_{\bullet})=l$, let $\bbv_{I_{\bullet}}$ be the subgroup of $\bbv$ generated by the elements $c_{I_{1}},\cdots,c_{I_{l}}$. Then any element of $\bbv_{I_{\bullet}}$ has a unique expansion in terms of $\{\alpha_{j}\}_{j=1}^{\mu_{0}}$.
Consider the set $S_{I_{\bullet}}$ of the index $j\in \{1,\cdots, \mu_{0}\}$ such that $\alpha_{j}$ appears with non-zero coefficient in the expansion of at least one element of $\bbv_{I_{\bullet}}$, the cardinal $|S_{I_{\bullet}}|$ will be called the \emph{height} of $\bbv_{I_{\bullet}}$.
In the following, we exclude the trivial partition $I=\{1,\cdots,r\}$ as $c_{I}=0$ in this case.

\begin{prop}\label{key dim 0}

The height of $\bbv_{I_{\bullet}}$ is 
$$
h_{I_{\bullet}}:=\sum_{j\neq j'=1}^{\ell(I_{\bullet})}\sum_{i\in I_{j},\, i'\in I_{j'}} C_{i}^{\circ}\cdot C_{i'}^{\circ}.
$$ 

\end{prop}

\begin{proof}

From the expression (\ref{regroup c}), we see that in the expansion of an element $c=\sum_{i=1}^{r}a_{i}c_{i}$ into $\alpha_{j}$'s the cancelations happen only among the items $a_{i}c_{i}$ with the same coefficient $a_{i}$. 
So $S_{I_{\bullet}}$ is also the set of index $j\in \{1,\cdots, \mu_{0}\}$ such that $\alpha_{j}$ appears with non-zero coefficient in the expansion of at least one $c_{I_{i}}$. 
By proposition \ref{non zero coeff c}, we obtain that
$$
S_{I_{\bullet}}=\bigcup_{j=1}^{\ell(I_{\bullet})}\bigsqcup_{i\in I_{j},\,  i'\in \bar{I}_{j}} J_{i}^{\circ}\cap J_{i'}^{\circ}.
$$
As $I_{\bullet}$ is a partition of $\{1,\cdots, r\}$, we have
$
\bar{I}_{j}=\bigsqcup_{j'=1,j'\neq j}^{\ell(I_{\bullet})} I_{j'}, 
$
and the above set can be written
$$
S_{I_{\bullet}}=\bigcup_{j=1}^{\ell(I_{\bullet})}\bigsqcup_{i\in I_{j}} \bigsqcup_{\substack{j'=1\\j'\neq j}}^{\ell(I_{\bullet})} \bigsqcup_{i'\in I_{j'}} J_{i}^{\circ}\cap J_{i'}^{\circ}
=\bigsqcup_{j\neq j'=1}^{\ell(I_{\bullet})}\bigsqcup_{i\in I_{j}} \bigsqcup_{i'\in I_{j'}} J_{i}^{\circ}\cap J_{i'}^{\circ}. 
$$
Combine it with the equality $|J_{i}^{\circ}\cap J_{i'}^{\circ}|=C_{i}^{\circ}\cdot C_{i'}^{\circ}$ of proposition \ref{inters number}, we get the proposition.

\end{proof}

Recall that the finite abelian covering $\Psi_{n}:\cc_{n}\to \cc\times_{\cb}\cb_{n}$ corresponds to the constant finite sub group scheme $\ct_{n}$ of $\pic_{\cc\times_{\cb}\cb_{n}/\cb_{n}}[n]$. 
Let $s_{n}$ be a geometric point of $\cb_{n}$, let $s$ be its image in $\cb$.
Recall that $J_{\cc_{s}}$ admits a d\'evissage
\begin{equation}\label{chevalley jac exp}
1\to J_{\cc_{s}}^{\rm a}\to J_{\cc_{s}}\xrightarrow{\phi_{s}^{*}} J_{\widetilde{\cc}_{s}}\to 1,
\end{equation} 
where $\phi_{s}:\widetilde{\cc}_{s}\to \cc_{s}$ is the normalization of $\cc_{s}$, and $J_{\cc_{s}}^{\rm a}$ is the maximal affine sub group scheme over $k(s)$ of $J_{\cc_{s}}$. The d\'evissage induces an exact sequence
$$
1\to J_{\cc_{s}}^{\rm a}[n]\to J_{\cc_{s}}[n]\to J_{\widetilde{\cc}_{s}}[n]\to 1.
$$
Note that the fiber $\ct_{n,s_{n}}$ is a sub group scheme of $J_{\cc_{s}}[n]$ by construction. 
As a special case of theorem \ref{irreducible components}, we get

\begin{lem}\label{key irred}

The irreducible components of $\cc_{n,s_{n}}$ can be naturally parametrized by the intersection
$$
\ct_{n,s_{n}}\cdot J_{\cc_{s}}^{\rm a}[n].
$$

\end{lem}

Recall that the finite group schemes $\ct_{n}$ over $\cb_{n}$ form a nice projective system with respect to the morphism $\widetilde{\varpi}_{m,n}:\ct_{m}\to \ct_{n}\times_{\cb_{n}}\cb_{m}$ for $n|m$, and that we have defined its projective limit $\widehat{\ct}$ over $\cb_{\infty}$. 
For a geometric point $s_{\infty}=(\cdots, s_{n},\cdots,s)$ of $\cb_{\infty}$ over $s$, we have then
$$
\ct_{n,s_{n}}\cong \widehat{\ct}_{s_{\infty}}/n.
$$
Similarly, it is well-known that the torsion points $J_{\cc_{s}}[n]$ forms a nice projective system with respect to the isogeny $[m/n]: J_{\cc_{s}}[m] \to  J_{\cc_{s}}[n]$ for $n|m$, and we can define 
$$
\widehat{\rT}(J_{\cc_{s}})
=\varprojlim_{(n,p)=1} J_{\cc_{s}}[n].
$$
Similar construction for $J_{\cc_{s}}^{\rm a}$. By construction, we can identify $\widehat{\ct}_{s_{\infty}}$ as a sub group scheme of $\widehat{\rT}(J_{\cc_{s}})$. It is then clear that 
\begin{equation}\label{tn jn infinite}
\ct_{n,s_{n}}\cdot J_{\cc_{s}}^{\rm a}[n]=
\Big[\widehat{\ct}_{s_{\infty}} \cdot \widehat{\rT}(J^{\rm a}_{\cc_{s}})\Big]\,\Big/\,n.
\end{equation}

The intersection $\widehat{\ct}_{s_{\infty}} \cdot \widehat{\rT}(J^{\rm a}_{\cc_{s}})$ can be reformulated more concretely in terms of the cohomological classes. 
Recall that we have constructed the sheaf $\crr_{n}$ on $\cb_{n}$ which correspond to $\ct_{n}\otimes \mu_{n}^{-1}$ under the isomorphism (\ref{h1 as jac tor rel}), and define their projective limit $\widehat{\crr}$ on $\cb_{\infty}$. 
By construction, the fiber $\widehat{\crr}_{s_{\infty}}$ is a subgroup of $H^{1}(\cc_{s},\widehat{\bz}^{(p)})=\prod_{l\neq p} H^{1}(\cc_{s},\bz_{l})$. 
For each $l$-adic component, we have the canonical weight filtration on $H^{1}(\cc_{s}, \bq_{l})$, which induces a weight filtration on the lattice $H^{1}(\cc_{s}, \bz_{l})$. 
As $l$ varies, we get the weight filtration on $H^{1}(\cc_{s}, \widehat{\bz}^{(p)})$. We denote by $\bbv_{s}^{0}$ its weight $0$ submodule. 
Under the isomorphism
$$
H^{1}(\cc_{s}, \widehat{\bz}^{(p)})\cong \widehat{\rT}(J_{\cc_{s}})\otimes \widehat{\bz}^{(p)}(-1),
$$
it is clear that $\bbv_{s}^{0}$ is sent to $\widehat{\rT}(J^{\rm a}_{\cc_{s}})\otimes \widehat{\bz}^{(p)}(-1)$.
Then we have canonical isomorphism
\begin{equation}\label{cr v0 infinite}
\widehat{\ct}_{s_{\infty}} \cdot \widehat{\rT}(J^{\rm a}_{\cc_{s}})
\cong
\widehat{\crr}_{s_{\infty}} \cap \bbv_{s}^{0}.
\end{equation}
Parallel to (\ref{tn jn infinite}), we have
\begin{equation}\label{rn v0 infinite}
\ct_{n,s_{n}}\cdot J_{\cc_{s}}^{\rm a}[n]
\cong
\Big(\widehat{\crr}_{s_{\infty}} \cap \bbv_{s}^{0}\Big)\,\Big/\,n.
\end{equation}


Let $R_{I_{\bullet},n}$ be the subspace of $H^{1}(C_{\gamma},\bz/n)$ generated by the classes $c_{I_{1}}\mod n,\cdots, c_{I_{l}}\mod n$.  
Let $\crr_{I_{\bullet},n}$ be the constant subsheaf of ${\crr}_{n}$ with fiber $R_{I_{\bullet},n}$ at $0_{n}$, and let 
$$
\widehat{\crr}_{I_{\bullet}}=\varprojlim_{(n,p)=1}\crr_{I_{\bullet},n}.
$$
Its fiber $\big(\widehat{\crr}_{I_{\bullet}}\big)_{s_{\infty}}$ can be naturally identified as a subgroup of $H^{1}(\cc_{s}, \widehat{\bz}^{(p)})$.

\begin{defn}\label{define si}

For a partition $I_{\bullet}$ of $\{1,\cdots,r\}$, let $\cs_{I_{\bullet}}$ be the closed subscheme of $\cb^{\circ}$ defined by
$$
\cs_{I_{\bullet}}=\Big\{s\in \cb^{\circ} \,\big|\, \big(\widehat{\crr}_{I_{\bullet}}\big)_{s_{\infty}} \subset \bbv_{s}^{0} \Big\}. 
$$ 
Let $\cs_{I_{\bullet}}^{\circ}$ be the subscheme of $\cs_{I_{\bullet}}$ consisting of the points $s$ such that all the singularities of $\cc_{\bar{s}}$ are ordinary double points and that $\dim(\bbv_{s}^{0})=h_{I_{\bullet}}$. Note that the second condition implies that $\cs_{I_{\bullet}}^{\circ}$ is contained in the $\delta$-stratum of $\cb^{\circ}$ with $\delta=h_{I_{\bullet}}$.

\end{defn}

\begin{rem}\phantomsection\label{strata closure}

\begin{enumerate}[itemsep=0pt, label=$(\arabic*)$, leftmargin=*]

\item

Over the strict Henselization $\cb_{\{0\}}$, in terms of the distinguished basis of A'Campo, the definition of $\cs_{I_{\bullet}}$ can be rewritten as
$$
\cs_{I_{\bullet}}=\Big\{s\in \cb^{\circ} \,\big|\, \alpha_{j} \in \bbv_{s}^{0} \text{ for } j\in S_{I_{\bullet}} \Big\}, 
$$ 
and $\cs_{I_{\bullet}}^{\circ}$ is the dense open subscheme of $\cs_{I_{\bullet}}$ consisting of the points $s$ such that all the singularities of $\cc_{\bar{s}}$ are ordinary double points and that the vanishing cycles at these points are exactly $\{\alpha_{j}\}_{j\in S_{I_{\bullet}}}$.



\item

Let $I_{\bullet}$ and $I_{\bullet}'$ be two partitions of $\{1,\cdots,r\}$, we say that $I_{\bullet}$ refines $I_{\bullet}'$, denoted $I_{\bullet}\vdash I_{\bullet}'$, if any $I_{j}'$ can be written as a union of several $I_{i}$'s. It is clear that $\cs_{I_{\bullet}}^{\circ}\subset \overline{\cs^{\circ}_{I_{\bullet}'}}$ if $I_{\bullet}$ refines $I_{\bullet}'$.
Let $\bbv_{I\bullet}^{\circ}$ be the complement in $\bbv_{I_{\bullet}}$ of the union $\bigcup_{I_{\bullet}\vdash I'_{\bullet}} \bbv_{I_{\bullet}'}$.
For any $n\in \bn$, let $(\bbv_{I\bullet}/n)^{\circ}$ be the complement in $\bbv_{I_{\bullet}}/n$ of the union $\bigcup_{I_{\bullet}\vdash I'_{\bullet}} \bbv_{I_{\bullet}'}/n$.

\end{enumerate}

\end{rem}

For $n\in \bn, (n,p)=1$, let $\cs_{I_{\bullet},n}=\cs_{I_{\bullet}}\times_{\cb^{\circ}}\cb_{n}$ and similarly for $\cs_{I_{\bullet},n}^{\circ}$.

\begin{thm}\label{pi 0 curve}

Let $s_{n}$ be a geometric point of $\cs_{I_{\bullet},n}^{\circ}$, then the irreducible components of the curve $\cc_{n,s_{n}}$  can be identified with $\bbv_{I_{\bullet}}/n$. Moreover, the union 
$
\bigcup_{I_{\bullet}\vdash \{1,\cdots,r\}} \cs_{I_{\bullet},n}
$ 
is the locus where the geometric fiber of $\pi_{n}:\cc_{n}\to \cb_{n}$ can have multiple irreducible components.

\end{thm}

\begin{proof}

It is enough to treat the case that $s=\varpi_{n}(s_{n})$ is a geometric point of $\cb_{\{0\}}$. 
In this case, we can identify $\widehat{\crr}_{s_{\infty}}$ and $\bbv^{\pi_{1}(\cb_{\{0\}}-\Delta,\,\bar{\xi}_{0})}$, where $s_{\infty}$ is a geometric point of $\cb_{\infty}$ lying over $s$.

The first assertion follows from lemma \ref{key irred},  combined with the equation (\ref{rn v0 infinite}) and the identity $\bbv_{I_{\bullet}}=\bbv_{s}^{0}\cap \bbv^{\pi_{1}(\cb_{\{0\}}-\Delta,\,\bar{\xi}_{0})}$ for $s\in \cs_{I_{\bullet}}^{\circ}$. 
For the second assertion, what need to be shown is that $\bbv_{s}^{0}\cap \bbv^{\pi_{1}(\cb_{\{0\}}-\Delta,\,\bar{\xi}_{0})}$ for any geometric point $s$ of $\cb_{\{0\}}$, if non-empty, must be of the form $\bbv_{I_{\bullet}}$ for some partition $I_{\bullet}$ of $\{1,\cdots,r\}$. 
By proposition \ref{h1 inv generator}, $\bbv^{\pi_{1}(\cb_{\{0\}}-\Delta,\,\bar{\xi}_{0})}$ is generated by $c_{i}, i=1,\cdots,r$.
Note that for an element of $\bbv_{s}^{0}\cap \bbv^{\pi_{1}(\cb_{\{0\}}-\Delta,\,\bar{\xi}_{0})}$ with expansion $\sum_{j\in J}a_{j}\alpha_{j}$, $a_{j}\neq 0$, $J\subset\{1,\cdots, \mu_{0}\}$, the elements $\alpha_{j}$ will necessarily lie in $\bbv_{s}^{0}$.
This implies the following property of $\bbv_{s}^{0}\cap \bbv^{\pi_{1}(\cb_{\{0\}}-\Delta,\,\bar{\xi}_{0})}$:
Let $c',c''$ be two elements of $\bbv^{\pi_{1}(\cb_{\{0\}}-\Delta,\,\bar{\xi}_{0})}$ such that $c'+c''\in \bbv_{s}^{0}\cap \bbv^{\pi_{1}(\cb_{\{0\}}-\Delta,\,\bar{\xi}_{0})}$, suppose that the height of $c'+c''$ equals the sum of the height of $c'$ and $c''$, then both $c'$ and $c''$ lies in $\bbv_{s}^{0}\cap \bbv^{\pi_{1}(\cb_{\{0\}}-\Delta,\,\bar{\xi}_{0})}$.
Indeed, the hypothesis on the height means that there is no cancel-out of any terms of $c'$ or $c''$, hence all the terms in the expansion of $c'$ and $c''$ into $\alpha_{j}$'s belongs to $\bbv_{s}^{0}$, and so do $c'$ and $c''$ themselves.

Let $b_{1}$ be an element of $\bbv_{s}^{0}\cap \bbv^{\pi_{1}(\cb_{\{0\}}-\Delta,\,\bar{\xi}_{0})}$ with minimal height. We claim that it must be a multiple of some $c_{I_{1}}$. 
Indeed, we can expand $b_{1}=\sum_{i\in I_{1}} a_{i}c_{i}$ with non-zero coefficients, and then regroup the summands with the same coefficient in the expansion, to write it as $b_{1}=\sum_{j}a'_{j} c_{I'_{j}}$ such that the coefficients $a'_{j}$ are distinct to each other. Then there is no cancelation between different items $a'_{j} c_{I'_{j}}$. By the above property, we obtain that all the $c_{I'_{j}}$ will belong to $\bbv_{s}^{0}\cap \bbv^{\pi_{1}(\cb_{\{0\}}-\Delta,\,\bar{\xi}_{0})}$, contradiction to the assumption that $b_{1}$ is of minimal height. Hence $b_{1}=a_{i}c_{I_{1}}$ and so $c_{I_{1}}$ belongs to $\bbv_{s}^{0}\cap \bbv^{\pi_{1}(\cb_{\{0\}}-\Delta,\,\bar{\xi}_{0})}$.

Let $\bbv/I_{1}=\bbv/\pair{\alpha_{j}\mid j\in J_{i} \text{ and } i\in I_{1}}$, let $\bbv^{0}_{s}/I_{1}$ be the image of $\bbv^{0}_{s}$ in $\bbv/I_{1}$ and $\bbv^{\pi_{1}(\cb_{\{0\}}-\Delta,\,\bar{\xi}_{0})}/I_{1}$ the image of $\bbv^{\pi_{1}(\cb_{\{0\}}-\Delta,\,\bar{\xi}_{0})}$ in $\bbv/I_{1}$. For $i\in \{1,\cdots,r\}\backslash I_{1}$, let $\bar{c}_{i}$ be the image of $c_{i}$ in $\bbv/I_{1}$, they generate $\bbv^{\pi_{1}(\cb_{\{0\}}-\Delta,\,\bar{\xi}_{0})}/I_{1}$. 
Geometrically, we are considering the germ of singularity obtained by erasing from the germ of singularity of $C_{\gamma}$ at $c$ the branches defined by $P_{i}(x,y), i\in I_{1}$. 
By induction on the number of branches $r$, we can assume that the theorem holds for $(\bbv^{0}_{s}/I_{1})\cap (\bbv^{\pi_{1}(\cb_{\{0\}}-\Delta,\,\bar{\xi}_{0})}/I_{1})\subset \bbv/I_{1}$.
Hence there exists a partition $\{1,\cdots,r\}\backslash I_{1}=\bigsqcup_{i=2}^{l}I_{j}$ such that $\bar{c}_{I_{2}},\cdots, \bar{c}_{I_{l}}$ generates $(\bbv^{0}_{s}/I_{1})\cap (\bbv^{\pi_{1}(\cb_{\{0\}}-\Delta,\,\bar{\xi}_{0})}/I_{1})$, where $\bar{c}_{I_{j}}$ are linear combinations of $\bar{c}_{i}$ similar to $c_{I_{j}}$. We claim that $c_{I_{1}},\cdots, c_{I_{l}}$ belongs to $\bbv^{0}_{s}\cap \bbv^{\pi_{1}(\cb_{\{0\}}-\Delta,\,\bar{\xi}_{0})}$ and they generate it.

By proposition \ref{non zero coeff c}, the non-zero terms in the expansion of $c_{I_{1}}$ into $\alpha_{j}$'s are indexed by $j\in J_{i}^{\circ}\cap J_{i'}^{\circ}, i\in I_{1}, i'\in \bar{I}_{1}$. This implies that
$$
\bbv_{s}^{0}\supset \big\{\alpha_{j}\mid j\in J_{i}^{\circ}\cap J_{i'}^{\circ}, i\in I_{1}, i'\in \bar{I}_{1}\big\}.
$$
In particular, $\bbv_{s}^{0}$ contains the subgroup generated by these cycles. 
Note that $\bar{c}_{I_{j}}$ can be lifted to ${c}_{I_{j}}$ with elements from this subgroup   for $j=2,\cdots,l$, hence ${c}_{I_{j}}\in \bbv_{s}^{0}$ and so belongs to $\bbv^{0}_{s}\cap \bbv^{\pi_{1}(\cb_{\{0\}}-\Delta,\,\bar{\xi}_{0})}$. 
To show that $c_{I_{1}},\cdots, c_{I_{l}}$ generates the intersection. Let $c$ be an element of the intersection, its image in $\bbv/I_{1}$ is then a linear combination of $\bar{c}_{I_{2}},\cdots, \bar{c}_{I_{l}}$ by induction. Take the difference $c'$ of $c$ and the same linear combination of $c_{I_{2}},\cdots, c_{I_{l}}$. Then $c'\in \bbv_{s}^{0}\cap \bbv^{\pi_{1}(\cb_{\{0\}}-\Delta,\,\bar{\xi}_{0})}$ and in the expansion of $c'$ in terms of $\alpha_{j}$, only the terms with index $j\in J_{i}^{\circ}\cap J_{i'}^{\circ}, i\in I_{1}, i'\in \bar{I}_{1}$ appear. 
Compare with the expression of $c_{I_{1}}$ in proposition \ref{non zero coeff c}, we obtain that the height of $c'$ is less than or equal to that of $c_{I_{1}}$. But $c_{I_{1}}$ has been taken to be of minimal height, hence $c'$, if non-zero, will have the same height as that of $c_{I_{1}}$ and have an expansion $c'=\sum_{i\in I_{1}, i'\in \bar{I}_{1}}\sum_{j\in J_{i}^{\circ}\cap J_{i'}^{\circ}} a_{j}''\alpha_{j}$ with non-zero coefficients.  
In the latter case, $c'$ has to be a multiple of $c_{I_{1}}$, otherwise a suitable linear combination of $c'$ and $c_{I_{1}}$ will have smaller height than $c_{I_{1}}$, contradiction to the assumption about $c_{I_{1}}$. This finishes the proof.

\end{proof}

\begin{example}

Let $G=\gl_{2}$, $\gamma=\begin{bmatrix}\varep^{n}&\\&-\varep^{n}
\end{bmatrix}\in \ggl_{2}
$.
The spectral curve $C_{\gamma}$ can be taken to the Zariski closure in $\bp^{2}$ of the affine curve 
$$
y^{2}=x^{2n}(x-1)
$$ 
with the singularity at infinity resolved. Consider the family of curves
$$
y^{2}=(x-\lambda_{1})\cdots (x-\lambda_{2n})(x-1)
$$  
over the hyperplane $$
H=\big\{(\lambda_{1},\cdots,\lambda_{2n})\in \bc^{2n}\mid \lambda_{1}+\cdots+\lambda_{2n}=0\big\}.
$$
Take the Zariski closure of the family in $H\times \bp^{2}$ and resolve the singularity at infinity, we get a flat family of projective curves $\tilde{\pi}:\widetilde{\cc}\to H$.
It is clear that the symmetric group $\kss_{2n}$ acts on $H$ and the action lifts to the family $\tilde{\pi}$. Take quotients, we get the family 
$$
\pi: \cc:=\widetilde{\cc}/\kss_{2n}\to \cb:=H/\kss_{2n}.
$$
It is an algebraization of a miniversal deformation of $C_{\gamma}$, with generic geometric fiber an irreducible smooth projective algebraic curve of genus $n$.
The discriminant locus $\widetilde{\Delta}$ of of the family $\tilde{\pi}:\widetilde{\cc}\to H$ is the union of the hyperplanes
$$
H_{i,j}=\big\{(\lambda_{1},\cdots,\lambda_{2n})\in H\mid \lambda_{i}=\lambda_{j}\big\},\quad i\neq j=1,\cdots, 2n,
$$ 
and the hyperplanes
$$
H_{i,2n+1}=\big\{(\lambda_{1},\cdots,\lambda_{2n})\in H\mid \lambda_{i}=1\big\},\quad i=1,\cdots, 2n.
$$
The discriminant $\Delta$ of the family $\pi:\cc\to \cb$ is then equal to $\widetilde{\Delta}/\kss_{2n}$, and the irreducible component $\Delta_{0}$ of $\Delta$ containing $0$ is $\big(\bigcup_{i\neq j=1}^{2n} H_{i,j}\big)/\kss_{2n}$.  
Note that the union of hyperplanes $\bigcup_{i\neq j=1}^{2n} H_{i,j}$ coincides with the union of the walls in the root system $A_{2n-1}$.
Let $D_{0}$ be the closure of the Weyl chamber associated to the standard Borel subgroup of the  upper triangular matrices. It is clear that the natural projection $H\to H/\kss_{2n}$ remains surjective when restricted to $D_{0}$.  
Hence to understand the degeneration to $C_{\gamma}$ of the family of curves $\pi:\cc\to \cb$, it is enough to examine the situation over $D_{0}$. Take a generic point $z$ in $D_{0}$, consider the cycles as indicated in figure \ref{cycles} on the algebraic curve over it.
\begin{figure}[h]
\centering
\includegraphics[width=9cm]{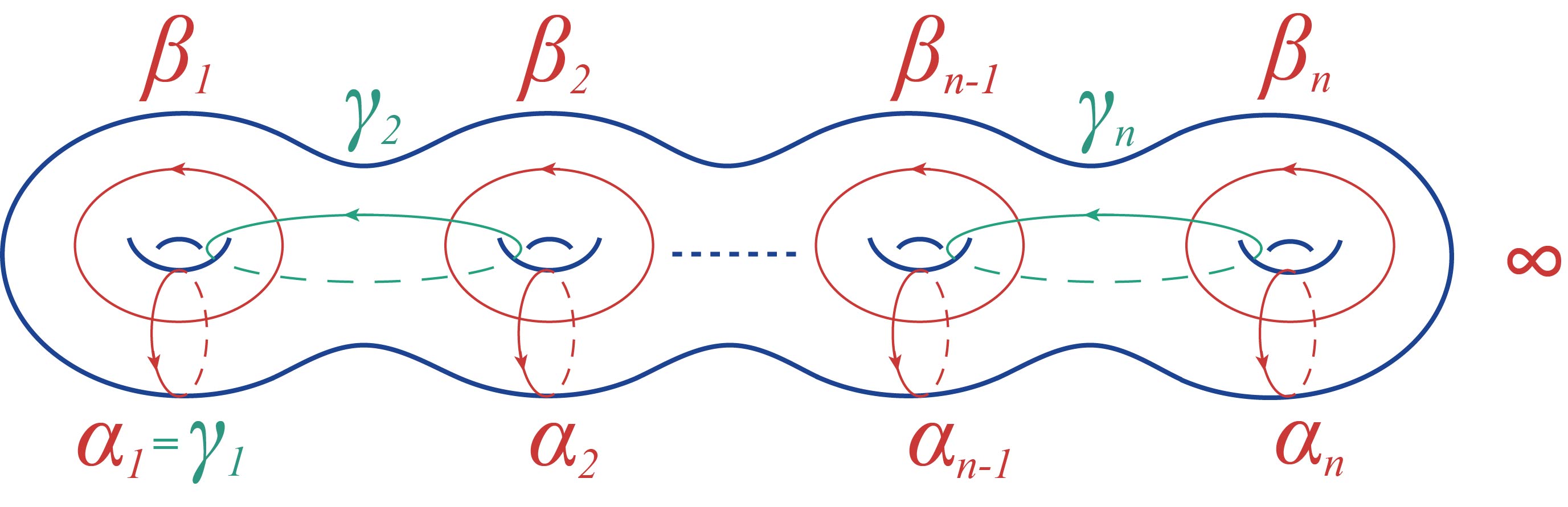}
\caption{Cycles on the algebraic curve.}
\label{cycles}
\end{figure}

 At a sufficiently small neighborhood of $0$, when the point $z$ hits one of the walls bounding $D_{0}$, one of the cycles 
$$
\gamma_{1}, \beta_{1}, \gamma_{2}, \beta_{2}, \cdots, \gamma_{n-1},\beta_{n-1}, \gamma_{n}
$$
will degenerate. Their union can then be taken as a set of distinguished basis of the vanishing cycles for the degeneration to $C_{\gamma}$.
As the local monodromy group $\pi_{1}(\cb_{\{0\}}-\Delta,\bar{\xi})$ is generated by the Picard-Lefschetz transformation for these vanishing cycles, the subspace of local invariant cycles is then generated by the cycle $\alpha_{n}$.
Note that we have the equality
\begin{align*}
\alpha_{n}&=\alpha_{n-1}+\gamma_{n}=\alpha_{n-2}+\gamma_{n-1}+\gamma_{n}=\cdots=\alpha_{1}+\gamma_{2}+\cdots+\gamma_{n}
\\
&=\gamma_{1}+\cdots+\gamma_{n}.
\end{align*} 
Hence to understand the finite abelian coverings of the family $\pi:\cc\to \cb$ over $\cb-\Delta^{(0)}$, we need to understand how the cycles $\gamma_{1},\cdots,\gamma_{n}$ degenerate.

As an illustration, for $n=2$, the figure \ref{degenerate} gives a picture of the family of curves $\tilde{\pi}:\widetilde{\cc}\to H$ over $D_{0}$ at a neighbourhood of $0$.
\begin{figure}[h]
\centering
\includegraphics[width=8cm]{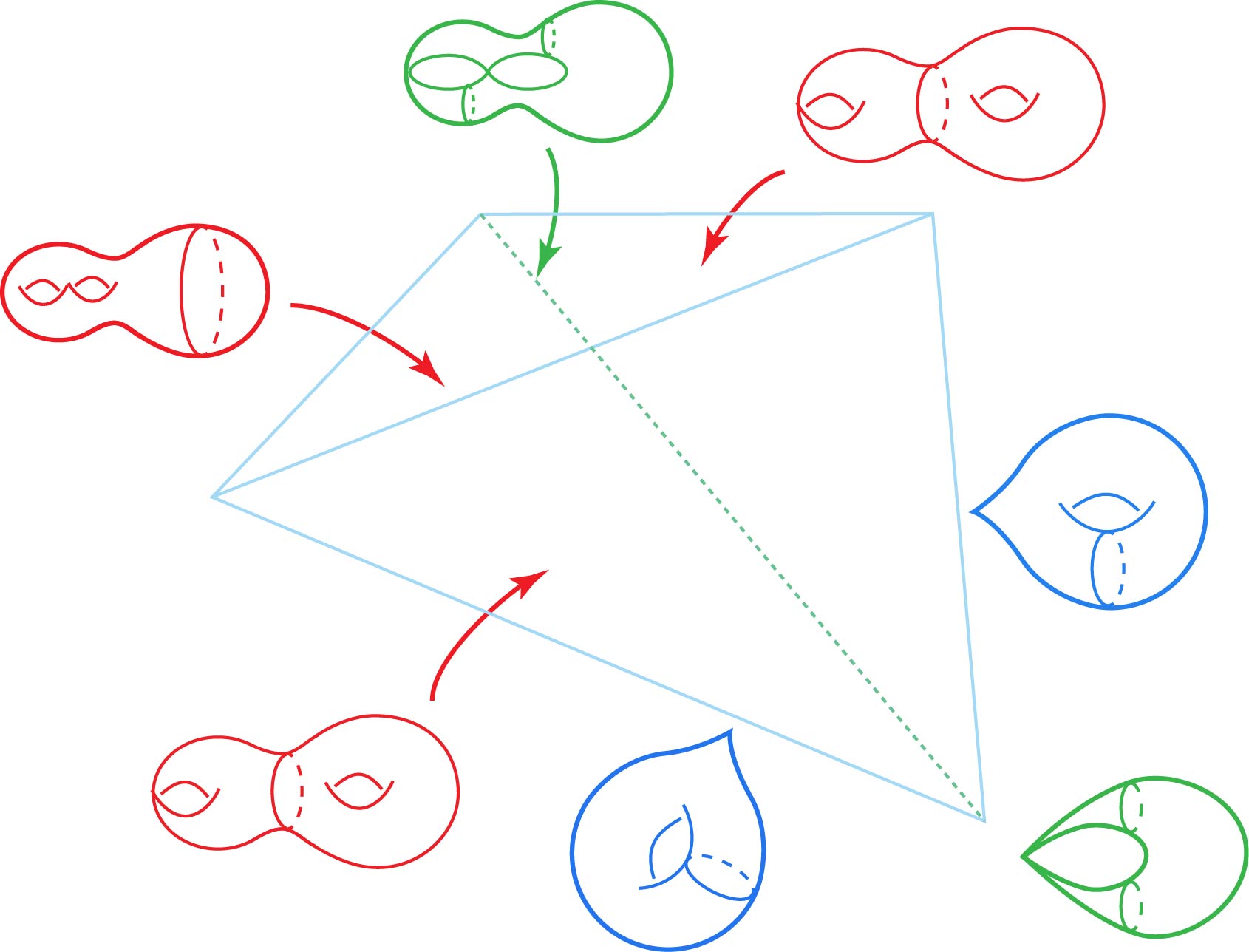}
\caption{Degeneration of the family $\tilde{\pi}$ over $D_{0}$.}
\label{degenerate}
\end{figure}
Examining the degeneration of the cycle $\alpha_{2}=\gamma_{1}+\gamma_{2}$ at the boundary of $D_{0}$, we can obtain a picture of the finite abelian covering $\cc_{m}\to \cc\times_{\cb}\cb_{m}$ (in this case a $\bz/m$-covering), hence also that of the family $\pi_{m}:\cc_{m}\to \cb_{m}$. The figure \ref{covering} gives a picture of the $\bz/3$-covering of the family $\tilde{\pi}:\widetilde{\cc}\to H$ over $D_{0}$.   
\begin{figure}[h]
\centering
\includegraphics[width=8cm]{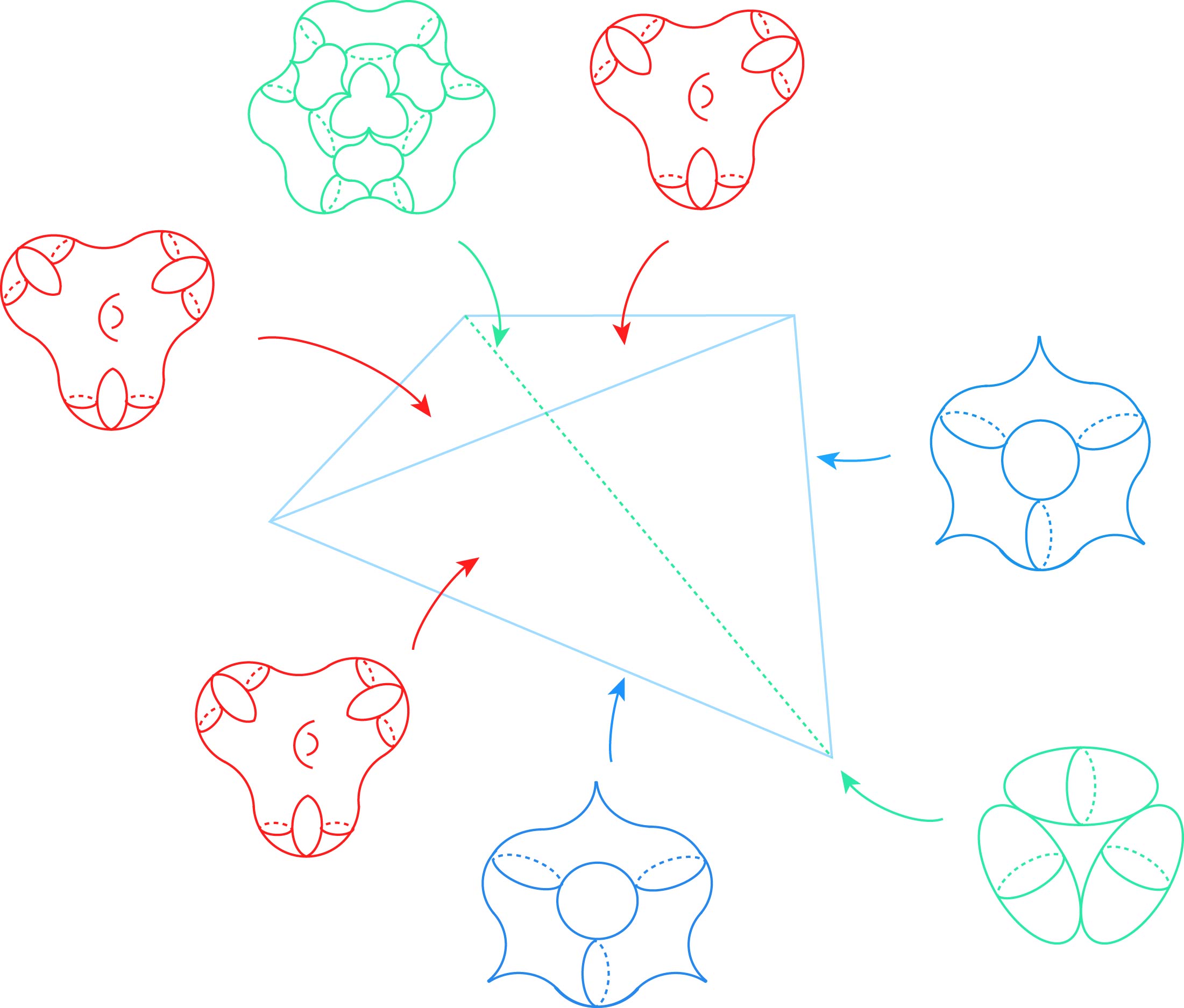}
\caption{$\bz/3$-covering of the family $\tilde{\pi}$ over $D_{0}$.}
\label{covering}
\end{figure}

\cqfd

\end{example}

By theorem \ref{irreducible components}, there is a bijection between the irreducible components of the geometric fibers of $f_{n}$ and $\pi_{n}$ over any geometric point of $\cb_{n}$. Theorem \ref{pi 0 curve} then translates to:

\begin{thm}\label{pi 0 picard}

Let $s_{n}$ be a geometric point of $\cs_{I_{\bullet},n}^{\circ}$, then the irreducible components of $\,\overline{\cp}_{n,s_{n}}$  can be identified with $\bbv_{I_{\bullet}}/n$. Moreover, the union $\bigcup_{I_{\bullet}\vdash \{1,\cdots,r\}} \cs_{I_{\bullet},n}$ is the locus where the geometric fibers of $f_{n}:\overline{\cp}_{n}\to \cb_{n}$ can have multiple irreducible components.

\end{thm}

\subsection{Reductions over smaller strata}\label{reduction over si}

For $i=1,\cdots, r$, let $\widehat{C}_{i}=\spf \big(k[\![x,y]\!]/(P_{i}(x,y))\big)$, it is a  branch of the  singularity of $C_{\gamma}$ at $c$. 
For a subset $I\subset \{1,\cdots,r\}$, let $\gamma_{I}=\diag(\gamma_{i})_{i\in I}$ and let $\widehat{C}_{I}=\bigcup_{i\in I} \widehat{C}_{i}$ and let $P_{I}(x,y)=\prod_{i\in I}P_{i}(x,y)$.  
Recall that $\widehat{\cb}$ is the completion of $\cb$ at $0$, and we can fix an isomorphism $\widehat{\cb}=\spec\big(k[\![t_{1},\cdots,t_{\tau_{\gamma}}]\!]\big)$, with $\tau_{\gamma}$ the Tjurina number of the  singularity of $C_{\gamma}$ at $0$. Suppose that the miniversal  deformation $\hat{\pi}:\widehat{\cc}_{c}\to \widehat{\cb}$ is described by the power series
$$
P(x,y,{\bf{t}})\in k[\![x,y,t_{1},\cdots,t_{\tau_{\gamma}}]\!].
$$
In other words, we have the commutative diagram
$$
\begin{tikzcd}
\widehat{\cc}_{c}\arrow[d]\arrow[r, "\simeq"] & {\rm V}\big(P(x,y,{\bf{t}})\big)\subset\widehat{\ba}_{0}^{2}\times \widehat{\ba}^{\tau_{\gamma}}_{0}  \arrow[d] 
\\
\widehat{\cb}\arrow[r, "\simeq"]& \widehat{\ba}^{\tau_{\gamma}}_{0}. 
\end{tikzcd}
$$   

For a partition $I_{\bullet}\vdash\{1,\cdots,r\}$ of length $l$, consider the family 
$$
\pi_{I_{\bullet}}:\cc_{I_{\bullet}}:=\cc\times_{\cb}\cs_{I_{\bullet}}\to \cs_{I_{\bullet}}.
$$ 
Let $\widehat{\cs}_{I_{\bullet}}$ be the completion of $\cs_{I_{\bullet}}$ at $0$.

\begin{prop}\label{si geom}

Over $\widehat{\cs}_{I_{\bullet}}$, we can factorize
$
P(x,y,{\bf{t}})=\prod_{j=1}^{l}P_{I_{j}}(x,y,{\bf t}),
$ 
with $P_{I_{j}}(x,y,{\bf t})\in \bar{k}[\![x,y,{\bf t}]\!]$ satisfying $P_{I_{j}}(x,y,0)=P_{I_{j}}(x,y), j=1,\cdots,l$.
Moreover, this factorization property defines $\widehat{\cs}_{I_{\bullet}}$ as a closed sub formal scheme of $\widehat{\cb}$.

\end{prop}

\begin{proof}

Recall that we have constructed a set of invertible sheaves $\cl_{i,\zeta_{n}}, i=1,\cdots,r,$ on the curve $C_{\gamma}$ in the proof of proposition \ref{ci express}, where $\zeta_{n}$ is a primitive $n$-th root of unity and $(n, p)=1$. 
For a subset $I\subset \{1,\cdots,r\}$, let $\cl_{I,\zeta_{n}}=\bigotimes_{i\in I} \cl_{i,\zeta_{n}}$, it corresponds to a point $[\cl_{I,\zeta_{n}}]$ of $J_{C_{\gamma}}[n]$. 
By construction, $[\cl_{I,\zeta_{n}}]\otimes\mu_{n}^{-1}$ is the image of the class $c_{I}\mod n$ under the isomorphism
$$
H^{1}(C_{\gamma},\bz/n)=J_{C_{\gamma}}[n]\otimes \mu_{n}^{-1}. 
$$
With the isomorphism (\ref{h1 curve mod n}) for $x=0$, the point $[\cl_{I,\zeta_{n}}]$ extends to a section of $\pic_{\cc/\cb}[n]$ over $\widehat{\cb}$, let $\widehat{\cl}_{I,\zeta_{n}}$ be the corresponding relative invertible sheaf on $\cc\times_{\cb}\widehat{\cb}$. 
Then we can rewrite the definition of $\widehat{\cs}_{I_{\bullet}}$ as
$$
\widehat{\cs}_{I_{\bullet}}=\Big\{s\in \widehat{\cb} \,\big|\, \big[\widehat{\cl}_{I_{j},\zeta_{n}}\big]_{s} \in J_{\cc_{s}}^{\rm a}[n] \text{ for all } j \text{ and all }n \Big\}, 
$$ 
where $J_{\cc_{s}}^{\rm a}$ is the maximal connected affine sub group scheme of $J_{\cc_{s}}$ as defined in (\ref{chevalley jac exp}).

By construction, the invertible sheaf $\cl_{I,\zeta_{n}}$ can be obtained by gluing $\co_{C_{\gamma}\backslash \{c\}}$ and the $A$-submodule of $E=\prod_{i=1}^{r}E_{i}$ generated by $(\cdots, \zeta_{n},\cdots,1,\cdots)$ along $\spec(E)=(C_{\gamma}\backslash \{c\})\times_{C_{\gamma}}\spec(A)$, with $\zeta_{n}$ sitting at the positions indexed by $i\in I$. 
Applying the construction to $\cl_{I_{j},\zeta_{n}}$'s, we see that for all the extensions $\widehat{\cl}_{I_{j},\zeta_{n},s}$ to satisfy the condition 
$$
\big[\widehat{\cl}_{I_{j},\zeta_{n},s}\big] \in J_{\cc_{s}}^{\rm a}[n], \quad \text{ for all } n,
$$
the deformation from $(\widehat{\cc}_{c})_{0}=\widehat{C}_{\gamma,c}$ to $(\widehat{\cc}_{c})_{s}$ has to satisfy the condition that the intersections between distinct $\widehat{C}_{I_{j}}$ and $\widehat{C}_{I_{j'}}$ were not smoothed. In particular, the intersection multiplicity between them doesn't change. 
Indeed, the above condition for a subset $I\subset \{1,\cdots,l\}$ is equivalent to the condition that the point $\big[\widehat{\cl}_{I,\zeta_{n},s}\big]\in J_{\cc_{s}}[n]$ goes to $0$ under the morphism
$$
J_{\cc_{s}}[n] \to J_{\widetilde{\cc}_{s}}[n], 
$$
which is equivalent to the condition that the deformations of $\widehat{C}_{I}$ and $\widehat{C}_{\bar{I}}$ in $\cc_{s}$ can be detached from each other in the normalization $\widetilde{\cc}_{s}$. In other words, their deformations in $\cc_{s}$ belong to different local branches of $(\widehat{\cc}_{c})_{s}$. 
The last condition for all $I_{j}$'s is clearly equivalent to the condition that the deformations of all the distinct $\widehat{C}_{I_{j}}$ and $\widehat{C}_{{I}_{j'}}$ in $\cc_{s}$ belong to different local branches of $(\widehat{\cc}_{c})_{s}$, which is exactly the condition in the proposition.

\end{proof}

By the construction of $\cs_{I_{\bullet}}$, the conclusion of proposition \ref{si geom} extends to the family $\pi_{I_{\bullet}}:\cc_{I_{\bullet}}\to \cs_{I_{\bullet}}$. 
Let $A=k[\![x,y]\!]/(P(x,y))$ and  $A_{I_{j}}=k[\![x,y]\!]/(P_{I_{j}}(x,y))$ for $j=1,\cdots,l$, let $A_{I_{\bullet}}=\prod_{j=1}^{l}A_{I_{j}}$.  
The proposition suggests that we should consider the deformation functor
$$
\defm_{A\hookrightarrow A_{I_{\bullet}}}^{\rm top}: {\rm Art}_{k}\to {\rm Sets}, 
$$
which sends an artinian $k$-algebra $R$ to the set of isomorphism classes of homomorphisms of $R$-algebras $A_{R}\to A_{I_{\bullet}, R}$ whose reduction modulo $\km_{R}$ equals $A\hookrightarrow A_{I_{\bullet}}$, here $A_{R}$ is a flat deformation of $A$ over $R$ and $A_{I_{\bullet}, R}$ is a flat deformation of $A_{I_{\bullet}}$. It is clear that we have the forgetful morphism
$$
\defm_{A\hookrightarrow A_{I_{\bullet}}}^{\rm top}\to \defm_{A}^{\rm top}.
$$
Similar deformation functor has been considered by Laumon \cite{laumon springer}, \S A.3, we will borrow his techniques. 

\begin{lem}\label{auto inj}

Any homomorphism of $R$-algebra $(A_{R}\to A_{I_{\bullet}, R})\in \defm_{A\hookrightarrow A_{I_{\bullet}}}^{\rm top}(R)$ is necessarily injective and its cokernel is automatically $R$-flat. 

\end{lem}

\begin{proof}

The proof is identical to \cite{laumon springer}, lemme A.3.2.
Let $N_{R}, I_{R}$ and $C_{R}$ be the kernel, image and cokernel of the morphism $A_{R}\to A_{I_{\bullet}, R}$. Then we have the exact sequences
$$
0\to \Tor^{R}_{1}(I_{R}, k)\to N_{R}\otimes_{R}k\to A\to I_{R}\otimes_{R}k\to 0,
$$
and
$$
0\to \Tor^{R}_{1}(C_{R}, k)\to I_{R}\otimes_{R}k\to A_{I_{\bullet}}\to C_{R}\otimes_{R}k\to 0.
$$
Note that the composite $A\twoheadrightarrow I_{R}\otimes_{R}k \to A_{I_{\bullet}}$ equals the injection $A\hookrightarrow A_{I_{\bullet}}$, whence
$$
A\cong I_{R}\otimes_{R}k \quad \text{and}\quad
I_{R}\otimes_{R}k \hookrightarrow A_{I_{\bullet}}.
$$
With the above exact sequences, they imply 
$$
\Tor^{R}_{1}(I_{R}, k)\cong N_{R}\otimes_{R}k \quad \text{and}\quad 
\Tor^{R}_{1}(C_{R}, k)=0.
$$
We deduce that $C_{R}$ must be $R$-flat and hence so is $I_{R}$ as $A_{I_{\bullet},R}$ is $R$-flat by assumption. Then $N_{R}\otimes_{R}k\cong \Tor^{R}_{1}(I_{R}, k)=0$ and so $N_{R}=0$ by Nakayama's lemma.

\end{proof}

Consider now the deformation functors 
$$
\defm^{\rm top}_{k[\![x,y]\!]\twoheadrightarrow A}, \defm^{\rm top}_{k[\![x,y]\!]\to A_{I_{\bullet}}}: {\rm Art}_{k}\to {\rm Sets}, 
$$
which send an artinian $k$-algebra $R$ to the set of isomorphism classes of homomorphisms of $R$-algebras 
$$
R[\![x,y]\!]\to A_{R} \quad \text{and}\quad
R[\![x,y]\!]\to A_{I_{\bullet},R},
$$
whose reduction modulo $\km_{R}$ equals respectively the morphism $k[\![x,y]\!]\twoheadrightarrow A$ and the composite $k[\![x,y]\!]\twoheadrightarrow A\hookrightarrow A_{I_{\bullet}}$, here as before $A_{R}$ is a flat deformation of $A$ over $R$ and $A_{I_{\bullet}, R}$ is a flat deformation of $A_{I_{\bullet}}$. 
Note that the morphism $R[\![x,y]\!]\to A_{R}$ is automatically surjective by Nakayama's lemma. We have the obvious morphisms of functors
$$
\begin{tikzcd}
&\defm_{A\hookrightarrow A_{I_{\bullet}}}^{\rm top}\ar[d]
\\ 
\defm^{\rm top}_{k[\![x,y]\!]\twoheadrightarrow A}\ar[r] &\defm_{A}^{\rm top}
\end{tikzcd},
$$
and we set
$$
\defm^{\rm top}_{k[\![x,y]\!]\twoheadrightarrow A\hookrightarrow A_{I_{\bullet}}}:=
\defm^{\rm top}_{k[\![x,y]\!]\twoheadrightarrow A}\times_{\defm_{A}^{\rm top}}\defm_{A\hookrightarrow A_{I_{\bullet}}}^{\rm top}.
$$
The functor $\defm_{A}^{\rm top}$, the morphism of functors $\defm^{\rm top}_{k[\![x,y]\!]\twoheadrightarrow A}\to\defm_{A}^{\rm top}$ hence also the functor $\defm^{\rm top}_{k[\![x,y]\!]\twoheadrightarrow A}$ are formally smooth over $k$.

\begin{thm}\label{iso embed}

The morphism coming from composite $\defm^{\rm top}_{k[\![x,y]\!]\twoheadrightarrow A\hookrightarrow A_{I_{\bullet}}} \to \defm^{\rm top}_{k[\![x,y]\!]\to A_{I_{\bullet}}}$ is an isomorphism of functors. 

\end{thm}

\begin{proof}

The proof is identical to \cite{laumon springer}, th\'eor\`eme A.3.1. We need to construct an inverse to the morphism in question. 
Let $R[\![x,y]\!]\to A_{I_{\bullet},R}$ be a  deformation of $k[\![x,y]\!]\to A_{I_{\bullet}}$. 
Then $A_{I_{\bullet}}$ and $A_{I_{\bullet},R}$ can be seen as modules over $k[\![x,y]\!]$ and $R[\![x,y]\!]$ with these morphisms.

As $A_{I_{\bullet}}$ is a finite and torsion free $A$-module, by lemma A.3.3 of \textit{loc. cit.}, $A_{I_{\bullet}}$ admits a resolution as $k[\![x,y]\!]$-module
\begin{equation}\label{resolve ai}
0\to k[\![x,y]\!]^{n}\xrightarrow{F} k[\![x,y]\!]^{n}\to  A_{I_{\bullet}}\to 0. 
\end{equation}
Let $F^{*}$ be the matrix of cofactors of $F$, then $F^{*}F=FF^{*}=\det(F)$ annihilates $A_{I_{\bullet}}$ as $k[\![x,y]\!]$-module, which implies that $\det(F)\in (P(x,y)).$
Conversely, let $g\in k[\![x,y]\!]$ be an annihilator of $A_{I_{\bullet}}$, then there exists a matrix $G\in {\rm Mat}_{n\times n}(k[\![x,y]\!])$ such that $g=G\circ F$, whence $g^{n}=\det(g)=\det(G)\det(F)$. 
In particular, we can take $g=P(x,y)$ and this implies that $(P(x,y))^{n}\subset (\det(F))$. 
As $k[\![x,y]\!]$ is reduced, we obtain then $(P(x,y))=(\det(F))$. In other words, we can reconstruct the algebra $A=k[\![x,y]\!]/(P(x,y))$ from the resolution (\ref{resolve ai}).  

As $A_{I_{\bullet},R}$ is flat over $R$ with $A_{I_{\bullet},R}\otimes_{R}k=A_{I_{\bullet}}$, the resolution (\ref{resolve ai}) can be lifted to a resolution
$$
0\to R[\![x,y]\!]^{n}\xrightarrow{F_{R}} R[\![x,y]\!]^{n}\to  A_{I_{\bullet},R}\to 0.
$$
With the same argument as above, we obtained that $\det(F_{R})$ is the annihilator of $A_{I_{\bullet},R}$ as $R[\![x,y]\!]$-module, i.e. we have factorization
$$
R[\![x,y]\!]\twoheadrightarrow R[\![x,y]\!]/(\det(F_{R})) \to A_{I_{\bullet},R}
$$
which lifts the morphism $k[\![x,y]\!]\twoheadrightarrow k[\![x,y]\!]/(P(x,y)) \to A_{I_{\bullet},R}$. 
Then $A_{R}:=R[\![x,y]\!]/(\det(F_{R}))$ is a flat lifting of $A$ over $R$ and we get deformation $A_{R}\to A_{I_{\bullet},R}$ of $A\to A_{I_{\bullet}}$.
By lemma \ref{auto inj}, the morphism $A_{R}\to A_{I_{\bullet},R}$ is necessarily injective, and we get the factorization $R[\![x,y]\!]\twoheadrightarrow R[\![x,y]\!]/(\det(F_{R})) \hookrightarrow A_{I_{\bullet},R}$ of the morphism $R[\![x,y]\!]\to A_{I_{\bullet},R}$. In this way, we get the inverse of the morphism in the theorem.

\end{proof}

\begin{cor}

The functor $\defm^{\rm top}_{k[\![x,y]\!]\twoheadrightarrow A\hookrightarrow A_{I_{\bullet}}}$ is formally smooth, and so is the functor $\defm^{\rm top}_{ A\hookrightarrow A_{I_{\bullet}}}$.

\end{cor}

\begin{proof}

The first assertion follows from theorem \ref{iso embed} and the fact that the functor $\defm^{\rm top}_{k[\![x,y]\!]\to A_{I_{\bullet}}}$ is formally smooth. As the morphism $\defm^{\rm top}_{k[\![x,y]\!]\twoheadrightarrow A\hookrightarrow A_{I_{\bullet}}} \to \defm^{\rm top}_{ A\hookrightarrow A_{I_{\bullet}}}$ is formally smooth, we deduce from it the formal smoothness of the latter.

\end{proof}

\begin{cor}\label{reduce to levi functor smooth}

The morphism $\defm^{\rm top}_{ A\hookrightarrow A_{I_{\bullet}}}\to \defm^{\rm top}_{A_{I_{\bullet}}}$ is formally smooth. 

\end{cor}

\begin{proof}

Consider the commutative diagram of morphisms of functors
$$
\begin{tikzcd}
\defm^{\rm top}_{k[\![x,y]\!]\twoheadrightarrow A\hookrightarrow A_{I_{\bullet}}} \ar[rd, "q"'] \ar[r, "q_{1}"]& \defm^{\rm top}_{ A\hookrightarrow A_{I_{\bullet}}}\ar[d, "q_{2}"]
\\
&\defm^{\rm top}_{A_{I_{\bullet}}}\end{tikzcd}.
$$
By theorem \ref{iso embed}, the morphism $q$ can be identified with the morphism $\defm^{\rm top}_{k[\![x,y]\!]\to A_{I_{\bullet}}} \to  \defm^{\rm top}_{A_{I_{\bullet}}}$, which is formally smooth. The morphism $q_{1}$ is formally smooth and surjective. As $q=q_{2}\circ q_{1}$, we obtain that $q_{2}$ must be smooth.

\end{proof}

\begin{rem}

Let $A\hookrightarrow A'$ be a morphism of algebras with $A'\subset \widetilde{A}$. Then the above results about the deformation functor $\defm^{\rm top}_{ A\hookrightarrow A_{I_{\bullet}}}$ continue to hold if we replace $A_{I_{\bullet}}$ with $A'$. 

\end{rem}

We can draw conclusions about $\cs_{I_{\bullet}}$ from the above local analysis. By proposition \ref{si geom}, $\widehat{\cs}_{I_{\bullet}}$ is the schematic image of the morphism $\defm^{\rm top}_{A\hookrightarrow A_{I_{\bullet}}} \to \defm^{\rm top}_{A}$.
As the morphism is finite and generically an isomorphism onto $\widehat{\cs}_{I_{\bullet}}$, and that the functor $\defm^{\rm top}_{ A\hookrightarrow A_{I_{\bullet}}}$ is formally smooth, we deduce that the functor $\defm^{\rm top}_{ A\hookrightarrow A_{I_{\bullet}}}$ is actually representable by the normalization of $\widehat{\cs}_{I_{\bullet}}$. 
Let $\varphi_{I_{\bullet}}:\widetilde{\cs}_{I_{\bullet}}\to {\cs}_{I_{\bullet}}$ be the normalization, then there exists a unique point $\tilde{0}$ of $\widetilde{\cs}_{I_{\bullet}}$ lying over $0\in \cs_{I_{\bullet}}$, it corresponds to $(A\hookrightarrow A_{I_{\bullet}})\in \defm^{\rm top}_{ A\hookrightarrow A_{I_{\bullet}}}(k)$. 
Let $\widehat{\widetilde{\cs}}_{I_{\bullet}}$ be the formal neighbourhood of $\widetilde{\cs}_{I_{\bullet}}$ at $\tilde{0}$, it can be identified with $\defm^{\rm top}_{ A\hookrightarrow A_{I_{\bullet}}}$. 
The morphism $\defm^{\rm top}_{ A\hookrightarrow A_{I_{\bullet}}} \to \defm^{\rm top}_{A_{I_{\bullet}}}$ then induces a morphism $\widehat{\widetilde{\cs}}_{I_{\bullet}} \to \prod_{j=1}^{l}\defm^{\rm top}_{\widehat{C}_{I_{j}}}$.
By corollary \ref{reduce to levi functor smooth}, the morphism is formally smooth. 
We summarize the above discussion as:

\begin{thm}\label{strata local}

The normalization $\widetilde{\cs}_{I_{\bullet}}$ of ${\cs}_{I_{\bullet}}$ is smooth over $k$, and there exists a natural morphism $\widehat{\widetilde{\cs}}_{I_{\bullet}} \to \prod_{j=1}^{l}\defm^{\rm top}_{\widehat{C}_{I_{j}}}$ which is formally smooth.

\end{thm}

The identification of $\widehat{\widetilde{\cs}}_{I_{\bullet}}$ with $\defm^{\rm top}_{ A\hookrightarrow A_{I_{\bullet}}}$ extends to $\widetilde{\cs}_{I_{\bullet}}$, let $\tilde{\phi}_{I_{\bullet}}^{+}: \widetilde{\cc}^{+}_{I_{\bullet}}\to {\cc}^{+}_{I_{\bullet}}:={\cc}_{I_{\bullet}}\times_{\cs_{I_{\bullet}}} \widetilde{\cs}_{I_{\bullet}}$ be the universal family over it. 
Then $\tilde{\phi}^{+}_{I_{\bullet}}$ is a finite flat morphism, it can be thought of as a simultaneous partial resolution. 
Geometrically, what it does is to detach the intersections between the local branches defined by $P_{I_{j}}(x,y, \mathbf{t}), j=1,\cdots, l$.
Over $\widetilde{\cs}_{I_{\bullet}}^{\circ}:=\widetilde{\cs}_{I_{\bullet}}\times_{{\cs}_{I_{\bullet}}}\cs_{I_{\bullet}}^{\circ}$, the geometric fibers of the family $\cc_{I_{\bullet}}^{+}\to \widetilde{\cs}_{I_{\bullet}}$ are projective irreducible curves with $h_{I_{\bullet}}$ ordinary double points, and the morphism $\tilde{\phi}^{+}_{I_{\bullet}}$ is just the simultaneous normalization of the restriction of ${\cc}^{+}_{I_{\bullet}}$ to $\widetilde{\cs}_{I_{\bullet}}^{\circ}$. 
Over $\tilde{0}\in \widetilde{\cs}_{I_{\bullet}}$, we get the partial resolution $\phi_{I_{\bullet}}:\widetilde{C}_{I_{\bullet}}\to C_{\gamma}$, which is an isomorphism over $C_{\gamma}\backslash \{c\}$, and such that $\phi_{I_{\bullet}}^{-1}(c)=\{c_{1},\cdots,c_{l}\}$ and that the singularity of $\widetilde{C}_{I_{\bullet}}$ at $c_{j}$ is isomorphic to $\widehat{C}_{I_{j}}$ for all $j$. 
Let $\tilde{\pi}^{+}_{I_{\bullet}}: \widetilde{\cc}^{+}_{I_{\bullet}}\to \widetilde{\cs}_{I_{\bullet}}$ be the structural morphism, it is a projective flat family of curves with generic geometric fibers being smooth projective of genus $\delta_{\gamma}-h_{I_{\bullet}}$.  
By theorem \ref{strata local}, the restriction of $\tilde{\pi}^{+}_{I_{\bullet}}$ to the formal neighbourhood $\widehat{\widetilde{\cs}}_{I_{\bullet}}$ is the product of versal deformations of the singularities of $\widetilde{C}_{I_{\bullet}}$.  
In other words, we have:

\begin{thm}\label{strata}

The family $\tilde{\pi}^{+}_{I_{\bullet}}: \widetilde{\cc}^{+}_{I_{\bullet}}\to \widetilde{\cs}_{I_{\bullet}}$ is an algebraization of a versal deformation of $\widetilde{C}_{I_{\bullet}}$.

\end{thm}

\begin{rem}

Following the same lines of reasoning as Teissier \cite{teissier resolution}, I-1.3.2, and with the fact that $\tilde{\phi}^{+}_{I_{\bullet}}$ is a finite flat morphism, one can show that $\widetilde{\cc}^{+}_{I_{\bullet}}$ coincides with the normalization of ${\cc}^{+}_{I_{\bullet}}$. 

\end{rem}

\section{The decomposition theorems}

Consider the families
$
f_{n}:\overline{\cp}_{n}\to \cb_{n}, (n,p)=1.
$
In this section, we will work out a decomposition theorem for the complex $Rf_{n,*}\ql$. Pass to the projective limit, we get a similar decomposition for the affine Springer fibers.

\subsection{The decomposition theorem for the finite abelian coverings}

By proposition \ref{smooth picard total}, the total space $\overline{\cp}_{n}$ is smooth over $k$. As $f_{n}$ is projective, the complex $Rf_{n,*}\ql$ is pure by Deligne's Weil-II \cite{weil2}. By the decomposition theorem of Beilinson, Bernstein, Deligne and Gabber \cite{bbdg}, we have
$$
Rf_{n,*}\ql=\bigoplus_{i=0}^{2\delta_{\gamma}} {}^{p}\ch^{i}(Rf_{n,*}\ql)[-i].
$$
We will make explicit the perverse sheaves ${}^{p}\ch^{i}(Rf_{n,*}\ql)$. 


By theorem \ref{pi 0 picard}, a geometric fiber of $f_{n}:\overline{\cp}_{n}\to \cb_{n}$ can have multiple irreducible components only when it lies over the union $\bigcup_{I_{\bullet}\vdash \{1,\cdots,r\}}\cs_{I_{\bullet}, n}$.  
Let $\pi_{I_{\bullet}}^{\circ}: \cc_{I_{\bullet}}^{\circ}:=\cc\times_{\cb}\cs_{I_{\bullet}}^{\circ}\to \cs_{I_{\bullet}}^{\circ}$ be the restriction of the family $\pi:\cc\to \cb$ to $\cs_{I_{\bullet}}^{\circ}$, it is a flat family of projective geometrically irreducible curves with $h_{I_{\bullet}}$ ordinary double points. 
Let $\tilde{\phi}^{\circ}_{I_{\bullet}}:\widetilde{\cc}_{I_{\bullet}}^{\circ}\to \cc_{I_{\bullet}}^{\circ}$ be the normalization, then the family $\tilde{\pi}_{I_{\bullet}}^{\circ}:\widetilde{\cc}_{I_{\bullet}}^{\circ}\to \cs_{I_{\bullet}}^{\circ}$ is a simultaneous resolution of singularities for the family $\pi_{I_{\bullet}}^{\circ}$. 
The pull-back of invertible sheaves along $\tilde{\phi}_{I_{\bullet}}^{\circ}$ defines a morphism
$$
(\tilde{\phi}_{I_{\bullet}}^{\circ})^{*}:\pic_{\cc^{\circ}_{I_{\bullet}}/\cs^{\circ}_{I_{\bullet}}}\to \pic_{\widetilde{\cc}^{\circ}_{I_{\bullet}}/\cs^{\circ}_{I_{\bullet}}},
$$
and this identifies $\pic^{0}_{\widetilde{\cc}^{\circ}_{I_{\bullet}}/\cs^{\circ}_{I_{\bullet}}}$ as the abelian factor of $\pic_{\cc^{\circ}_{I_{\bullet}}/\cs^{\circ}_{I_{\bullet}}}$. 
Let $\widetilde{\ct}^{\circ}_{n,I_{\bullet}}\subset \pic_{\widetilde{\cc}^{\circ}_{I_{\bullet}}/\cs^{\circ}_{I_{\bullet}}}[n]\times_{\cs^{\circ}_{I_{\bullet}}}\cs^{\circ}_{I_{\bullet},n}$ be the image under $(\tilde{\phi}_{I_{\bullet}}^{\circ})^{*}$ of the base change to $\cs^{\circ}_{I_{\bullet},n}$ of $\ct_{n}\subset \pic_{{\cc}/\cb}[n]\times_{\cb} \cb_{n}$. The finite group scheme $\widetilde{\ct}^{\circ}_{n, I_{\bullet}}$ defines a finite abelian covering 
\begin{equation*}
\widetilde{\Phi}^{\circ}_{n,I_{\bullet}}:\widetilde{\cp}^{\circ}_{I_{\bullet},n}\to  \pic^{0}_{\widetilde{\cc}^{\circ}_{I_{\bullet}}/\cs^{\circ}_{I_{\bullet}}}\times_{\cs^{\circ}_{I_{\bullet}}}\cs^{\circ}_{I_{\bullet},n}. 
\end{equation*}
Compose it with the structural morphism to $\cs^{\circ}_{I_{\bullet},n}$, we get an Abelian scheme
\begin{equation*}
\tilde{f}^{\circ}_{I_{\bullet},n}: \widetilde{\cp}^{\circ}_{I_{\bullet}, n}\to \cs^{\circ}_{I_{\bullet},n}.
\end{equation*}
Recall that we have defined a fiberwise dense open subfamily $f_{n}^{\circ}:\cp_{n}\to \cb_{n}$ of $f_{n}$ before proposition \ref{smooth picard total}, which is the $\Lambda^{0}/n$-covering of the base change to $\cb_{n}$ of the family $\cp\to \cb$ determined by $\ct_{n}$. As a corollary of theorem \ref{pi 0 picard} and the proof of theorem \ref{irreducible components}, we have:

\begin{prop}\label{picard si}

The irreducible components of the geometric fibers of $f_{n}:\overline{\cp}_{n}\to \cb_{n}$ over $\cs_{I_{\bullet},n}^{\circ}$ are parametrized by $\mathbb{V}_{I_{\bullet}}/n$. Moreover, the abelian factor of $\cp_{n}|_{\cs_{I_{\bullet},n}^{\circ}}$ can be identified with $\widetilde{\cp}^{\circ}_{I_{\bullet}, n}$.


\end{prop}

Let $\cf_{I_{\bullet},n}^{i}=R^{i}(\tilde{f}^{\circ}_{I_{\bullet},n})_{*}\ql$ and we simplify $\cf_{I_{\bullet}}^{i}=\cf_{I_{\bullet},1}^{i}$. As $\widetilde{\Phi}^{\circ}_{n,I_{\bullet}}$ is an isogeny of abelian schemes, we have actually $\cf_{I_{\bullet},n}^{i}=\varpi_{n}^{*}\cf_{I_{\bullet}}^{i}$. Moreover, we have $\cf_{I_{\bullet}}^{i}=\bigwedge^{i}\cf_{I_{\bullet}}^{1}$ and $\cf_{I_{\bullet}}^{1}\cong R^{1}\tilde{\pi}_{I_{\bullet},*}^{\circ}\ql$.

\begin{thm}\label{support variant}

For the family $f_{n}:\overline{\cp}_{n}\to \cb_{n}$, we have 
$$
Rf_{n,*}\ql=\bigoplus_{i=0}^{2\delta_{\gamma}} j_{n,!*}(R^{\,i}f^{\sm}_{n,*}\ql)[-i]
\oplus 
\bigoplus_{\substack{I_{\bullet} \text{ partition of}\\ \{1,\cdots,r\}}} 
\bigoplus_{i'=0}^{2(\delta_{\gamma}-h_{I_{\bullet}})}
\bigg\{(j_{I_{\bullet},n})_{!*}\cf_{I_{\bullet},n}^{i'}(-h_{I_{\bullet}})[-i'-2h_{I_{\bullet}}]\bigg\}^{\oplus\, |(\bbv_{I\bullet}/n)^{\circ}|},
$$
where $j_{n}:\cb_{n}^{\sm}\to \cb_{n}$ and $j_{I_{\bullet},n}:\cs_{I_{\bullet},n}^{\circ}\to \cs_{I_{\bullet},n}$ are the natural inclusions.

\end{thm}

\begin{proof}

We follow the strategy of Ng\^o, as briefly recalled in \S\ref{old support}.
It is clear that the action $\cp_{n}\times_{\cb_{n}} \overline{\cp}_{n}\to \overline{\cp}_{n}$ forms a weak abelian fibration. 
For each point $s_{n}\in \cb_{n}$, let $\delta_{s_{n}}^{\aff}$ be the dimension of the maximal connected affine sub group scheme of $\cp_{n,\,\bar{s}_{n}}$. By theorem \ref{ngo support main}, any support $s_{n}$ of $Rf_{n,*}\ql$ satisfies the inequality
\begin{equation}\label{ngo inequality}
\delta_{\gamma}-\delta_{s_{n}}^{\aff}\le \frac{1}{2}\amp_{s_{n}}(Rf_{n,*}\ql)\le \delta_{\gamma}-\codim_{\cb_{n}}(\{s_{n}\}).
\end{equation}
Let $s=\varpi_{n}(s_{n})$. As $\cp_{n}\to \cp\times_{\cb}\cb_{n}$ is finite \'etale, we have
$$
\delta_{s_{n}}^{\aff}=\delta(\cc_{s}). 
$$
On the other hand, we have the inequality of Severi
\begin{equation}\label{severi n}
\codim_{\cb}(\{s\})\ge \delta(\cc_{s}).
\end{equation}
It is clear that $\codim_{\cb_{n}}(\{s_{n}\})=\codim_{\cb}(\{s\})$.
The inequalities (\ref{ngo inequality}) and (\ref{severi n}) imply that both must be equality. Hence the extra support $\{s_{n}\}$ of $Rf_{n,*}\ql$ must be among the preimages in $\cb_{n}$ of the strict $\delta$-strata in $\cb$. 
Moreover, any extra local system over $s_{n}$ appears in the form $\boldsymbol{\Lambda}_{s_{n}}^{\bullet}\otimes \cl$, $\cl$ of degree $0$, such that the top degree piece $\boldsymbol{\Lambda}_{s_{n}}^{\rm top}\otimes \cl$ appears as direct summands of $(R^{2\delta_{\gamma}}f_{n,*}\ql)_{s_{n}}$, where $\boldsymbol{\Lambda}_{s_{n}}^{\bullet}$ is the cohomological ring of the abelian part of $\cp_{n,\bar{s}_{n}}$.

By theorem \ref{pi 0 picard}, for the geometric fibers of $f_{n}$ to have multiple irreducible components over $\overline{\{s_{n}\}}$, we need $s_{n}$ to lie in the union $\bigcup_{I_{\bullet}} \cs_{I_{\bullet},n}$. The extra constraint $\delta(\cc_{s})=\codim_{\cb}(\{s\})$ implies that $s_{n}$ must be $\cs_{I_{\bullet},n}^{\circ}$ for some partition $I_{\bullet}$. Moreover, we have $(R^{2\delta_{\gamma}}f_{n,*}\ql)_{\bar{s}_{n}}=\ql(-\delta_{\gamma})^{\bbv_{I_{\bullet}}/n}$ by the same theorem.
We claim that $(R^{2\delta_{\gamma}}f_{n,*}\ql)|_{\cs_{I_{\bullet}, n}^{\circ}}$ is constant. 
By theorem \ref{irreducible components}, the irreducible components of $\overline{\cp}_{n,s_{n}}$ are naturally parametrized by $\Ker(\phi_{\bar{s}_{n}}^{*}|_{\ct_{n,s_{n}}})$, where $\phi_{\bar{s}_{n}}:\widetilde{\cc}_{\bar{s}_{n}}\to {\cc}_{\bar{s}_{n}}$ is the normalization of ${\cc}_{\bar{s}_{n}}$ and $\phi_{\bar{s}_{n}}^{*}$ is the induced morphism between their Jacobians.
Compare with the parametrization of irreducible components of $\overline{\cp}_{n,0}$, as $\ct_{n}$ is constant over $\cb_{n}$, we get an {injection} from the set of irreducible components of $\overline{\cp}_{n,s_{n}}$ into that of $\overline{\cp}_{n,0}$. Hence the sheaf $(R^{2\delta_{\gamma}}f_{n,*}\ql)|_{\cs_{I_{\bullet}, n}^{0}}$ can not have non-trivial monodromy and has to be constant, so
$$
(R^{2\delta_{\gamma}}f_{n,*}\ql)|_{\cs_{I_{\bullet}, n}^{\circ}}=\ql(-\delta_{\gamma})^{\bbv_{I_{\bullet}}/n}.
$$

By proposition \ref{picard si}, the abelian factor of $\cp_{n}|_{\cs_{I_{\bullet},n}^{\circ}}$ can be identified with $\widetilde{\cp}^{\circ}_{I_{\bullet}, n}$. Hence we can identify $\boldsymbol{\Lambda}_{s_{n}}^{\bullet}$ with the fiber $\big(\bigoplus_{i}\cf^{i}_{I_{\bullet}, n}\big)_{s_{n}}$. 
Now that $\boldsymbol{\Lambda}_{s_{n}}^{\rm top}\otimes \cl$ appears as direct summands of $(R^{2\delta_{\gamma}}f_{n,*}\ql)_{s_{n}}$, and that $\boldsymbol{\Lambda}_{s_{n}}^{\rm top}=\cf^{2(\delta_{\gamma}-h_{I_{\bullet}})}_{I_{\bullet}, n}=\ql(-\delta_{\gamma}+h_{I_{\bullet}})$, the local system $\cl$ has to be $\ql(-h_{I_{\bullet}})^{\oplus \,m}$ for some multiplicity $m$. 
The multiplicity can be calculated by subtracting from $\ql(-\delta_{\gamma})^{\bbv_{I_{\bullet}}/n}|_{\cs^{\circ}_{I_{\bullet},n}}$  the contributions from the intermediate extensions of similar local systems over the stratum $\cs^{\circ}_{I_{\bullet}'}$ having $\cs^{\circ}_{I_{\bullet}}$ in the closure. 
The stratum closure relation has been explained in remark \ref{strata closure}, and the multiplicity $m$ can be easily calculated to be $|(\bbv_{I\bullet}/n)^{\circ}|$.

\end{proof}

We call the first summand in the theorem the \emph{main term} for the decomposition of $Rf_{n,*}\ql$. 
We will show that the remaining terms are of the same nature as that of the main term.

We use the notations of \S\ref{reduction over si}. 
Let $I_{\bullet}$ be a non-trivial partition of $\{1,\cdots, r\}$ of length $l$. 
Recall that for the normalization $\varphi_{I_{\bullet}}:\widetilde{\cs}_{I_{\bullet}}\to \cs_{I_{\bullet}}$  we have constructed a simultaneous partial resolution $\tilde{\phi}_{I_{\bullet}}^{+}: \widetilde{\cc}^{+}_{I_{\bullet}}\to {\cc}^{+}_{I_{\bullet}}$ and get a flat family of projective curves $\tilde{\pi}^{+}_{I_{\bullet}}: \widetilde{\cc}^{+}_{I_{\bullet}}\to \widetilde{\cs}_{I_{\bullet}}$, which is an algebraization of a versal deformation of the partial resolution $\widetilde{C}_{I_{\bullet}}$ by theorem \ref{strata}. 
The restriction of $\varphi_{I_{\bullet}}$ to $\widetilde{\cs}_{I_{\bullet}}^{\circ}$ is an isomorphism, and with it we can identify  the restriction of the family $\tilde{\pi}_{I_{\bullet}}^{+}$ to $\widetilde{\cs}_{I_{\bullet}}^{\circ}$ with the family $\tilde{\pi}_{I_{\bullet}}^{\circ}$.

Let $\tilde{f}_{I_{\bullet}}^{+}:\widetilde{\cp}_{I_{\bullet}}^{+}\to \widetilde{\cs}_{I_{\bullet}}$ be the relative compactified Jacobian of the family $\tilde{\pi}^{+}_{I_{\bullet}}: \widetilde{\cc}^{+}_{I_{\bullet}}\to \widetilde{\cs}_{I_{\bullet}}$. The partial resolution $\tilde{\phi}_{I_{\bullet}}^{+}: \widetilde{\cc}^{+}_{I_{\bullet}}\to {\cc}^{+}_{I_{\bullet}}$ induces a morphism
$$
(\tilde{\phi}_{I_{\bullet}}^{+})^{*}:\pic_{\cc^{+}_{I_{\bullet}}/\widetilde{\cs}_{I_{\bullet}}}\to \pic_{\widetilde{\cc}^{+}_{I_{\bullet}}/\widetilde{\cs}_{I_{\bullet}}}.
$$ 
Let $\widetilde{\cs}_{I_{\bullet},n}=\widetilde{\cs}_{I_{\bullet}}\times_{\cs_{I_{\bullet}}}\cs_{I_{\bullet}, n}$. 
Note that $\widetilde{\cs}_{I_{\bullet},n}^{\circ}:=\widetilde{\cs}_{I_{\bullet},n}\times_{\cs_{I_{\bullet}, n}}\cs_{I_{\bullet},n}^{\circ}$ coincides with $\cs_{I_{\bullet},n}^{\circ}$ and $\widetilde{\cs}_{I_{\bullet},n}$ is the normalization of ${\cs}_{I_{\bullet},n}$ as $\cs_{I_{\bullet}, n}$ is \'etatle over $\cs_{I_{\bullet}}$.
Let
$$\widetilde{\ct}^{+}_{n,I_{\bullet}}\subset \pic_{\widetilde{\cc}^{+}_{I_{\bullet}}/\widetilde{\cs}_{I_{\bullet}}}[n]\times_{\widetilde{\cs}_{I_{\bullet}}}\widetilde{\cs}_{I_{\bullet},n}
$$
be the image under $(\tilde{\phi}_{I_{\bullet}}^{+})^{*}$ of the base change to $\widetilde{\cs}_{I_{\bullet},n}$ of $\ct_{n}\subset \pic_{{\cc}/\cb}[n]\times_{\cb} \cb_{n}$. 
Then $\widetilde{\ct}^{+}_{n,I_{\bullet}}$ defines a finite abelian covering 
\begin{equation*}
\widetilde{\Phi}^{+}_{n,I_{\bullet}}:\widetilde{\cp}^{+}_{I_{\bullet},n}\to  \widetilde{\cp}^{+}_{I_{\bullet}}\times_{\widetilde{\cs}_{I_{\bullet}}}\widetilde{\cs}_{I_{\bullet},n}. 
\end{equation*}
Compose it with the structural morphism to $\widetilde{\cs}_{I_{\bullet},n}$, we get a family
\begin{equation*}
\tilde{f}^{+}_{I_{\bullet},n}: \widetilde{\cp}^{+}_{I_{\bullet}, n}\to \widetilde{\cs}_{I_{\bullet},n}.
\end{equation*}
By construction, the restriction of $\widetilde{\ct}^{+}_{n,I_{\bullet}}$ to $\widetilde{\cs}_{I_{\bullet},n}^{\circ}$ coincides with $\widetilde{\ct}^{\circ}_{n,I_{\bullet}}$, and the restriction of $\tilde{f}^{+}_{I_{\bullet},n}$ to $\widetilde{\cs}_{I_{\bullet},n}^{\circ}$ coincides with  $\tilde{f}^{\circ}_{I_{\bullet},n}$.

By theorem \ref{strata}, $\tilde{\pi}^{+}_{I_{\bullet}}: \widetilde{\cc}^{+}_{I_{\bullet}}\to \widetilde{\cs}_{I_{\bullet}}$ is an algebraization of a versal deformation of $\widetilde{C}_{I_{\bullet}}$. 
With the same lines of reasoning as theorem \ref{support variant}, we get a decomposition theorem for the family $\tilde{f}^{+}_{I_{\bullet},n}$ in the form
\begin{equation}\label{decomposition si tilde}
R\big(\tilde{f}^{+}_{I_{\bullet},n}\big)_{*}\ql=\bigoplus_{i=0}^{2(\delta_{\gamma}-h_{I_{\bullet}})} \Big\{\big(\tilde{j}^{+}_{I_{\bullet},n}\big)_{!*}R^{\,i}\big(\tilde{f}^{+, \sm}_{I_{\bullet},n}\big)_{*}\ql\Big\}[-i]
\oplus 
\bigoplus \text{\it Remaining terms},
\end{equation}
where $\tilde{j}^{+}_{I_{\bullet},n}: \widetilde{\cs}_{I_{\bullet},n}^{\circ}\to \widetilde{\cs}_{I_{\bullet},n}$ is the inclusion. As the restriction of $\tilde{f}^{+}_{I_{\bullet},n}$ to $\widetilde{\cs}_{I_{\bullet},n}^{\circ}$ coincides with  $\tilde{f}^{\circ}_{I_{\bullet},n}$, we can identify
$$
R^{\,i}\big(\tilde{f}^{+, \sm}_{I_{\bullet},n}\big)_{*}\ql=\cf_{I_{\bullet},n}^{i}.
$$
Let ${\varphi}_{I_{\bullet},n}: \widetilde{\cs}_{I_{\bullet},n}\to \cs_{I_{\bullet}, n}$ be the structural morphism, it is a finite morphism as it is the base change of the normalization $\varphi_{I_{\bullet}}:\widetilde{\cs}_{I_{\bullet}}\to  \cs_{I_{\bullet}}$. 
Hence we have
\begin{equation}\label{push down along normalization}
(j_{I_{\bullet}, n})_{!*}\cf_{I_{\bullet},n}^{i}=\big({\varphi}_{I_{\bullet},n}\big)_{*} 
\big(\tilde{j}^{+}_{I_{\bullet},n}\big)_{!*}R^{\,i}\big(\tilde{f}^{+, \sm}_{I_{\bullet},n}\big)_{*}\ql
\end{equation}

On the other hand, we can exploit the fact that the curve $\widetilde{C}_{I_{\bullet}}$ is projective rational with singularities at $c_{1},\cdots,c_{l}$, such that $\widehat{\co}_{\widetilde{C}_{I_{\bullet}},c_{j}}\cong \co[\gamma_{I_{j}}]$ for all $j$, to get a factorization of the decomposition (\ref{decomposition si tilde}). 
For $I\subsetneq \{1,\cdots,r\}$, let $C_{\gamma_{I}}$ be the spectral curve of $\gamma_{I}$, let $A_{I}=\co[\gamma_{I}]$ and $E_{I}$ the field of fractions of $A_{I}$.  
Let $\pi_{I}:\cc_{I}\to \cb_{I}$ be an algebraization of a miniversal deformation of $C_{\gamma_{I}}$, and let $f_{I}:\overline{\cp}_{I}\to \cb_{I}$ be its relative compactified Jacobian.
Let $j_{I}:\cb_{I}^{\rm sm}\to \cb_{I}$ be the inclusion and let $\cf_{I}^{i}=R^{i}f_{I,*}^{\sm}\ql$. 
We define $f_{I,n}:\overline{\cp}_{I,n}\to \cb_{I,n}$ and denote $j_{I,n}:\cb_{I,n}^{\rm sm}\to \cb_{I,n}$ as before, let $\cf_{I,n}^{i}=R^{i}(f_{I,n}^{\sm})_{*}\ql$. 
Let $M_{I_{\bullet}}$ be the Levi subgroup of $\gl_{d}$ defined by 
\begin{equation*}
M_{I\bullet}=\gl_{I_{1}}\times \cdots \times \gl_{I_{l}}=\prod_{j=1}^{l}\gl(E_{I_{j}}).
\end{equation*}
It is clear that $\gamma$ belongs to the Lie algebra of $M_{I\bullet}$, and we have the affine Springer fiber 
\begin{equation*} 
\xx_{\gamma}^{I_{\bullet}}:=\xx_{\gamma}^{M_{I_{\bullet}}}=\xx_{\gamma_{I_{1}}}\times\cdots\times \xx_{\gamma_{I_{l}}}.
\end{equation*}
Let $\Lambda_{I_{j}}^{0}$ be the free discret abelian group acting on $\xx_{\gamma_{I_{j}}}^{0}$. 
As explained in \S\ref{review laumon}, we have the radical finite surjective morphisms
$$
\Lambda_{I_{j}}^{0}\backslash \xx_{\gamma_{I_{j}}}^{0}\to \overline{P}_{C_{\gamma_{I_{j}}}}, \quad  j=1,\cdots, l,$$
and
\begin{equation*}
\big(\Lambda_{I_{1}}^{0}\backslash \xx_{\gamma_{I_{1}}}^{0}\big)\times\cdots\times \big(\Lambda_{I_{l}}^{0}\backslash \xx_{\gamma_{I_{l}}}^{0}\big)\to \overline{P}_{\widetilde{C}_{I_{\bullet}}}.
\end{equation*}
Hence we get the factorization
\begin{equation}\label{factorize cj si}
\overline{P}_{\widetilde{C}_{I_{\bullet}}} \cong \overline{P}_{C_{\gamma_{I_{1}}}}\times\cdots\times \overline{P}_{C_{\gamma_{I_{l}}}}.
\end{equation}
This factorization property is compatible with the local-global property of the deformation theory of curves and their compactified Jacobians.
Hence the formal smooth morphism $\widehat{\widetilde{\cs}}_{I_{\bullet}}\to \prod_{j=1}^{l}\defm^{\rm top}_{\widehat{C}_{I_{j}}}$ of theorem \ref{strata local} is compatible with the above factorization of $\overline{P}_{\widetilde{C}_{I_{\bullet}}}$. This implies the factorization
$$
\big(R\tilde{f}^{+}_{I_{\bullet}, *}\ql\big)_{\tilde{0}}=(R{f}_{I_{1},*}\ql)_{0}\otimes\cdots\otimes (R{f}_{I_{l},*}\ql)_{0}.
$$
The factorization (\ref{factorize cj si}), restricted to the Jacobian of the curve, induces a factorization of the finite group scheme $\widetilde{\ct}^{+}_{n,I_{\bullet}}$. Each factor determines a $\Lambda_{I_{j}}^{0}/n$-covering $f_{I_{j},n}: \overline{\cp}_{I_{j}, n}\to \cb_{I_{j},n}$ of the family $f_{I_{j}}: \overline{\cp}_{I_{j}}\to \cb_{I_{j}}$. 
Hence we get the factorization
$$
\big(R(\tilde{f}^{+}_{I_{\bullet}, n})_{*}\ql\big)_{\tilde{0}_{n}}=\big(R({f}_{I_{1},n})_{*}\ql\big)_{0_{n}}\otimes\cdots\otimes \big(R({f}_{I_{l},n}\big)_{*}\ql)_{0_{n}}.
$$
Compare the main term of the decomposition (\ref{decomposition si tilde}) and those of the decomposition of $R({f}_{I_{j},n}\big)_{*}\ql$ given by theorem \ref{support variant}, we get the factorization
$$
\Big\{\big(\tilde{j}^{+}_{I_{\bullet},n}\big)_{!*}R\big(\tilde{f}^{+, \sm}_{I_{\bullet},n}\big)_{*}\ql\Big\}_{\tilde{0}_{n}}
=\big((j_{I_{1}, n})_{!*} R(f^{\rm sm}_{I_{1},n})_{*}\ql\big)_{0_{n}}\otimes \cdots \otimes \big((j_{I_{l}, n})_{!*} R(f^{\rm sm}_{I_{l},n})_{*}\ql\big)_{0_{n}}.
$$
With the equation (\ref{push down along normalization}) and the fact that there is a unique point $\tilde{0}\in \widetilde{\cs}_{I_{\bullet}}$ lying over $0\in {\cs}_{I_{\bullet}}$, we conclude:

\begin{thm}\label{reduction theorem}
With the above notations, we have 
$$
((j_{I_{\bullet}, n})_{!*}\cf_{I_{\bullet},n}^{i})_{0_{n}}=\bigoplus_{i_{1}+\cdots+i_{l}=i}
\big((j_{I_{1}, n})_{!*} \cf^{i_{1}}_{I_{1},n}\big)_{0_{n}}\otimes \cdots \otimes \big((j_{I_{l}, n})_{!*} \cf^{i_{l}}_{I_{l},n}\big)_{0_{n}}.
$$
\end{thm}

The decomposition theorem \ref{support variant} and the reduction theorem \ref{reduction theorem} are very similar to the Arthur-Kottwitz reduction of the affine Springer fiber $\xx_{\gamma}$. 
Recall that we  \cite{chen fundamental} have introduced a \emph{fundamental domain} $F_{\gamma}$ of $\xx_{\gamma}^{0}$, and associate to it a decomposition as disjoint union of locally closed subschemes
\begin{equation}\label{arthur-kottwitz}
\xx_{\gamma}^{0}=F_{\gamma}\cup \bigcup_{\substack{Q\in \cf(M_0)\\ Q\neq G}}\bigcup_{\nu\in \Lambda^{0}_{M_{Q}}\cap R_{Q}} S_{Q}^{\nu}.
\end{equation}
Moreover, for each locally closed subscheme $S_{Q}^{\nu}$, the parabolic reduction of the affine Grassmannian
$$
f_{Q}:\xx\to \xx^{M_{Q}}
$$
induces an iterated affine fibration
$$
f_{Q}^{\nu}: S_{Q}^{\nu}\to F_{\gamma}^{M_{Q},\nu},
$$
where $F_{\gamma}^{M_{Q},\nu}$ is a variant of the fundamental domain $F_{\gamma}^{M_{Q}}$ for $\xx_{\gamma}^{M_{Q}}$, they are very close to each other but slightly different. 
The relative dimension of $f_{Q}^{\nu}$ is
$$
d_{M_{Q}}:=\sum_{\alpha\in \Phi(N_{Q},T)}\val(\alpha(\gamma)).
$$
This is the so-called \emph{Arthur-Kottwitz} reduction. As $f_{Q}^{\nu}$ is an iterated affine fibration of relative dimension $d_{M_{Q}}$, we have 
$$
H_{c}^{*}(S_{Q}^{\nu},\ql)\cong H^{*}(F_{\gamma}^{M_{Q},\nu},\ql)(-d_{M_{Q}})[-2d_{M_{Q}}]. 
$$
Then the decomposition (\ref{arthur-kottwitz}) implies 
\begin{equation}\label{ak cohomology}
H^{*}(\xx_{\gamma}^{0},\ql)=H^{*}(F_{\gamma}, \ql)\oplus \bigoplus_{\substack{Q\in \cf(M_0)\\ Q\neq G}}\bigoplus_{\nu\in \Lambda^{0}_{M_{Q}}\cap R_{Q}} H^{*}\big(F_{\gamma}^{M_{Q},\nu},\ql\big)(-d_{M_{Q}})[-2d_{M_{Q}}].
\end{equation}
which resembles the decomposition in theorem \ref{support variant} and the reduction in theorem \ref{reduction theorem}. 

Despite this structural analogy, the decomposition in theorem \ref{support variant} is not a shadow of the Arthur-Kottwitz reduction, it is of different nature. 
Recall that we have the diagram (\ref{cd picard}), and that the Galois group of the finite abelian covering 
$
\Phi_{n}:\overline{\cp}_{n}\to \cp\times_{\cb}\cb_{n}
$
is $\Lambda^{0}/n$, hence the group $\Lambda^{0}/n$ acts on the complex $Rf_{n,*}\ql$.

\begin{prop}\label{main term n}

The main term $(Rf_{n,*}\ql)_{\rm main}$ is the unique summand in the decomposition of   $Rf_{n,*}\ql$ which is invariant under the action of $\Lambda^{0}/n$. 

\end{prop}

\begin{proof}

For the action of $\Lambda^{0}/n$ on the main term, note that the restriction of the family $f:\overline{\cp}\to \cb$ to $\cb^{\sm}$ coincides with the family of Jacobians of curves
$$
f^{\sm}:\pic^{0}_{\cc^{\sm}/\cb^{\sm}}\to \cb^{\sm},
$$
and the finite abelian covering $\Phi_{n}:{\cp}_{n}^{\sm}\to \pic^{0}_{\cc^{\sm}/\cb^{\sm}}\times \cb_{n}^{\sm}$ is an isogeny of abelian scheme with kernel isomorphic to the constant finite group scheme $\ct_{n}\subset {\cp}_{n}^{\sm}$.
The group $\ct_{n}$ acts on $\cp_{n}^{\sm}$ by translation, and the induced action on $Rf_{n,*}^{\sm}\ql$ coincides with that of  $\Lambda^{0}/n$.
It is clear that the action of $\ct_{n}$ on $\cp_{n}^{\sm}$ extends to the action of $\cp_{n}^{\sm}$ on itself. Hence the action of $\ct_{n}$ on $Rf_{n,*}^{\sm}\ql$ extends also to the action of $\cp_{n}^{\sm}$. By a well-known homotopy lemma (cf. \cite{laumon ngo}, Lem. 3.2.3), the latter action must be trivial as  $f_{n}^{\sm}: \cp_{n}^{\sm}\to \cb_{n}^{\sm}$ is a smooth group scheme with geometrically connected fibers. As the intermediate extension is functorial, the action of $\Lambda^{0}/n$ on $j_{n,!*}(Rf_{n,*}^{\sm}\ql)$ must be trivial as well.

For the action of $\Lambda^{0}/n$ on the remaining terms in theorem \ref{support variant}, note that their restrictions to $\cs_{I_{\bullet}}^{\circ}$ appear as a submodule of $Rf_{n,*}\ql|_{\cs^{\circ}_{I_{\bullet},n}}$. 
The same argument as above with the homotopy lemma implies that the action of $\Lambda^{0}/n$ on $Rf_{n,*}\ql|_{\cs^{\circ}_{I_{\bullet},n}}$ factorizes through the group $\pi_{0}(\cp_{n}|_{\cs^{\circ}_{I_{\bullet},n}})$, and it acts by permuting the irreducible components of $\overline{\cp}_{n}|_{\cs^{\circ}_{I_{\bullet},n}}$. As the irreducible components of the geometric fibers of $\cp_{n}$ and $\overline{\cp}_{n}$ are the same, it is  clear that the trivial action appears with multiplicity 1 and it comes from the action on $j_{n,!*}(Rf_{n,*}^{\sm}\ql)|_{\cs^{\circ}_{I_{\bullet},n}}$, hence the action on the other summands of $Rf_{n,*}\ql|_{\cs^{\circ}_{I_{\bullet},n}}$ must be non-trivial. 
This finishes the proof of the proposition.

\end{proof}

\subsection{Consequence for the affine Springer fibers}

In this section, we restrict to the case $k=\bc$ as we will use homology instead of cohomology. The reasoning can be easily extended to the finite field case by defining the $\ell$-adic homology formally as the dual of $\ell$-adic cohomology.

For $m| n$, we have finite abelian covering $n\Lambda^{0}\backslash \xx_{\gamma}^{0} \to m\Lambda^{0}\backslash \xx_{\gamma}^{0}$, which induces natural morphisms
\begin{equation}\label{transition mn}
H_{*}\big(n\Lambda^{0}\backslash \xx_{\gamma}^{0},\bz\big)\to H_{*}\big(m\Lambda^{0}\backslash \xx_{\gamma}^{0},\bz\big).
\end{equation}
As $n$ varies, we get a projective system. 

\begin{prop}\label{coh as limit}

The homology of $\xx_{\gamma}^{0}$ can be recovered from that of the quotients $n\Lambda^{0}\backslash \xx_{\gamma}^{0}$ by
\begin{equation*}
H_{*}\big({\xx}^{0}_{\gamma},\bz\big)=\varprojlim_{n}  H_{*}\big(n\Lambda^{0}\backslash \xx_{\gamma}^{0},\bz\big).
\end{equation*}

\end{prop}

\begin{proof}

By construction, the singular chains defining the homology of a topological space are finite sums of singular simplices. In particular, they are compactly supported. Hence the natural projection $\xx_{\gamma}^{0} \to n\Lambda^{0}\backslash \xx_{\gamma}^{0}$ induces morphism of homology
$$
H_{*}\big(\xx_{\gamma}^{0},\bz\big)\to H_{*}\big(\Lambda^{0}\backslash \xx_{\gamma}^{0},\bz\big).
$$
It is clearly compatible with the transition map (\ref{transition mn}), hence we get a morphism
\begin{equation}\label{transition}
H_{*}\big(\xx_{\gamma}^{0},\bz\big)\to \varprojlim_{n} H_{*}\big(n\Lambda^{0}\backslash \xx_{\gamma}^{0},\bz\big).
\end{equation}
As the singular chains are compactly supported, any chain for $\xx_{\gamma}^{0}$ together with its boundary will be mapped identically into $n\Lambda^{0}\backslash \xx_{\gamma}^{0}$ for $n$ sufficiently large. Hence the morphism (\ref{transition}) must be an injection. On the other hand, for $(c_{n})_{n}\in \varprojlim_{n} H_{*}\big(n\Lambda^{0}\backslash \xx_{\gamma}^{0},\bz\big)$, as we are working with $\bz$-coefficient (in particular the coefficients are not infinitely divisible), for $n$ sufficiently large the chains representing $c_{n}$ should be stabilized, i.e. for all $n|n'$ a chain representing $c_{n'}$ should be a lift of a chain for $c_{n}$ . This implies that the chains actually come from a chain in $\xx_{\gamma}^{0}$, hence the morphism (\ref{transition}) is surjective and so is an isomorphism.

\end{proof}


\begin{rem}\label{no zl}

Let $A$ be a finitely generated abelian group, we denote by $A_{\rm fr}$ the free summand of $A$. The above proposition implies that
$$
H_{*}\big({\xx}^{0}_{\gamma},\bz\big)_{\rm fr}=\varprojlim_{n}  H_{*}\big(n\Lambda^{0}\backslash \xx_{\gamma}^{0},\bz\big)_{\rm fr}. 
$$
Indeed, with the Arthur-Kottwitz reduction we get an expression for 
$
H_{*}\big({\xx}^{0}_{\gamma},\bz\big)$
similar to (\ref{ak cohomology}). As all the $H_{*}\big(F^{M}_{\gamma},\bz\big)$ are finitely generated abelian groups, $
H_{*}\big({\xx}^{0}_{\gamma},\bz\big)$ doesn't contain any summand of the form $\bz_{l}$. Hence the projective limit of the torsion summand of $H_{*}\big(n\Lambda^{0}\backslash \xx_{\gamma}^{0},\bz\big)$ remains torsion, it doesn't contribute to the free summand $H_{*}\big({\xx}^{0}_{\gamma},\bz\big)_{\rm fr}$.

\end{rem}

Let $\bbf_{n}=R^{1}f_{n,*}^{\sm}\bz$ and $\bbf_{n}^{\vee}=\ch om(\bbf_{n}, \bz)$. 
For a non-trivial partition $I_{\bullet}\vdash \{1,\cdots, r\}$, let $\bbf_{I_{\bullet}, n}=R^{1}(\tilde{f}_{I_{\bullet},n}^{\,\rm sm})_{*}\bz$ and $\bbf_{I_{\bullet}, n}^{\vee}=\ch om(\bbf_{I_{\bullet},n}, \bz)$. 
For a subset $I\subsetneq \{1,\cdots,r\}$, we have introduced the families $f_{I}:\overline{\cp}_{I}\to \cb_{I}$ and $f_{I, n}:\overline{\cp}_{I,n}\to \cb_{I,n}$, let $\tau_{I}=\dim(\cb_{I})$ and let $\bbf_{I, n}=R^{1}({f}_{I,n}^{\,\rm sm})_{*}\bz$ and $\bbf_{I, n}^{\vee}=\ch om(\bbf_{I,n}, \bz)$.

\begin{prop}\label{decomposition 0}

For any $i\in \bn_{0}$ we have
\begin{align*}
H_{i}\big({\xx}_{\gamma}^{0},\bz\big)_{\rm fr}=&
\varprojlim_{n}\bigoplus_{i'=0}^{i} 
\ch^{2\tau_{\gamma}-(i-i')}_{\{0_{n}\}}\big(j_{n,!*}\BigWedge^{i'}\bbf_{n}^{\vee}\big)_{\rm fr}
\oplus \bigoplus_{\substack{I_{\bullet} \text{ partition of}\\ \{1,\cdots,r\}}} 
\bigoplus_{i''=0}^{i-2h_{I_{\bullet}}} 
\bigg\{\bigoplus_{i_{1}+\cdots+i_{l}=i''} \bigoplus_{j_{1}+\cdots+j_{l}=i-2h_{I_{\bullet}}-i''}  \nonumber
\\
&
\varprojlim_{n}\ch^{2\tau_{I_{1}}-j_{1}}_{\{0_{n}\}}\big((j_{I_{1}, n})_{!*} \cf^{i_{1}}_{I_{1},n}\big)_{\rm fr}\otimes \cdots \otimes \ch^{2\tau_{I_{l}}-j_{l}}_{\{0_{n}\}}\big((j_{I_{l}, n})_{!*} \cf^{i_{l}}_{I_{l},n}\big)_{\rm fr}(h_{I_{\bullet}})\bigg\}^{\oplus\, \bbv^{\circ}_{I_{\bullet}}},
\end{align*}
where $l$ is the length of the partition $I_{\bullet}$, and the superscript $\oplus\, \bbv^{\circ}_{I_{\bullet}}$ signifies that it is a direct sum of copies indexed by $\bbv^{\circ}_{I_{\bullet}}$.

\end{prop}

\begin{proof}

For finitely generated abelian groups $A_{1}$ and $A_{2}$, we denote $A_{1}\doteq A_{2}$ if the free summand of $A_{1}$ and $A_{2}$ are equal. 
By theorem \ref{support variant} and \ref{reduction theorem}, we have 
{\small
\begin{align*}
H^{i}\big(n\Lambda^{0}\backslash{\xx}_{\gamma}^{0},\bz\big)\doteq &
\bigoplus_{i'=0}^{i} 
\ch^{i-i'}(j_{n,!*}(R^{\,i'}f^{\sm}_{n,*}\bz)_{0_{n}})
\oplus \bigoplus_{\substack{I_{\bullet} \text{ partition of}\\ \{1,\cdots,r\}}} 
\bigoplus_{i''=0}^{i-2h_{I_{\bullet}}} 
\\
&
\bigg\{\ch^{i-2h_{I_{\bullet}}-i''}\big((j_{I_{\bullet},n})_{!*}
\big(R^{i''}(\tilde{f}_{I_{\bullet},n}^{\,\rm sm})_{*}\bz\big)_{0_{n}}\big)(-h_{I_{\bullet}})\bigg\}^{\oplus\, |(\bbv_{I\bullet}/n)^{\circ}|}
\\
=&\bigoplus_{i'=0}^{i} 
\ch^{i-i'}(j_{n,!*}(R^{\,i'}f^{\sm}_{n,*}\bz)_{0_{n}})
\oplus \bigoplus_{\substack{I_{\bullet} \text{ partition of}\\ \{1,\cdots,r\}}} 
\bigoplus_{i''=0}^{i-2h_{I_{\bullet}}} 
\\
&
\bigg\{\ch^{i-2h_{I_{\bullet}}-i''}\bigg(\bigoplus_{i_{1}+\cdots+i_{l}=i''}
\big((j_{I_{1}, n})_{!*} \cf^{i_{1}}_{I_{1},n}\big)_{0_{n}}\otimes \cdots \otimes \big((j_{I_{l}, n})_{!*} \cf^{i_{l}}_{I_{l},n}\big)_{0_{n}}\bigg)(-h_{I_{\bullet}})\bigg\}^{\oplus\, |(\bbv_{I\bullet}/n)^{\circ}|}
\\
=&\bigoplus_{i'=0}^{i} 
\ch^{i-i'}(j_{n,!*}(R^{\,i'}f^{\sm}_{n,*}\bz)_{0_{n}})
\oplus \bigoplus_{\substack{I_{\bullet} \text{ partition of}\\ \{1,\cdots,r\}}} 
\bigoplus_{i''=0}^{i-2h_{I_{\bullet}}} 
\bigg\{\bigoplus_{i_{1}+\cdots+i_{l}=i''} \bigoplus_{j_{1}+\cdots+j_{l}=i-2h_{I_{\bullet}}-i''}
\\
&
\ch^{j_{1}}\big((j_{I_{1}, n})_{!*} \cf^{i_{1}}_{I_{1},n}\big)_{0_{n}}\otimes \cdots \otimes \ch^{j_{l}}\big((j_{I_{l}, n})_{!*} \cf^{i_{l}}_{I_{l},n}\big)_{0_{n}}(-h_{I_{\bullet}})\bigg\}^{\oplus\, |(\bbv_{I\bullet}/n)^{\circ}|}.
\end{align*}
}

\noindent Applying $\Hom(\bullet,\bz)$, we get an expression of $H_{i}\big(n\Lambda^{0}\backslash{\xx}_{\gamma}^{0},\bz\big)$ modulo torsion. Note that for any finitely generated abelian group $A$ the group $\Hom(A,\bz)$ is alway \emph{free} and finitely generated. Hence we actually get an expression of the free summand
\begin{align}
&H_{i}\big(n\Lambda^{0}\backslash{\xx}_{\gamma}^{0},\bz\big)_{\rm fr} \nonumber
\\
=&\bigoplus_{i'=0}^{i} 
\Hom\big(\ch^{i-i'}(j_{n,!*}(R^{\,i'}f^{\sm}_{n,*}\bz)_{0_{n}}), \bz\big)
\oplus \bigoplus_{\substack{I_{\bullet} \text{ partition of}\\ \{1,\cdots,r\}}} 
\bigoplus_{i''=0}^{i-2h_{I_{\bullet}}} 
\bigg\{\bigoplus_{i_{1}+\cdots+i_{l}=i''} \bigoplus_{j_{1}+\cdots+j_{l}=i-2h_{I_{\bullet}}-i''}  \nonumber
\\
&
\Hom\big(\ch^{j_{1}}\big((j_{I_{1}, n})_{!*} \cf^{i_{1}}_{I_{1},n}\big)_{0_{n}},\bz\big)\otimes \cdots \otimes \Hom\big(\ch^{j_{l}}\big((j_{I_{l}, n})_{!*} \cf^{i_{l}}_{I_{l},n}\big)_{0_{n}}, \bz\big)(h_{I_{\bullet}})\bigg\}^{\oplus\, |(\bbv_{I\bullet}/n)^{\circ}|}. \label{decomposition 00}
\end{align}

Let $\bbd_{\cb_{n}}$ be the dualizing complex for $\cb_{n}$ with $\bz$-coefficient. As $\cb_{n}$ is smooth of complex dimension $\tau_{\gamma}$, we have $\bbd_{\cb_{n}}=\bz[2\tau_{\gamma}]$.
With Poincar\'e-Verdier duality (cf. \cite{verdier duality}, \cite{ks} Chap. 3), we have 
\begin{align*}
\rhom\big(j_{n,!*}(R^{\,i'}f^{\sm}_{n,*}\bz[\tau_{\gamma}])_{0_{n}}, \bz\big)
&=\rhom\big(i_{0_{n}}^{*}j_{n,!*}(R^{\,i'}f^{\sm}_{n,*}\bz[\tau_{\gamma}]), i_{0_{n}}^{!}\bbd_{\cb_{n}}\big)
\\
&=i_{0_{n}}^{!}\rhom\big(j_{n,!*}(R^{\,i'}f^{\sm}_{n,*}\bz[\tau_{\gamma}]), \bbd_{\cb_{n}}\big)
\\
&=i_{0_{n}}^{!}j_{n,!*}\BigWedge^{i'}\bbf_{n}^{\vee}[\tau_{\gamma}].
\end{align*}
Take the free summand of the cohomology of both sides, we obtain
\begin{equation*}\label{id dual 1}
\Hom\big(\ch^{i}\big(j_{n,!*}(R^{\,i'}f^{\sm}_{n,*}\bz)\big)_{0_{n}}, \bz\big)
=\ch^{2\tau_{\gamma}-i}_{\{0_{n}\}}\big(j_{n,!*}\BigWedge^{i'}\bbf_{n}^{\vee}\big)_{\rm fr}.
\end{equation*}
The second summand of the equation (\ref{decomposition 00}) can be simplified in the same way, and we get 
\begin{align*}
H_{i}\big(n\Lambda^{0}\backslash{\xx}_{\gamma}^{0},\bz\big)_{\rm fr}=&
\bigoplus_{i'=0}^{i} 
\ch^{2\tau_{\gamma}-(i-i')}_{\{0_{n}\}}\big(j_{n,!*}\BigWedge^{i'}\bbf_{n}^{\vee}\big)_{\rm fr}
\oplus \bigoplus_{\substack{I_{\bullet} \text{ partition of}\\ \{1,\cdots,r\}}} 
\bigoplus_{i''=0}^{i-2h_{I_{\bullet}}} 
\bigg\{\bigoplus_{i_{1}+\cdots+i_{l}=i''} \bigoplus_{j_{1}+\cdots+j_{l}=i-2h_{I_{\bullet}}-i''}  \nonumber
\\
&
\ch^{2\tau_{I_{1}}-j_{1}}_{\{0_{n}\}}\big((j_{I_{1}, n})_{!*} \cf^{i_{1}}_{I_{1},n}\big)_{\rm fr}\otimes \cdots \otimes \ch^{2\tau_{I_{l}}-j_{l}}_{\{0_{n}\}}\big((j_{I_{l}, n})_{!*} \cf^{i_{l}}_{I_{l},n}\big)_{\rm fr}(h_{I_{\bullet}})\bigg\}^{\oplus\, |(\bbv_{I\bullet}/n)^{\circ}|}. 
\end{align*}
By proposition \ref{coh as limit} and remark \ref{no zl}, we can take projective limit of the free summand of the right hand side of the above equation to get the free summand of $H_{*}(\xx_{\gamma}^{0}, \bz)$, whence the theorem.

\end{proof}

As a corollary of proposition \ref{main term n}, we have a geometric interpretation of the main term in the above theorem. 

\begin{prop}\label{main term 0}

For all $i\in \bn_{0}$, we have
$$
H_{i}\big({\xx}_{\gamma}^{0},\bz\big)_{\rm fr}^{\Lambda^{0}}=\varprojlim_{n}\bigoplus_{i'=0}^{i} 
\ch^{2\tau_{\gamma}-(i-i')}_{\{0_{n}\}}\big(j_{n,!*}\BigWedge^{i'}\bbf_{n}^{\vee}\big)_{\rm fr}.
$$

\end{prop}

\begin{rem}\label{reduce to invariant}

With theorem \ref{reduction theorem}, similar interpretation holds for the remaining terms in the theorem. This leads to the quite surprising result that we can reduce the study of the homology  $H_{*}(\xx_{\gamma}^{0}, \bq)$ to its $\Lambda^{0}$-invariant subspace. This provides yet another reduction of $H_{*}(\xx_{\gamma}^{0}, \bq)$, in addition to the Arthur-Kottwitz reduction and the Harder-Narasimhan reduction \cite{chen fundamental}.

\end{rem}

The formula in proposition \ref{decomposition 0} can be further simplified. 
To begin with, we examine the local monodromies at $0$ of the sheaf $\cf\cong R^{1}\pi_{*}^{\rm sm}\bq$. 
We have fixed a geometric generic point $\bar{\xi}_{0}$ of $\cb_{\{0\}}$, let $\rff=\cf_{\bar{\xi}_{0}}$, consider the local monodromy action 
$$
\rho_{0}: \pi_{1}(\cb_{\{0\}}-\Delta,\,\bar{\xi}_{0})\to \aut(\rff).
$$
The cup product on the cohomology of curves induces a non-degenerate alternating pairing
\begin{equation}\label{pair f}
\pair{\, ,}: \cf\times \cf\to \bq(-1).
\end{equation}
It induces a non-degenerate alternating bilinear form
\begin{equation}\label{cup prod on fiber}
\pair{\, ,}: \rff\times \rff\to \bq(-1).
\end{equation}
Let $B_{0}(\ep)$ an open ball of sufficiently small radius centered at $0$ in $\cb$, such that the restriction of $\pi$ to it is a representative of the miniversal deformation of $\spf(\co[\gamma])$. 
Let $E$ be the subspace of $\rff$ generated by the vanishing cycles at the smooth points of $\Delta\cap  B_{0}(\ep)$, let $E^{\perp}$ be its orthogonal complement with respect to the above pairing. 
Let $L$ be a generic line sufficiently close to $0$ in $B_{0}(\ep)$ and let $\{s_{i}\}_{i\in I}$ be the intersection $L\cap \Delta$. Let $\delta_{i}\in E$ be the vanishing cycle at ${s_{i}}$, and let $\sigma_{i}$ be the associated Picard-Lefschetz transformation. 
Then $\{\delta_{i}\}_{i\in I}$ forms a basis of $E$, and $\pi_{1}(\cb_{\{0\}}-\Delta,\,\bar{\xi}_{0}) 
$ is generated by $\{\sigma_{i}\}_{i\in I}$, which implies that
\begin{equation}\label{e perp c}
E^{\perp}=(\rff)^{\pi_{1}(\cb_{\{0\}}-\Delta,\,\bar{\xi}_{0})}=H^{1}(C_{\gamma},\bq),
\end{equation}
where the second equality is due to lemma \ref{support curve}. 
Moreover, we have $E^{\perp}\subset E$ because
$$
E^{\perp}/(E^{\perp}\cap E)=H^{1}(\widetilde{C}_{\gamma},\bq)=0, 
$$
where $\widetilde{C}_{\gamma}$ is the normalization of $C_{\gamma}$. 
It is clear that $\pi_{1}(\cb_{\{0\}}-\Delta,\,\bar{\xi}_{0})$ preserves both subspaces $E$ and $E^{\perp}$, and we obtain the filtration of $\pi_{1}(\cb_{\{0\}}-\Delta,\,\bar{\xi}_{0})$-modules
\begin{equation}\label{local galois module filt}
0\subset E^{\perp}\subset E\subset \rff.
\end{equation}
By the Picard-Lefschetz formula, both gradings $E^{\perp}$ and $\rff/E$ are trivial $\pi_{1}(\cb_{\{0\}}-\Delta,\,\bar{\xi}_{0})$-modules.
By theorem \ref{monodromy}, $E/E^{\perp}$ is irreducible $\pi_{1}(\cb_{\{0\}}-\Delta,\,\bar{\xi}_{0})$-module. 
With the pairing (\ref{cup prod on fiber}), we can identify canonically 
$$
\rff/E^{\perp}\cong \Hom(E, \bq(-1))=E^{*}(-1). 
$$
Let $\ce$ and $\ce^{\perp}$ be the sub sheaf of $\cf|_{B_{0}(\ep)}$ corresponding to the $\pi_{1}(B_{0}(\ep)-\Delta, \bar{\xi}_{0})$-module $E$ and $E^{\perp}$ respectively. 
Let $\cf^{\vee}=\ch om(\cf,\bq)$ and let $(\cf/\ce^{\perp})^{\vee}=\ch om(\cf/\ce^{\perp}, \bq)$.

\begin{prop}\label{reduce limit 1}
For all $i, i'\in \bn_{0}$, we have
$$
\Big(\varprojlim_{n}
\ch^{i}_{\{0_{n}\}}\big(j_{n,!*}\BigWedge^{i'}\bbf_{n}^{\vee}\big)_{\rm fr}\Big)\otimes_{\bz}\bq
={\rm Im}\Big\{\ch^{i}_{\{0\}}\big(j_{!*}\BigWedge^{i'}(\cf/\ce^{\perp})^{\vee}\big)\to  \ch^{i}_{\{0\}}\big(j_{!*}\BigWedge^{i'}\cf^{\vee}\big)\Big\}.
$$

\end{prop}

\begin{proof}

Let $U_{n}$ be the connected component of $B_{0}(\ep)\times_{\cb}\cb_{n}$ containing $0_{n}$. As $\cb_{n}\to \cb$ is \'etale and $B_{0}(\ep)$ is sufficiently small, we can identify $B_{0}(\ep)$ and $U_{n}$. In this way, all the sheaves introduced above can be seen as sheaves on $B_{0}(\ep)$. 
Note that the restriction of $\Phi_{n}: \overline{\cp}_{n}\to \overline{\cp}\times_{\cb}\cb_{n}$ to $\cb_{n}^{\sm}$ is an isogeny of abelian schemes. Hence the natural morphism $\bbf_{n}^{\vee}\to \bbf_{1}^{\vee}$ is just an inclusion of sheaves.
This holds also for the morphism $\bbf_{n'}^{\vee}\to \bbf_{n}^{\vee}$ with $n|n'$. 
In particular, we have
\begin{equation}\label{id dual 2}
\varprojlim_{n} \BigWedge^{i'}\bbf_{n}^{\vee}=\bigcap_{n} \big(\BigWedge^{i'}\bbf_{n}^{\vee}\big)=\BigWedge^{i'} \big( \textstyle{\bigcap}_{n} \bbf_{n}^{\vee}\big).
\end{equation}

Note that $\bbf_{n}^{\vee}\otimes \bq=\varpi_{n}^{*}\cf^{\vee}$, hence we have inclusion
\begin{equation}\label{first inclusion}
\ch^{i}_{\{0_{n}\}}\big(j_{n,!*}\BigWedge^{i'}\bbf_{n}^{\vee}\big)_{\rm fr}\subset \ch^{i}_{\{0\}}\big(j_{!*}\BigWedge^{i'}\cf^{\vee}\big).
\end{equation}
Let $c=(c_{n})_{n}$ be an element of $\varprojlim_{n}
\ch^{i}_{\{0_{n}\}}\big(j_{n,!*}\BigWedge^{i'}\bbf_{n}^{\vee}\big)_{\rm fr}$. 
By the chain construction of the intermediate extension of Goresky and MacPherson \cite{gm1}, $c_{n}$ is represented by a singular chain with coefficients in $\BigWedge^{i'}\bbf_{n}^{\vee}$ satisfying certain intersection properties. The compatibility between $c_{n}$ then implies that the coefficients can be taken in $\bigcap_{n}\big(\BigWedge^{i'}\bbf_{n}^{\vee}\big)$. Hence $c=(c_{n})_{n}$ actually lies in the image of
$$
\ch^{i}_{\{0\}}\big(j_{!*}\big(\textstyle{\bigcap}_{n}\BigWedge^{i'}\bbf_{n}^{\vee}\big)\big)_{\rm fr}
\to 
\varprojlim_{n}\ch^{i}_{\{0_{n}\}}\big(j_{n,!*}\BigWedge^{i'}\bbf_{n}^{\vee}\big)_{\rm fr}.
$$
With the inclusion (\ref{first inclusion}), this implies that
\begin{equation}\label{main term 000}
\varprojlim_{n}
\ch^{i}_{\{0_{n}\}}\big(j_{n,!*}\BigWedge^{i'}\bbf_{n}^{\vee}\big)_{\rm fr}
={\rm Im}\Big\{\ch^{i}_{\{0\}}\big(j_{!*}\big(\textstyle{\bigcap}_{n}\BigWedge^{i'}\bbf_{n}^{\vee}\big)\big)\to  \ch^{i}_{\{0\}}\big(j_{!*}\BigWedge^{i'}\cf^{\vee}\big)\Big\}.
\end{equation}

To make $\varprojlim_{n}\bbf_{n}^{\vee}=\textstyle{\bigcap}_{n} \bbf_{n}^{\vee}$ explicit, by proposition \ref{killed by finite covering curve} and its proof, we have the exact sequence
\begin{equation}\label{picard torsion n}
1\to \ct_{n}\to \big(\pic_{\cc\times_{\cb}\cb_{n}/\cb_{n}}\big)_{\rm tor} \to \big(\pic_{\cc_{n}/\cb_{n}}\big)_{\rm tor}\to 1, 
\end{equation}
or equivalently
$$
0\to \crr_{n}\to \varpi_{n}^{*}R^{1}\pi_{*}(\bq/\bz)\to R^{1}\pi_{n,*}(\bq/\bz)\to 0. 
$$
Applying $\Hom(\bullet,\bq/\bz)$, we get the exact sequence
$$
0\to \bbf_{n}^{\vee}\otimes \widehat{\bz} \to \varpi_{n}^{*}\bbf_{1}^{\vee}\otimes \widehat{\bz} \to \Hom(\crr_{n}, \bz/n)\to 0,
$$
and so 
\begin{equation}\label{fn}
0\to \bbf_{n}^{\vee} \to \varpi_{n}^{*}\bbf_{1}^{\vee} \to \Hom(\crr_{n}, \bz/n)\to 0.
\end{equation}
The morphism $\varpi_{n}^{*}\bbf_{1}^{\vee} \to \Hom(\crr_{n}, \bz/n)$ in the above exact sequence is induced from the morphism
$$
\varpi_{n}^{*}\bbf_{1}^{\vee} \to \Hom(\varpi_{n}^{*}R^{1}\pi^{\sm}_{*}\bz/n, \bz/n),
$$
which is itself induced from the cap product between homology and cohomlogy. With the exact sequence (\ref{fn}) we can interpret $\bbf_{n}^{\vee}$ as the subsheaf of $\varpi_{n}^{*}\bbf_{1}^{\vee}$ which is killed under the cap product with $\crr_{n}\subset \varpi_{n}^{*}R^{1}\pi^{\sm}_{*}\bz/n$. 
Note that $\crr_{n}$ is the reduction modulo $n$ of an integral structure on $\ce^{\perp}$. This implies that $\big(\varprojlim_{n} \bbf_{n}^{\vee}\big)\otimes_{\bz}\bq$ is the subsheaf of $\bbf_{1}^{\vee}\otimes \bq$ which is killed under cap product with $\ce^{\perp}$,
hence
\begin{equation*}
\big(\varprojlim_{n} \bbf_{n}^{\vee}\big)\otimes_{\bz}\bq=\ch om(\cf/\ce^{\perp}, \bq).
\end{equation*}
The proposition then follows from the isomorphism and the identification (\ref{main term 000}).

\end{proof}

For the family $\pi_{I}:\cc_{I}\to \cb_{I}$, we have local systems $\ce_{I}, \ce_{I}^{\perp}\subset \cf_{I}$ and their dual $\ce_{I}^{\vee}, (\cf_{I}/\ce_{I}^{\perp})^{\vee}\subset \cf_{I}^{\vee}$ as before. 
The remaining terms of proposition \ref{decomposition 0} can be simplified in the same way as above, and we can reformulate it and proposition \ref{main term 0} as:

\begin{thm}\label{decomposition}

For any $i\in \bn_{0}$, the homology $H_{i}\big({\xx}_{\gamma}^{0},\bq\big)$ equals 
{\small
\begin{align*}
&
\bigoplus_{i'=0}^{i} 
{\rm Im}\Big\{\ch^{2\tau_{\gamma}-(i-i')}_{\{0\}}\big(j_{!*}\BigWedge^{i'}(\cf/\ce^{\perp})^{\vee}\big)\to  \ch^{2\tau_{\gamma}-(i-i')}_{\{0\}}\big(j_{!*}\BigWedge^{i'}\cf^{\vee}\big)\Big\}
\oplus \bigoplus_{\substack{I_{\bullet} \text{ partition of}\\ \{1,\cdots,r\}}} 
\bigoplus_{i''=0}^{i-2h_{I_{\bullet}}}
\bigg\{\bigoplus_{i_{1}+\cdots+i_{l}=i''}   \nonumber
\\
&
\bigoplus_{j_{1}+\cdots+j_{l}=i-2h_{I_{\bullet}}-i''} \bigotimes_{k=1}^{l}
{\rm Im}\Big\{\ch^{2\tau_{I_{k}}-j_{k}}_{\{0\}}\big((j_{I_{k}})_{!*}\BigWedge^{i_{k}}(\cf_{I_{k}}/\ce_{I_{k}}^{\perp})^{\vee}\big)\to  \ch^{2\tau_{I_{k}}-j_{k}}_{\{0\}}\big((j_{I_{k}})_{!*}\BigWedge^{i_{k}}\cf_{I_{k}}^{\vee}\big)\Big\}(h_{I_{\bullet}})\bigg\}^{\oplus\, \bbv^{\circ}_{I_{\bullet}}},
\end{align*}
}

\noindent where $l$ is the length of $I_{\bullet}$. Moreover, the main term calculates $H_{i}\big({\xx}_{\gamma}^{0},\bq\big)^{\Lambda^{0}}$. 
\end{thm}

With Poincar\'e-Verdier duality, the theorem can be reformulated as:

\begin{thm}\label{decomposition co}

For all $i\in \bn_{0}$, the cohomology $H^{i}\big({\xx}_{\gamma}^{0},\bq\big)$ equals
{\small
\begin{align*}
&\bigoplus_{i'=0}^{i} 
{\rm Im}\Big\{\ch^{i-i'}\big(j_{!*}\BigWedge^{i'}\cf\big)_{0}\to  
\ch^{i-i'}\big(j_{!*}\BigWedge^{i'}(\cf/\ce^{\perp})\big)_{0} \Big\}
\oplus \bigoplus_{\substack{I_{\bullet} \text{ partition of}\\ \{1,\cdots,r\}}} 
\bigoplus_{i''=0}^{i-2h_{I_{\bullet}}} \bigg\{\bigoplus_{i_{1}+\cdots+i_{l}=i''} 
\\
&
\bigoplus_{j_{1}+\cdots+j_{l}=i-2h_{I_{\bullet}}-i''}  
\bigotimes_{k=1}^{l}
{\rm Im}\Big\{\ch^{j_{k}}\big((j_{I_{k}})_{!*}\BigWedge^{i_{k}}\cf_{I_{k}}\big)_{\{0\}}
\to
\ch^{j_{k}}\big((j_{I_{k}})_{!*}\BigWedge^{i_{k}}(\cf_{I_{k}}/\ce_{I_{k}}^{\perp})\big)_{\{0\}} \Big\}(-h_{I_{\bullet}})\bigg\}^{\oplus\, \bbv^{\circ}_{I_{\bullet}}},
\end{align*}
}

\noindent where $l$ is the length of $I_{\bullet}$. Moreover, the main term calculates $H^{i}\big({\xx}_{\gamma}^{0},\bq\big)^{\Lambda^{0}}$. 

\end{thm}

As a corollary, we can reformulate the purity hypothesis of Goresky, Kottwitz and MacPherson as:

\begin{cor}\label{reduce purity}

The homology of the affine Springer fiber is pure in the sense of Grothendieck-Deligne if and only if the group
$$
\bigoplus_{i'=0}^{i}{\rm Im}\Big\{\ch^{2\tau_{\gamma}-(i-i')}_{\{0\}}\big(j_{!*}\BigWedge^{i'}(\cf/\ce^{\perp})^{\vee}\big)\to  \ch^{2\tau_{\gamma}-(i-i')}_{\{0\}}\big(j_{!*}\BigWedge^{i'}\cf^{\vee}\big)\Big\}
$$
is pure of weight $-i$ for any $i\in \bn_{0}$.
\end{cor}


\begin{rem}

We prefer the reformulation in terms of homology because the target of the morphism $j_{!*}\BigWedge^{i'}\cf^{\vee}$ is defined on $\cb$ while for cohomology the target $j_{!*}\BigWedge^{i'}(\cf/\ce^{\perp})$ is only locally defined. Note that the target of the morphism is actually $H_{i}(\overline{P}_{C_{\gamma}}, \bq)$ by the support theorem of Ng\^o, hence the formula can be regarded as a conjectural intrinsic characterization of the pure part of $H_{i}(\overline{P}_{C_{\gamma}}, \bq)$.

\end{rem}

\bigskip
\small
\noindent
\begin{tabular}{ll}
&Zongbin {\sc Chen} \\ 
\\
&School of mathematics, Shandong University\\
&JiNan, 250100,\\
&Shandong, China \\
&email: {\tt zongbin.chen@email.sdu.edu.cn}

\end{tabular}

\end{document}